\documentclass[]{amsart}

\pdfoutput=1

\usepackage{amssymb}
\usepackage{amsfonts}
\usepackage{amsmath}
\usepackage{amsthm}
\usepackage{mathtools}
\usepackage{graphicx}
\usepackage{mathrsfs}
\usepackage{dsfont}
\usepackage{amscd}
\usepackage{multirow}

\usepackage[all]{xy}
\normalfont
\usepackage[T1]{fontenc}
\usepackage{calligra}
\usepackage{verbatim}
\usepackage{fourier}
\usepackage[usenames,dvipsnames]{xcolor}
\usepackage[colorlinks=true,linkcolor=Blue,citecolor=Violet]{hyperref}
\usepackage{enumerate}
\usepackage{slashed}
\usepackage{nicematrix}
\usepackage{microtype}
\usepackage{tikz}
\usetikzlibrary{snakes,math}
\usepackage{caption}
\usepackage{subcaption}
\usepackage{faktor}

\newcommand{\N}{{\mathds{N}}}
\newcommand{\Z}{{\mathds{Z}}}

\newcommand{\R}{{\mathds{R}}}
\newcommand{\C}{{\mathds{C}}}
\newcommand{\T}{{\mathds{T}}}

\newcommand{\A}{{\mathfrak{A}}}
\newcommand{\B}{{\mathfrak{B}}}

\newcommand{\Lip}[1][L]{{\mathsf{#1}}}
\newcommand{\TLip}{{\mathsf{T}}}
\newcommand{\SLip}{{\mathsf{S}}}

\newcommand{\Hilbert}[1][H]{{\mathscr{#1}}}

\newcommand{\dpropinquity}[1]{{\mathsf{\Lambda}^\ast_{#1}}}

\newcommand{\spectralpropinquity}[1]{{\mathsf{\Lambda}^{\mathsf{spec}}_{#1}}}

\newcommand{\Kantorovich}[1]{{\mathsf{mk}_{#1}}}

\newcommand{\Haus}[1]{{\mathsf{Haus}_{{#1}}\,}}

\newcommand{\StateSpace}{{\mathscr{S}}}

\newcommand{\MongeKant}{{Mon\-ge-Kan\-to\-ro\-vich metric}}

\newcommand{\qcms}{quantum compact metric space}

\newcommand{\sa}[1]{{\mathfrak{sa}\left({#1}\right)}}

\newcommand{\inner}[3]{{\left<{#1},{#2}\right>_{#3}}}

\newcommand{\bridge}[1]{#1}
\newcommand{\dom}[1]{{\operatorname*{dom}\left({#1}\right)}}
\newcommand{\domsa}[1]{{\operatorname*{dom_{\mathsf{sa}}}\left({#1}\right)}}

\newcommand{\diam}[2]{{\mathrm{diam}\left({#1},{#2}\right)}}

\newcommand{\norm}[2]{\left\|{#1}\right\|_{#2}}

\newcommand{\bridgereach}[2]{{\varrho\left(\bridge{#1}\middle\vert {#2}\right)}}

\newcommand{\bridgeheight}[1]{{\operatorname{height}\left(\bridge{#1}\right)}}

\newcommand{\bridgelength}[2]{{\lambda\left(\bridge{#1}\middle\vert{#2}\right)}}

\newcommand{\CDN}{{\mathsf{DN}}}
\newcommand{\TDN}{{\mathsf{TN}}}

\newcommand{\worknote}[1]{}
\newcommand{\opnorm}[3]{{\left|\mkern-1.5mu\left|\mkern-1.5mu\left| {#1} \right|\mkern-1.5mu\right|\mkern-1.5mu\right|_{#3}^{#2}}}

\newcommand{\tunnelextent}[1]{{\chi\left({#1}\right)}}

\newcommand{\Aut}[1]{{\mathrm{Aut}\left({#1}\right)}}

\newcommand{\supp}[1]{\mathrm{supp}\left({#1}\right)}
\newcommand{\ModState}[1]{\widehat{\StateSpace}}

\renewcommand{\geq}{\geqslant}
\renewcommand{\leq}{\leqslant}

\newcommand{\Dirac}[1][D]{{\slashed{#1}}}

\theoremstyle{plain}
\newtheorem{theorem}{Theorem}[section]
\newtheorem*{theorem*}{Theorem}

\newtheorem{corollary}[theorem]{Corollary}

\newtheorem{lemma}[theorem]{Lemma}
\newtheorem{proposition}[theorem]{Proposition}

\newtheorem{theorem-definition}[theorem]{Theorem-Definition}

\theoremstyle{definition}
\newtheorem{definition}[theorem]{Definition}
\newtheorem*{definition*}{Definition}

\newtheorem{notation}[theorem]{Notation}

\newtheorem{hypothesis}[theorem]{Hypothesis}

\theoremstyle{remark}

\newtheorem{example}[theorem]{Example}

\newtheorem{remark}[theorem]{Remark}

\numberwithin{equation}{section}

\let\oldtocsection=\tocsection
\let\oldtocsubsection=\tocsubsection
\let\oldtocsubsubsection=\tocsubsubsection

\renewcommand{\tocsection}[2]{\bfseries \hspace{0em}\oldtocsection{#1}{#2}}
\renewcommand{\tocsubsection}[2]{\itshape \hspace{1em}\oldtocsubsection{#1}{#2}}
\renewcommand{\tocsubsubsection}[2]{\hspace{2em}\oldtocsubsubsection{#1}{#2}}

\begin{document}

  \title{Metric Spectral Geometry of $\Z^d$-Crossed Products beyond Equicontinuity}

  \author{Fr\'{e}d\'{e}ric Latr\'{e}moli\`{e}re}
  \email{frederic@math.du.edu}
  \urladdr{http://fredericlatremoliere.com/}
  \address{Department of Mathematics \\ University of Denver \\ Denver CO 80208}

  \date{\today}

  \begin{abstract}
    We construct linearly twisted spectral triples on $C^\ast$-crossed products $\A\rtimes_\alpha\Z^d$, starting from a spectral triple $(\A,\Hilbert,\Dirac)$ and an action $\alpha$ admitting an adapted weight. The weight compensates for the growth of the base Dirac operator under the dynamics while preserving the canonical differential structure in the $\Z^d$-directions. The construction takes place directly on the original crossed product and does not require the action to be equicontinuous. We prove that, whenever the base spectral triple is metric, the resulting twisted spectral triple is $\kappa$-metric for every $\kappa>0$: its associated Monge--Kantorovich metric metrizes the weak* topology on the state space of the crossed product. Our framework applies, in particular, to every action preserving the domain of the Lipschitz seminorm of the base triple. In this case, an exponential weight may be chosen so that the induced metric geometry restricts exactly to the original metric geometry on the base algebra. In the equicontinuous case, the weight may be chosen constant, the twist disappears, and the standard crossed-product spectral triple is recovered. We also establish a general finite-mode criterion for convergence in the spectral propinquity. As an application, we prove joint continuity, with respect to the deformation parameter and suitable fluctuations of the flat metric, for the twisted spectral triples associated with crossed products of quantum $2$-tori by fixed automorphisms in $\operatorname{SL}(2,\Z)$, including Anosov automorphisms.
  \end{abstract}
  \maketitle
  \tableofcontents


  \section{Introduction}

  Noncommutative metric geometry studies noncommutative analogues of metric spaces and, in particular, of Riemannian manifolds, in a setting where the preferred encoding of geometry is provided by spectral triples, a far-reaching abstraction of the role played by Dirac operators in Riemannian geometry \cite{Connes89,Connes}. A spectral triple is metric when its Connes' distance indeed metrizes the weak* topology on the state space \cite{Rieffel98a,Rieffel00,Latremoliere18g}, thus extending the Monge--Kantorovich distance \cite{Kantorovich40,Kantorovich58} to the noncommutative realm. An emerging core tool for the study of the space of metric spectral triples is the spectral propinquity \cite{Latremoliere18g,Latremoliere20a,Latremoliere21a,Latremoliere23a,Latremoliere23b,Latremoliere24a,Latremoliere24b,Latremoliere26a,Latremoliere26b}, which metrizes that space up to unitary equivalence, and for which many spectral data become continuous \cite{Latremoliere22}. Our purpose in this paper is to bring $C^\ast$-crossed products \cite{Zeller-Meier68}, seen as noncommutative analogues of orbit spaces, within this framework of metric spectral geometry and the spectral propinquity.

  $C^\ast$-dynamical systems constitute an important source of $C^\ast$-algebras through the crossed-product construction. Thus, if an action is defined on the algebra of a spectral triple, it is natural to ask whether the spectral geometry of the base algebra can be extended to the associated crossed product. In this work, we answer this question in the affirmative for a broad and natural class of actions of $\Z^d$ by constructing directly on the associated crossed product a linearly twisted spectral triple, in the sense of \cite{Latremoliere26b}, and by proving that our construction yields a metric twisted spectral triple whenever the base spectral triple is metric. Our framework includes all actions of $\Z^d$ which preserve the domain of the Lipschitz seminorm associated with the base spectral triple. Our construction thus substantially extends the usual construction for equicontinuous actions of $\Z^d$ \cite{Bellissard10,Hawkins13,Paterson14}, which it recovers when the twist disappears. Moreover, we prove convergence for the spectral propinquity for the Anosov actions of $\Z$ on quantum tori as the deformation parameter and an appropriate fluctuation of the underlying metric vary.

  \medskip

  Connes introduced in \cite{Connes89} the notion of a spectral triple. Spectral triples are modeled after certain elliptic operators, most notably Dirac operators, and define unbounded $K$-homology cycles. Much of the early success of this theory arose from the extension of index theorems to the noncommutative setting. In this sense, Connes' approach to noncommutative geometry may be regarded as a noncommutative extension of spectral geometry. Our interest is in the additional metric information carried by certain spectral triples and in the geometry of the space formed by these triples. We moreover work with the linearly twisted spectral triples introduced in \cite{Latremoliere26b}, whose definition arose from our study of approximations of quantum tori by physically motivated matrix models.

  \begin{definition}\label{twisted-spectral-def}
    A \emph{twisted spectral triple} $(\A,\Hilbert,\Dirac,\rho)$ is given by a unital $C^\ast$-algebra $\A$ acting on a Hilbert space $\Hilbert$, a self-adjoint operator $\Dirac$ defined on a dense subspace $\dom{\Dirac}\subseteq\Hilbert$, and a linear map
    \begin{equation*}
      \rho:\A_\rho\longrightarrow\B(\Hilbert)
    \end{equation*}
    defined on a dense subspace $\A_\rho$ of $\A$, such that:
    \begin{enumerate}
      \item $\Dirac$ has compact resolvent,
      \item for all $a\in\A_\rho$,
      \begin{equation*}
        a\dom{\Dirac}\subseteq\dom{\Dirac} \qquad\text{and}\qquad [\Dirac,a]_\rho\coloneqq\Dirac a-\rho(a)\Dirac
      \end{equation*}
      extends to a bounded operator on $\Hilbert$.
    \end{enumerate}
  \end{definition}

  \begin{definition}
    A \emph{spectral triple} $(\A,\Hilbert,\Dirac)$ is a twisted spectral triple whose twist is the identity map on its domain.
  \end{definition}

  The fact that our twists are linear, rather than necessarily multiplicative, is an essential feature of this framework. Linearly twisted spectral triples first arose in \cite{Latremoliere26b} from the problem of approximating the flat spectral geometry of quantum tori by fuzzy tori. They arise again, independently, in the present work: the correction forced by the renormalization of the base Dirac operator is natural and explicit, but is generally not multiplicative. Thus, the same enlargement of the notion of a spectral triple resolves two distinct obstructions, one originating from finite-dimensional approximation and the other from non-equicontinuous dynamics. This additional flexibility does not come at the expense of metric control. The displacement term involving
  \begin{equation*}
    \norm{a-\rho(a)}{\B(\Hilbert)}
  \end{equation*}
  provides the estimates needed for the module action, while the modular and inner Leibniz inequalities required by the spectral propinquity remain available. Consequently, neither the metric theory nor the convergence theory requires the twist to be an automorphism. Requiring multiplicativity would exclude the natural twists arising in these constructions without providing an additional advantage for the metric questions considered here. In the present context, the principal consequence of requiring the twist to be multiplicative is that twisted commutators satisfy a twisted Leibniz identity, and hence that their natural domain is an algebra. Neither this twisted derivation property nor the resulting algebraic closure is required for the coincidence property of the spectral propinquity (namely, distance zero implies unitary equivalence).

  The first matter we address in this work is the construction of a (possibly twisted) spectral triple on the C*-crossed product
  \begin{equation*}
    \A\rtimes_\alpha\Z^d
  \end{equation*}
  induced by a spectral triple $(\A,\Hilbert,\Dirac)$ and an action $\alpha$ of $\Z^d$ on $\A$, in the broadest generality we can. Our construction immediately involves the analogue of the Lipscthiz seminorm induced by a spectral triple.

  \begin{notation}
    If $\Hilbert$ is a Hilbert space, the $C^\ast$-algebra of its bounded linear endomorphisms is denoted by $\B(\Hilbert)$.
  \end{notation}

  \begin{notation}
    If $E$ is a normed vector space, then its norm is denoted by $\norm{\cdot}{E}$ unless specified otherwise. If $\Hilbert$ is a Hilbert space, its inner product is denoted by $\inner{\cdot}{\cdot}{\Hilbert}$.
  \end{notation}

  Let
  \begin{equation*}
    \dom{\Lip} \coloneqq \left\{ a \in \A : a\dom{\Dirac}\subseteq\dom{\Dirac}\text, [\Dirac,a]\text{ bounded } \right\}
  \end{equation*}
  and, for all $a \in \dom{\Lip}$, we set
  \begin{equation*}
    \Lip(a)\coloneqq\norm{[\Dirac,a]}{\B(\Hilbert)}
  \end{equation*}
  The obstruction to the usual crossed-product construction is the possible growth of
  \begin{equation*}
    z \in \Z^d \mapsto \Lip(\alpha^z(a))
  \end{equation*}
  We compensate for this growth by introducing, whenever possible, a function
  \begin{equation*}
    w:\Z^d\longrightarrow(0,1]
  \end{equation*}
  such that $w(0)=1$ and, for all $k\in\Z^d$,
  \begin{equation*}
    \sup_{z\in\Z^d}\frac{w(z+k)}{w(z)}<\infty\text,
  \end{equation*}
  and such that the space
  \begin{equation*}
    \A_w\coloneqq\left\{a\in\A:\sup_{z\in\Z^d}w(z)\Lip(\alpha^{-z}(a))<\infty\right\}
  \end{equation*}
  is dense in $\A$. Such a function is called an adapted weight. Such a weight defines a new seminorm
  \begin{equation*}
    \Lip_w(a)\coloneqq\sup_{z\in\Z^d}w(z)\Lip(\alpha^{-z}(a))
  \end{equation*}
  on $\A_w$, which is essentially a renormalized Lipschitz seminorm for our action. The existence of such a weight is our only assumption at this point.

  We then construct a twisted spectral triple on the canonical representation of the C*-crossed-product on the Hilbert space:
  \begin{equation*}
    \Hilbert[S]\coloneqq\ell^2(\Z^d,\Hilbert)\text,
  \end{equation*}
  where $\Z^d$ acts by the left regular representation tensored with the identity on $\Hilbert$, while, for all $\xi = (\xi_z)_{z\in \Z^d}$, for all $a\in \A$, and for all $(z_1,\ldots,z_d) \in \Z^d$,
  \begin{equation*}
    a(\xi_z)_{z\in\Z} = (\alpha^{-z}(a)\xi_z)_{z\in\Z^d}\text, \qquad (\Dirac_w\xi)_z\coloneqq w(z)\Dirac\xi_z \qquad\text{and}\qquad (\partial_j\xi)_z\coloneqq z_j\xi_z \quad(j\in\{1,\ldots,d\})
  \end{equation*}

  The Dirac operator of our twisted spectral triple is then given on $\Hilbert[S]\otimes\C^q$ by:
  \begin{equation*}
    \nabla_{w,\varepsilon} \coloneqq \Dirac_w\otimes\gamma_0 + \sum_{j=1}^d\partial_j\otimes\gamma_j + \varepsilon 1_{\Hilbert[S]}\otimes\gamma_{d+1}\text.
  \end{equation*}
  for some anticommuting, self-adjoint unitaries $\gamma_0,\ldots,\gamma_{d+1}$ on $\C^q$ for any appropriate $q$, and for any $\varepsilon > 0$.

  The purpose of this construction is to keep the calculus on $C^\ast(\Z^d) \subseteq \A\rtimes_\alpha\Z^d$ the standard one, but to renormalize the Dirac operator from the base to induce the renormalized Lipschitz seminorm on the base. The cost of this renormalization is that we must introduce a twist to be able to define a Lipscthiz seminorm on the C*-crossed product. Thus, a key result of this paper is Theorem \ref{twisted-spectral-triple-thm} which establishes the existence of a twist $\rho_\varepsilon$ such that
  \begin{equation*}
    \left( \A\rtimes_\alpha\Z^d, \Hilbert[S]\otimes\C^q, \Dirac[\nabla]_{w,\varepsilon}, \rho_\varepsilon \right)
  \end{equation*}
  is a twisted spectral triple, per Definition \ref{twisted-spectral-def}. Moreover, with our choice of twist,
  \begin{equation*}
    [\nabla_{w,\varepsilon},\pi(a)U^k]_{\rho_\varepsilon} = [\Dirac_w,\pi(a)]U^k\otimes\gamma_0 + \sum_{j=1}^d k_j\pi(a)U^k\otimes\gamma_j
  \end{equation*}
  for all $a\in\A_w$ and $k=(k_1,\ldots,k_d)\in\Z^d$.

  This first result already extends the classical crossed-product construction beyond the equicontinuous setting: it produces spectral data directly on the original crossed product without requiring the action to be equicontinuous. If the action is equicontinuous, we may choose $w=1$, in which case the twist is simply the identity and we recover the usual crossed-product spectral triple, up to the Clifford representation (and the added $\varepsilon$ added to our Dirac operator).

  The existence of a twisted spectral triple is only the first part of our construction, however. We wish our twisted spectral triples to be points in the space of the spectral propinquity, so that discussions of their stability, continuity of various families of such triples, and more, can be discussed. A twisted spectral triple is metric when it satisfies the following additional property: the distance induced on the state space of its C*-algebra metrizes the weak* topology. This notion was designed in \cite{Latremoliere26b} to fit our construction of the spectral propinquity, originally developed in \cite{Latremoliere18g} for ordinary metric spectral triples. In fact, we realized in \cite{Latremoliere26b} the importance of introducing twists, which were essential to obtain the convergence of fuzzy tori to quantum tori for their usual, ``flat'', spectral triples. This was the root of the current work.

  \begin{notation}
    If $\A$ is a $C^\ast$-algebra, then the state space of $\A$ is denoted by $\StateSpace(\A)$, and its self-adjoint $\R$-subspace is
    \begin{equation*}
      \sa{\A}\coloneqq\left\{a\in\A:a=a^\ast\right\}\text.
    \end{equation*}
  \end{notation}

  \begin{definition}\label{mongekant-def}
    If $\Lip$ is a seminorm defined on a subspace $\dom{\Lip}$ of a unital $C^\ast$-algebra $\A$, then the {\MongeKant} $\Kantorovich{\Lip}$ is the extended pseudo-metric defined on the state space $\StateSpace(\A)$ of $\A$ by setting, for all $\varphi,\psi\in\StateSpace(\A)$,
    \begin{equation*}
      \Kantorovich{\Lip}(\varphi,\psi) \coloneqq \sup\left\{ \left|\varphi(a)-\psi(a)\right|: a\in\domsa{\Lip},\ \Lip(a)\leq1 \right\}\text,
    \end{equation*}
    where
    \begin{equation*}
      \domsa{\Lip}\coloneqq\dom{\Lip}\cap\sa{\A}\text.
    \end{equation*}
  \end{definition}

  \begin{definition}\label{metric-twisted-spectral-def}
    A \emph{$\kappa$-metric twisted spectral triple} $(\A,\Hilbert,\Dirac,\rho)$, with $\kappa>0$, is a twisted spectral triple such that:
    \begin{enumerate}
      \item if we define
      \begin{equation*}
        \domsa{\Lip}\coloneqq\left\{a\in\A_\rho:a=a^\ast\right\}
      \end{equation*}
      and, for all $a\in\domsa{\Lip}$,
      \begin{equation*}
        \Lip(a)\coloneqq \max\left\{ \norm{[\Dirac,a]_\rho}{\B(\Hilbert)}, \kappa\norm{a-\rho(a)}{\B(\Hilbert)} \right\}\text,
      \end{equation*}
      then $\domsa{\Lip}$ is dense in $\sa{\A}$, and $\Kantorovich{\Lip}$ metrizes the weak* topology on the state space $\StateSpace(\A)$ of $\A$,
      \item if $(a_j)_{j\in J}$ is a net in $\A_\rho$, $a\in\A$, and $T\in\B(\Hilbert)$, and if
      \begin{equation*}
        a_j\xrightarrow{j\in J}a\text{ in norm,} \qquad \rho(a_j)\xrightarrow{j\in J}T\text{ in WOT,} \qquad\text{and}\qquad \sup_{j\in J}\norm{[\Dirac,a_j]_\rho}{\B(\Hilbert)}<\infty\text,
      \end{equation*}
      then
      \begin{equation*}
        a\in\A_\rho \qquad\text{and}\qquad \rho(a)=T\text;
      \end{equation*}
      \item $\rho(1)=1$.
    \end{enumerate}
  \end{definition}

  \begin{definition}
    A \emph{metric spectral triple} is a $\kappa$-metric twisted spectral triple whose twist is the identity map on its domain. This notion does not depend on the choice of $\kappa>0$.
  \end{definition}

  Our second main result, Theorem \ref{metric-thm}, is that if, in addition to the above, $(\A,\Hilbert,\Dirac)$ is a metric spectral triple, then, for all $\varepsilon>0$ and $\kappa>0$, our twisted spectral triple constructed above is metric.
  \begin{equation*}
    \left( \A\rtimes_\alpha\Z^d, \Hilbert[S]\otimes\C^q, \nabla_{w,\varepsilon}, \rho_\varepsilon \right)
  \end{equation*}
  is a $\kappa$-metric twisted spectral triple. In particular, the seminorm
  \begin{equation*}
    \Lip_{w,\varepsilon,\kappa}(b) \coloneqq \max\left\{ \norm{[\nabla_{w,\varepsilon},b]_{\rho_\varepsilon}}{\B(\Hilbert[S]\otimes\C^q)}, \kappa\norm{b-\rho_\varepsilon(b)}{\B(\Hilbert[S]\otimes\C^q)} \right\}
  \end{equation*}
  induces a Monge--Kantorovich metric which metrizes the weak* topology on
  \begin{equation*}
    \StateSpace(\A\rtimes_\alpha\Z^d)\text.
  \end{equation*}

  The restriction of this seminorm to the base algebra is
  \begin{equation*}
    \Lip_{w,\varepsilon,\kappa}(\pi(a))=\Lip_w(a)
  \end{equation*}
  for all $a\in\A_w$. Thus, the construction retains explicit control of the original metric geometry. In particular, if
  \begin{equation*}
    \alpha^z(\dom{\Lip})\subseteq\dom{\Lip}
  \end{equation*}
  for all $z\in\Z^d$, then there exists $C\geq1$ such that
  \begin{equation*}
    w(z)=C^{-|z|}
  \end{equation*}
  is admissible and adapted to $\alpha$, while
  \begin{equation*}
    \A_w=\dom{\Lip} \qquad\text{and}\qquad \Lip_w=\Lip\text.
  \end{equation*}
  Consequently, whenever the action preserves the Lipschitz domain, our construction not only produces a metric twisted spectral triple on the crossed product, but also recovers exactly the initial metric geometry on the base algebra.

  The proof that our twisted spectral triples are metric uses regularization under the dual action. We prove that Fourier truncations approximate the Lipschitz unit ball uniformly without increasing its Lipschitz seminorm. The Fourier coefficients of these truncations lie in uniformly bounded subsets of the Lipschitz ball of the base algebra, whose compactness follows from the metricity of the base spectral triple.

  We now place this conclusion within the framework of quantum compact metric spaces introduced by Rieffel \cite{Rieffel98a,Rieffel00}.

  \begin{definition}\label{g-qcms-def}
    A \emph{generalized \qcms} $(\A,\Lip)$ is given by a unital $C^\ast$-algebra $\A$ and a seminorm $\Lip$ defined on a dense subspace $\domsa{\Lip}$ of $\sa{\A}$ such that:
    \begin{enumerate}
      \item $\Lip(1)=0$,
      \item $\Kantorovich{\Lip}$ is a metric which induces the weak* topology on the state space $\StateSpace(\A)$ of $\A$,
      \item
      \begin{equation*}
        \left\{a\in\domsa{\Lip}:\Lip(a)\leq1\right\}
      \end{equation*}
      is closed in norm in $\A$.
    \end{enumerate}

    The seminorm $\Lip$ is then referred to as a Lipschitz seminorm.
  \end{definition}

  \begin{definition}\label{qcms-def}
    A {\qcms} $(\A,\Lip)$ is a generalized {\qcms} such that, for all $a,b\in\domsa{\Lip}$, both
    \begin{equation*}
      \Re(ab)=\frac{ab+ba}{2} \qquad\text{and}\qquad \Im(ab)=\frac{ab-ba}{2i}
    \end{equation*}
    belong to $\domsa{\Lip}$, and
    \begin{equation*}
      \max\left\{ \Lip(\Re(ab)), \Lip(\Im(ab)) \right\} \leq \Lip(a)\norm{b}{\A} + \norm{a}{\A}\Lip(b)\text.
    \end{equation*}
  \end{definition}

  By \cite[Theorem 4.9]{Latremoliere26b}, if $(\A,\Hilbert,\Dirac,\rho)$ is a $\kappa$-metric twisted spectral triple, then $(\A,\Lip_{|\sa{\A}})$, with $\Lip$ defined in Definition \ref{metric-twisted-spectral-def}, is a generalized {\qcms}. If $\rho$ is the identity map on its domain, so that $(\A,\Hilbert,\Dirac)$ is a metric spectral triple, then $(\A,\Lip_{|\sa{\A}})$ is a {\qcms}.

  In the remainder of this article, when a Lipschitz seminorm arises from a twisted spectral triple, we use the same notation $\Lip$ for its extension to its natural complex domain and for its restriction to the self-adjoint part of this domain. Replacing $\Lip_{|\sa{\A}}$ with $\Lip$ in Definition \ref{mongekant-def} does not change the associated metric, since the Monge--Kantorovich metric only involves
  \begin{equation*}
    \domsa{\Lip}=\dom{\Lip}\cap\sa{\A}\text.
  \end{equation*}

  Having constructed metric twisted spectral triples on the crossed products, we next study the continuity of this construction for the spectral propinquity. We first establish a general finite-mode convergence criterion. We apply this criterion to quantum $2$-tori. For every fixed
  \begin{equation*}
    A\in\operatorname{SL}(2,\Z),
  \end{equation*}
  including every hyperbolic matrix, let $\alpha_{A,\theta}$ be the associated action of $\Z$ on the quantum torus $\A_\theta$. We construct a family of weighted twisted spectral triples
  \begin{equation*}
    \left( \A_\theta\rtimes_{\alpha_{A,\theta}}\Z, \Hilbert[K], \Dirac[\nabla]_G, \rho_{\theta,G} \right)_{(\theta,G)\in\R\times\mathcal G}
  \end{equation*}
  and prove that this family is jointly continuous for the twisted spectral propinquity. Thus, in particular, the crossed products associated with an Anosov automorphism vary continuously as both the deformation parameter of the quantum torus and the chosen fluctuation of its flat metric vary. This result establishes that our construction is natural and well-behaved in the space of the spectral propinquity: not merely a means to avoid a standard obstuction at all cost, but rather, a natural construction which can then be used fruitfully to estabslish new continuity results for varying spectral triples.

  \medskip

  Our construction answers a delicate question raised for a long time, which motivated other approaches to the noncommutative geometry of C*-dynamical systems. The usual construction of a spectral triple on a crossed product requires the action to be equicontinuous for the Lipschitz seminorm of the base spectral triple \cite{Bellissard10, Hawkins13, Paterson14}. Under this assumption, the natural crossed-product Dirac operator has bounded commutators on a dense subalgebra. Whether this construction gives a quantum metric has subsequently been studied under various assumptions on the acting group choices such as length functions \cite{Klisse23}. Our construction recovers this spectral triple when $w=1$, but its main purpose is to address actions for which equicontinuity fails.

  Bellissard, Marcolli, and Reihani proposed a different method for treating non-equicontinuous actions of $\Z$ \cite{Bellissard10}. Their construction replaces the original dynamical system with a metric bundle on which the induced action becomes isometric. It thus makes the usual equicontinuous construction available, but over a different nonunital algebra. In particular, it does not construct a spectral triple directly on the crossed product of the original system. By contrast, our construction remains on the original unital crossed product and modifies its spectral geometry through the weight and the resulting twist. In turns, this allows us to see our twisted spectral triples as points in the spectral propinquity space.

  Iochum and Masson constructed extensions of spectral triples to reduced crossed products which, in appropriate examples, can be promoted to modular-type twisted spectral triples \cite{Iochum14}. Their twists are modular or automorphic in nature and are particularly suited to affine and conformal actions. The linear twist used here is of a different type: it need not be multiplicative and is determined by the weight used to renormalize the base Dirac operator. Moreover, unlike our work, theirs does not address the quantum metric aspects of their constructions.

  Kaad and Kyed constructed quantum compact metric spaces on crossed products by $\Z$ for classes of actions which need not be equicontinuous \cite{Kaad21}. Their construction uses ordinary and twisted derivations and provides convergence results for Rieffel's quantum Gromov--Hausdorff distance, including applications to families of $C^1$-diffeomorphisms (which we will address in a companion paper). These derivations do not arise as the commutators of a single self-adjoint operator with compact resolvent and therefore do not define (twisted) spectral triples. Consequently, the spectral propinquity does not apply directly to these structures. Moreover, the associated seminorms need not satisfy a Leibniz inequality, so the dual propinquity does not apply directly either.

  More recently, Austad, Kaad, and Kyed constructed quantum metrics on crossed products by groups of polynomial growth \cite{Kaad25}. They work with more general groups of polynomial growth, but for equicontinuous actions, and obtain interesting {\qcms s} and some convergence results in that context. We expect that the work presented here can be extended to these more general groups, removing the equicontinuity restriction.

  The contribution of the present paper is therefore not merely the construction of a quantum metric on a crossed product. We construct, directly on the original crossed product, a twisted spectral triple for a broad class of non-equicontinuous actions; we prove that this triple is metric whenever the base spectral triple is metric; we recover both the usual crossed-product spectral triple and the initial metric on the base whenever they should be recovered; and we prove that the construction is compatible with convergence for the spectral propinquity.

  \medskip

  This work is organized in two parts. In the first part, we develop the twisted spectral triple comstruction, including establishing the existence and fundamental properties of the associated linear twist, and proving that the resulting twisted spectral triple is metric whenever the base spectral triple is metric. We also discuss several classes of actions to which our construction applies. In the second part, we establish a general strategy to prove convergence of metric twisted spectral triples and apply this criterion to the crossed products of quantum tori by fixed automorphisms of $\Z^2$, including Anosov automorphisms. Further applications of the construction and of our convergence criterion will be considered separately.

  \section{Twisted Spectral Triples for C*-crossed-products}

  For this section, we fix the following data. Let $(\A,\Hilbert,\Dirac)$ be a metric spectral triple. Let $\alpha : \Z^d \to \mathrm{Aut}(\A)$ be a group morphism, i.e. an action of $\Z^d$ on the unital C*-algebra $\A$. This section addresses the difficult question: how does one construct a twisted spectral triple on the C*-crossed product $\A\rtimes_\alpha\Z^d$

  Throughout this section, we let
  \begin{equation*}
    \A_{\Dirac} \coloneqq \{ a \in \A : a\dom{\Dirac}\subseteq\dom{\Dirac}\text, [\Dirac,a] \text{ bounded } \}
  \end{equation*}
  and for all $a \in \A_{\Dirac}$:
  \begin{equation*}
    \Lip(a)=\norm{[\Dirac,a]}{\B(\Hilbert)}\text.
  \end{equation*}
  Note that in this section, as is common when working with spectral triples, we define $\Lip$ on the dense subalgebra $\A_{\Dirac}$ of $\A$ rather than just its self-adjoint part (which is dense in $\sa{\A}$ since closed under $*$). Also note that $\Kantorovich{\Lip} = \Kantorovich{\Lip\left|_{\sa{\A}}\right.}$.

  \subsection{Weights and Adjusted metric}

  We introduce numerical weights to normalize the Lipschitz seminorm along orbits for the action $\alpha$.

  \begin{definition}\label{weight-def}
    An \emph{admissible weight} $w : \Z^d \to (0,1]$ is a function such that $w(0)=1$ and, for all $k \in \Z^d$,
    \begin{equation}\label{r-eq}
      r_w(k) \coloneqq \sup_{z \in \Z^d} \frac{w(z+k)}{w(z)} < \infty \text.
    \end{equation}
  \end{definition}

  \begin{definition}\label{adapted-def}
    Let $\Lip$ be a lower semicontinuous, $*$-invariant seminorm defined on a dense $*$-subspace $\dom{\Lip}$ of $\A$ such that $(\A,\Lip_{|\sa{\A}})$ is a generalized {\qcms}. An admissible weight $w$ is \emph{adapted} to an action $\alpha : \Z^d \to \Aut{\A}$ and to $\Lip$ when, with the convention that $\Lip(a)=\infty$ for $a\notin\dom{\Lip}$,
    \begin{equation*}
      \A_w \coloneqq \left\{ a \in \A : \sup_{z \in \Z^d} w(z) \Lip(\alpha^{-z}(a)) < \infty \right\}
    \end{equation*}
    is dense in $\A$.
  \end{definition}

  We thus can renormalize the Lipschitz seminorm of a generalized {\qcms} to obtain a new quantum metric adapted to the action.
  \begin{lemma}\label{Lipw-lemma}
    Let $\Lip$ be a lower semicontinuous, $*$-invariant seminorm defined on a dense $*$-subspace $\dom{\Lip}$ of $\A$ such that $(\A,\Lip_{|\sa{\A}})$ is a generalized {\qcms}. Let $\alpha:\Z^d\to\Aut{\A}$ be a group morphism, and let $w$ be an adapted admissible weight. If, for all $a \in \A$, we set (allowing for $\infty$),
    \begin{equation*}
      \Lip_w(a) \coloneqq \sup\left\{ w(n) \Lip(\alpha^{-n}(a)) : n \in \Z^d \right\} \text,
    \end{equation*}
    then $(\A,(\Lip_w)_{|\sa{\A}})$ is a generalized {\qcms} such that, for all $z\in\Z^d$,
    \begin{equation*}
      \Lip_w(\alpha^{z}(a)) \leq r_w(z) \; \Lip_w(a) \text.
    \end{equation*}
    Moreover, if $\Lip$ is Leibniz, so is $\Lip_w$.
  \end{lemma}

  \begin{proof}
    Since $\Lip \leq \Lip_w$ as $w(0)=1$, and since $\A_w = \dom{\Lip_w}$ is a dense $*$-subspace, the distance $\Kantorovich{\Lip_w}$ is a metric and $\Kantorovich{\Lip_w}\leq\Kantorovich{\Lip}$. The weak* topology is compact and the topology induced by $\Kantorovich{\Lip_w}$ is Hausdorff, so $\Kantorovich{\Lip_w}$ metrizes the weak* topology on $\StateSpace(\A)$. Moreover, for each $z\in\Z^d$, the seminorm $w(z)\Lip\circ\alpha^{-z}$ is lower semi-continuous since $\Lip$ is, and since $\alpha^{-z}$ is continuous. As the supremum of lower semicontinuous functions, $\Lip_w$ is lower semicontinuous. Thus $(\A,(\Lip_w)_{|\sa{\A}})$ is a generalized {\qcms}.

    Let $z \in \Z^d$. By definition of $\Lip_w$, for all $a \in \A_w$,
    \begin{align*}
      \Lip_w(\alpha^{z}(a)) &= \sup_{k \in \Z^d} w(k)\Lip(\alpha^{z}(\alpha^{-k}(a))) \\
      &= \sup_{k \in \Z^d} w(k)\Lip(\alpha^{z-k}(a)) \\
      &= \sup_{m \in \Z^d} w(m+z) \Lip(\alpha^{-m}(a)) \\
      &\leq \sup_{m \in \Z^d} w(m+z) \cdot \frac{1}{w(m)} \Lip_w(a) \\
      &\leq r_w(z) \Lip_w(a) \text.
    \end{align*}

    Assume now that $\Lip$ is Leibniz. Since $w(0)=1$,
    \begin{align*}
      \Lip_w(ab) &= \sup_{m\in\Z^d} w(m) \Lip(\alpha^{-m}(ab)) \\
      &= \sup_{m\in\Z^d} w(m) \Lip(\alpha^{-m}(a)\alpha^{-m}(b)) \\
      &\leq \sup_{m\in\Z^d} w(m) \bigl(\Lip(\alpha^{-m}(a))\norm{\alpha^{-m}(b)}{\A} + \norm{\alpha^{-m}(a)}{\A} \Lip(\alpha^{-m}(b))\bigr) \\
      &\leq \left(\sup_{m\in\Z^d} w(m) \Lip(\alpha^{-m}(a))\norm{b}{\A} + \norm{a}{\A} \sup_{m\in\Z^d} w(m) \Lip(\alpha^{-m}(b))\right) \\
      &\leq \Lip_w(a)\norm{b}{\A} + \norm{a}{\A}\Lip_w(b) \text.
    \end{align*}
    This concludes our proof.
  \end{proof}

  We note that $\Lip_w$, like $\Lip$, is defined on a dense $\C$-subspace of $\A$, rather than a dense $\R$-subspace of $\sa{\A}$. Again, however, this is a convenience for our construction, and when discussing metric properties, one simply work with the restriction $\Lip_w\left|_{\sa{\A}}\right.$ to $\sa{\A}$, which is defined on $\A_w\cap\sa{\A}$; since $\A_w$ is closed under $*$, it is easily checked that $\A_w\cap \sa{\A}$ is dense in $\sa{\A}$.

  In the literature, an action $\alpha$ on a {\qcms} $(\A,\Lip)$ is called equicontinuous when for all $a \in \dom{\Lip}$,
  \begin{equation*}
    \sup_{z \in \Z^d} \Lip(\alpha^z(a)) < \infty \text.
  \end{equation*}
  For an equicontinuous action, the natural weight to consider is simply $w:z\in\Z^d\mapsto 1$. Now, more generally, any action which preserves the domain of the L-seminorm admits an admissible weight, of the exponential type.
  \begin{lemma}\label{exponential-weight-lemma}
    Let $|z| = \sum_{j=1}^d |z_j|$ for all $z=(z_1,\ldots,z_d) \in \Z^d$.

    If $\alpha^z(\dom{\Lip}) \subseteq \dom{\Lip}$ for all $z \in \Z^d$, then there exists $C \geq 1$ such that the weight
    \begin{equation*}
      w : z \in \Z^d \mapsto C^{-|z|}
    \end{equation*}
    is admissible and adapted to $\alpha$; moreover $\dom{\Lip} = \A_w$ and $\Lip_w = \Lip$.
  \end{lemma}

  \begin{proof}
    By \cite{Latremoliere16c}, there exists $C_j > 0$ such that $\Lip\circ\alpha^{e_j} \leq C_j \Lip$ and $C_{-j} > 0$ such that $\Lip\circ\alpha^{-e_j} \leq C_{-j}\Lip$. Let $C \coloneqq \max\{ 1, C_j, C_{-j} : j \in \{1,\ldots,d\} \}$.

    We set $w(z) = C^{-|z|}$. Note that
    \begin{equation*}
      \frac{w(m+z)}{w(m)} = C^{ |m|-|z+m|} \leq C^{|z|} \quad\text{ so }\quad \sup_{m\in\Z^d}\frac{w(m+z)}{w(m)} \leq C^{|z|} < \infty \text,
    \end{equation*}
    so $w$ is an admissible weight.

    We then note that, for all $a \in \dom{\Lip}$,
    \begin{equation*}
      \Lip(\alpha^z(a)) \leq \prod_{j=1}^d C_j^{z_j^+} C_{-j}^{z_j^-} \Lip(a) \leq C^{|z|}\Lip(a) \text,
    \end{equation*}
    where $z_j^+=\max\{z_j,0\}$ and $z_j^-=\max\{-z_j,0\}$. Therefore,
    \begin{equation*}
      C^{-|z|} \Lip(\alpha^{z}(a)) \leq \Lip(a) \text.
    \end{equation*}
    Therefore, for all $z \in \Z^d$ and $a\in \dom{\Lip}$,
    \begin{equation*}
      w(z) \Lip(\alpha^{-z}(a)) \leq \Lip(a) \text{ so }\Lip_w(a) = \sup_{z\in\Z^d} w(z)\Lip(\alpha^{-z}(a)) \leq \Lip(a) \text;
    \end{equation*}
    since $w(0) = 1$, we have $\Lip(a) \leq \Lip_w(a)$ by construction. So $w$ is adapted to $\alpha$ and $\Lip$, and moreover, $\A_w = \dom{\Lip}$ and $\Lip_w=\Lip$.
  \end{proof}

  \begin{remark}
    In this section, we do \emph{not} assume that the action $\alpha$ preserves the domain $\dom{\Lip}$ of the Lipschitz seminorm. We only require that, for some weight $w$, the algebra $\A_w$ is dense. As seen above, if the action preserves the domain of $\Lip$, then it fits our framework for some exponential weight.
  \end{remark}

  \begin{remark}
    For $d=1$, and $\mu,\lambda \geq 1$, Kaad and Kyed \cite{Kaad21} used
    \begin{equation*}
      n \in \Z \mapsto
      \begin{cases}
        \lambda^{-n} \text{ if }n \in \N \text, \\
        \mu^{n} \text{ if }n \in -\N \text,
      \end{cases}
    \end{equation*}
    as what they called a modular weight; it is indeed an admissible weight in our sense. Moreover, as seen in Lemma \ref{exponential-weight-lemma}, if the action of $\Z$ preserves the domain of $\Lip$, then there always exists such a weight adapted to the action.
  \end{remark}

  \subsection{Weighted Dirac Operator}

  The natural, induced covariant representation $(\A,\alpha,\Z^d)$ acts on the Hilbert space:
  \begin{equation*}
    \Hilbert[S] \coloneqq \ell^2(\Z^d,\Hilbert) = \ell^2(\Z^d)\otimes\Hilbert
  \end{equation*}
  as follows. For all $a \in \A$, and for all $\xi\coloneqq (\xi_z)_{z\in\Z^d} \in \Hilbert[S]$, we set
  \begin{equation*}
    \pi(a) \xi = \left( \alpha^{-z}(a) \xi_z \right)_{z \in \Z^d} \text.
  \end{equation*}
  Moreover, for each $z \in \Z^d$, we set:
  \begin{equation*}
    U^z \xi = (\xi_{k-z})_{k \in \Z^d} \text.
  \end{equation*}
  Since
  \begin{equation*}
    U^z \pi(a) U^{-z} = \pi(\alpha^z(a))
  \end{equation*}
  for all $z \in \Z^d$, the pair $(\pi,U)$ is a covariant representation of the C*-dynamical system $(\A,\alpha,\Z^d)$.

  Since $\Z^d$ is Abelian, it is amenable, and therefore, the C*-crossed product $\A\rtimes_\alpha\Z^d$ is *-isomorphic to
  \begin{equation*}
    C^\ast\left( \pi(a), U^z : a \in \A, z \in \Z^d \right)
  \end{equation*}
  and we shall work with this identification for the remainder of this paper.

  Our construction begins with a natural extension of the operator $\Dirac$ to $\Hilbert[S]$. For any $\xi = (\xi_z)_{z\in \Z^d} \in \dom{\Dirac_w}$, where $\dom{\Dirac_w}$ is the dense subspace of $\Hilbert[S]$:
  \begin{equation}\label{Dw-domain-eq}
    \dom{\Dirac_w} \coloneqq \left\{ (\xi_z)_{z\in\Z^d} \in \Hilbert[S] : \forall z \in \Z^d \quad \xi_z \in \dom{\Dirac}, \sum_{z \in \Z^d} w(z)^2 \norm{\Dirac\xi_z}{\Hilbert}^2 < \infty \right\} \text,
  \end{equation}
  we define
  \begin{equation}\label{Dw-eq}
    \Dirac_w \xi \coloneqq \left( w(z) \Dirac \xi_z \right)_{z \in \Z^d} \text.
  \end{equation}
  We now connect the operator $\Dirac_w$ to $\Lip_w$ introduced in the previous section.
  \begin{lemma}\label{Dw-Lipw-lemma}
    For all $a\in \A_w$, we have $\pi(a) \dom{\Dirac_w} \subseteq \dom{\Dirac_w}$, and
    \begin{equation*}
      \Lip_w(a) = \norm{[\Dirac_w,\pi(a)]}{\B(\Hilbert[S])} \text.
    \end{equation*}
  \end{lemma}

  \begin{proof}
    Let $a \in \A_w$. If $\xi = (\xi_z)_{z\in\Z^d} \in \dom{\Dirac_w}$, then $\pi(a)\xi = (\alpha^{-z}(a)\xi_z)_{z\in\Z^d}$. Now $\alpha^{-z}(a) \dom{\Dirac}\subseteq\dom{\Dirac}$ by assumption, and (since $(x+y)^2 \leq 2x^2 + 2y^2$ for $x,y\in\R$):
    \begin{align*}
      \sum_{z\in\Z^d} w(z)^2 \norm{\Dirac (\pi(a)\xi)_z}{\Hilbert}^2 &= \sum_{z\in\Z^d} w(z)^2 \norm{\Dirac \alpha^{-z}(a)\xi_z}{\Hilbert}^2 \\
      &= \sum_{z\in\Z^d} w(z)^2 \norm{\alpha^{-z}(a) \Dirac\xi_z + [\Dirac,\alpha^{-z}(a)]\xi_z}{\Hilbert}^2 \\
      &\leq 2\norm{a}{\A}^2\sum_{z\in\Z^d} w(z)^2 \norm{\Dirac\xi_z}{\Hilbert}^2 + 2 \sum_{z\in\Z^d} \norm{w(z)[\Dirac,\alpha^{-z}(a)]\xi_z}{\Hilbert}^2 \\
      &\leq 2\norm{a}{\A}^2 \norm{\Dirac_w\xi}{\Hilbert[S]}^2 + 2 \sup_{z\in\Z^d}\norm{w(z)[\Dirac,\alpha^{-z}(a)]}{\B(\Hilbert)}^2 \sum_{z\in\Z^d} \norm{\xi_z}{\Hilbert}^2 \\
      &\leq 2 \norm{a}{\A}^2 \norm{\Dirac_w\xi}{\Hilbert[S]}^2 + 2 \Lip_w(a)^2 \norm{\xi}{\Hilbert[S]}^2 < \infty \text.
    \end{align*}
    So $\pi(a)\xi \in \dom{\Dirac_w}$. Moreover, for all $\xi\in\dom{\Dirac_w}$,
    \begin{equation*}
      ([\Dirac_w,\pi(a)]\xi)_{z\in\Z^d}= \left( w(z)[\Dirac,\alpha^{-z}(a)]\xi_z \right)_{z\in\Z^d}\text.
    \end{equation*}
    Thus this commutator extends to a bounded diagonal operator, whose norm is
    \begin{equation*}
      \opnorm{[\Dirac_w,\pi(a)]}{}{\Hilbert[S]}=\sup_{z\in\Z^d}w(z) \opnorm{[\Dirac,\alpha^{-z}(a)]}{}{\Hilbert} =\Lip_w(a)\text.
    \end{equation*}
    This completes our lemma.
  \end{proof}

  We introduce the usual calculus on $C^\ast(\Z^d) = C(\T^d)$ on $\ell^2(\Z^d)$, extended for our current presentation to $\Hilbert[S]$. For each $j \in \{1,\ldots,d\}$, we define
  \begin{equation*}
    \dom{\partial_j} \coloneqq \left\{ (\xi_z)_{(z_1,\ldots,z_d) \in \Z^d} \in\Hilbert[S] : \sum_{z \in \Z^d} |z_j|^2 \norm{\xi_z}{\Hilbert}^2 < \infty \right\}
  \end{equation*}
  and for all $\xi = (\xi_z)_{z \in \Z^d} \in \dom{\partial_j}$,
  \begin{equation*}
    \partial_j \xi = \left( z_j \xi_z \right)_{(z_1,\ldots,z_d)\in \Z^d} \text.
  \end{equation*}

  Let now $\gamma_0,\gamma_1,\ldots,\gamma_d,\gamma_{d+1}$ be $d+2$ anticommuting self-adjoint unitaries on $\C^q$ for some appropriate $q$, for instance $q = 2^{\lfloor \frac{d+2}{2}\rfloor}$. Fix $q$ henceforth. By assumption,
  \begin{equation*}
    \gamma_j \gamma_k + \gamma_k \gamma_j =
    \begin{cases}
      2 I_q\text{ if }k=j\text, \\
      0 \text{ otherwise.}
    \end{cases}
  \end{equation*}

  We now have all the components needed to construct our Dirac operator for the C*-crossed-product $\A\rtimes_\alpha\Z^d$. We will work on the Hilbert space $\Hilbert[S]\otimes \C^q$, and by abuse of notation, we use the same notation $\pi(a)$ (for $a\in\A$) and $U^k$ (for $k\in \Z^d$) for the operators $\pi(a)\otimes I_q$ and $U^k\otimes I_q$ on $\Hilbert[S]\otimes\C^q$, respectively.

  For all $\varepsilon\geq 0$, we define
  \begin{equation}\label{nabla-def-eq}
    \nabla_{w,\varepsilon} = \Dirac_w \otimes \gamma_0 + \sum_{j=1}^d \partial_j \otimes \gamma_j + \varepsilon (1_{\Hilbert[S]}\otimes\gamma_{d+1})
  \end{equation}
  on
  \begin{equation}\label{nabla-domain-eq}
    \dom{\nabla_w} \coloneqq \left( \dom{\Dirac_w}\cap\bigcap_{j=1}^d \dom{\partial_j} \right) \otimes\C^q \text,
  \end{equation}
  which is easily checked to be a dense subspace of $\Hilbert[S]\otimes\C^q$ (note that this domain is the same for all $\varepsilon\geq 0$, hence our choice of notation).

  The dense subspace $\Hilbert[F]$ of $\Hilbert[S]$ of eigencombinations, defined as the linear span of the elements $\phi_z\otimes e_n$, where $(e_n)_{n\in\N}$ is an orthonormal basis of eigenvectors of $\Hilbert$ for the self-adjoint operator $\Dirac$ with compact resolvent, and
  \begin{equation*}
    \phi_z : m \in \Z^d \mapsto
    \begin{cases}
      1 \text{ if }z=m \text,\\
      0 \text{ otherwise,}
    \end{cases}
  \end{equation*}
  will be helpful in our study of our operators $\nabla_{w,\varepsilon}$. We shall prove that $\Hilbert[F]\otimes\C^q$ is in fact a core for these operators. We fix these notations for this section, as well.

  The parameter $\varepsilon$ is added so that, for $\varepsilon > 0$, our operator $\nabla_{w,\varepsilon}$ is invertible. In fact, we obtain the following properties.

  \begin{proposition}\label{nabla-compact-res-prop}
    With the notation of this section, the operators $\nabla_{w,\varepsilon}$ are self-adjoint operators with compact resolvent such that, on the common core $\Hilbert[F]\otimes\C^q$,
    \begin{equation*}
      \nabla_{w,\varepsilon}^2 = \Dirac_w^2 + \sum_{j=1}^d \partial_j^2 + \varepsilon^2 \text.
    \end{equation*}
    If $\varepsilon > 0$ then $\nabla_{w,\varepsilon}$ is invertible from $\dom{\nabla_w}$ onto $\Hilbert[S]\otimes\C^q$, and moreover its inverse is a bounded operator on $\Hilbert[S]\otimes\C^q$ with norm:
    \begin{equation*}
      \opnorm{\nabla_{w,\varepsilon}^{-1}}{}{\Hilbert[S]\otimes\C^q} \leq \frac{1}{\varepsilon}\text.
    \end{equation*}
  \end{proposition}

  \begin{proof}
    Since $\Dirac$ is self-adjoint, it is easy to check that $\Dirac_w$ is symmetric. Similarly, since $\partial_1$,\ldots,$\partial_d$ are also self-adjoint, $\nabla_{w,\varepsilon}$ is symmetric on its domain.

    Let $\lambda_n \in \R$ be defined by $\Dirac e_n = \lambda_n e_n$ for all $n\in\N$. We also will repeatedly use the following notation. If $\xi \in \Hilbert[S]\otimes\C^q$, then we write $\xi = (\xi_z)_{z\in\Z^d}$ with $\xi_z = \sum_{n\in\N} e_n\otimes \xi_{n,z}$ where $\xi_{n,z} \in\C^q$ for all $(n,z)\in\N\times\Z^d$.

    A direct computation shows that $\nabla_{w,\varepsilon}(\Hilbert[F]\otimes\C^q)\subseteq\Hilbert[F]\otimes\C^q$, specifically
    \begin{equation*}
      \nabla_{w,\varepsilon}(\phi_z\otimes e_n \otimes c) = \left(w(z) \lambda_n (1_{\Hilbert[S]}\otimes \gamma_0) + \sum_{j=1}^d z_j (1_{\Hilbert[S]}\otimes \gamma_j) + \varepsilon(1_{\Hilbert[S]}\otimes\gamma_{d+1})\right) \phi_z\otimes e_n \otimes c \text.
    \end{equation*}
    Since $\gamma_j \gamma_k + \gamma_k \gamma_j =
    \begin{cases}
      2 \text{ if }j=k\text, \\
      0 \text{ otherwise}
    \end{cases}$ for all $j,k \in \{0,\ldots,d+1\}$, and since $\Dirac_w$ commutes with $\partial_j$ for all $j \in \{1,\ldots,d\}$, we then obtain that on $\Hilbert[F]\otimes\C^q$:
    \begin{equation}
      \nabla_{w,\varepsilon}^2 = \Dirac_w^2 + \sum_{j=1}^{d} \partial_j^2 + \varepsilon^2 \text,
    \end{equation}\label{nabla-eq} and in particular,
    \begin{equation*}
      \nabla_{w,\varepsilon}^2 (\phi_z\otimes e_n\otimes c) = t_{n,z} (\phi_z\otimes e_n\otimes c) \text,
    \end{equation*}
    where
    \begin{equation*}
      t_{n,z} = w(z)^2 \lambda_n^2 + \sum_{j=1}^d |z_j|^2 + \varepsilon^2 \text.
    \end{equation*}

    This shows that $\Hilbert[F]\otimes\C^q$ is a core for $\nabla_{w,\varepsilon}$. Indeed, let $\eta\in\dom{\nabla_w}$. If $P_F$ is the projection onto the span of $\{ \phi_z\otimes e_n \otimes c : (n,z) \in F, c \in \C^q \}$ for $F\subseteq \N\times\Z^d$ finite, then $P_F\eta\in\Hilbert[F]\otimes\C^q$ and
    \begin{equation*}
      \norm{\eta-P_F\eta}{\Hilbert[S]\otimes\C^q}^2+ \norm{\nabla_{w,\varepsilon}(\eta-P_F\eta)}{\Hilbert[S]\otimes\C^q}^2 =\sum_{(n,z)\notin F}(1+t_{n,z})\norm{\eta_{n,z}}{\C^q}^2 \xrightarrow{F\subseteq \N\times\Z^d \text{ finite}\to\N\times\Z^d} 0
    \end{equation*}
    where convergence is for a net indexed by the directed set of finite subsets of $\N\times\Z^d$, ordered by inclusion.

    We define the operator $K$ as the unique linear operator on $\Hilbert[F]\otimes\C^q$ such that
    \begin{equation*}
      K(\phi_z\otimes e_n\otimes c) = (1+t_{n,z})^{-1} \phi_z\otimes e_n\otimes c \text.
    \end{equation*}
    By construction, since $(1+t_{n,z})^{-1} \leq 1$, the operator $K$ is bounded, of norm at most $1$, and thus extends by continuity to an operator on $\Hilbert[S]\otimes\C^q$ which we still denote by $K$.

    For each $M>0$, the condition $t_{n,z}\leq M$ first restricts $z$ to a finite set. For each such $z$, the inequality $w(z)^2\lambda_n^2\leq M$, together with $w(z)>0$ and compactness of the resolvent of $\Dirac$, restricts $n$ to a finite set. Consequently, the diagonal operator $K$ is compact.

    Let $\xi \in \Hilbert[S]\otimes\C^q$. Since
    \begin{equation}\label{K-eq}
      \sum_{(n,z)\in\N\times\Z^d}t_{n,z}^{\,2}\norm{(K\xi)_{n,z}}{\C^q}^2 =\sum_{(n,z)\in\N\times\Z^d}\frac{t_{n,z}^{\,2}}{(1+t_{n,z})^2} \norm{\xi_{n,z}}{\C^q}^2 \leq\norm{\xi}{\Hilbert[S]\otimes\C^q}^2\text,
    \end{equation}
    the definition of $\dom{\nabla_w}$ shows that $K\xi\in\operatorname{dom}(\nabla_{w,\varepsilon}^2)$: since $t\leq 1+t^2$ for $t\geq 0$, Expression \eqref{K-eq} shows that $\sum_{(n,z)\in\N\times\Z^d} t_{n,z} \norm{(K\xi)_{n,z}}{\C^q}^2 < \infty$, so $K\xi \in \dom{\nabla_w}$; then Expression \eqref{K-eq} allows us to conclude again that $\nabla_{w,\varepsilon}K\xi \in \dom{\nabla_w}$.

    The identity
    \begin{equation*}
      (\nabla_{w,\varepsilon}^2+1)K\xi=\xi
    \end{equation*}
    therefore holds for every $\xi$. Hence both $\nabla_{w,\varepsilon}+i$ and $\nabla_{w,\varepsilon}-i$ are surjective. Since $\nabla_{w,\varepsilon}$ is symmetric, it is self-adjoint \cite[Theorem VIII3]{ReedSimon}. The compactness of $K=(1+\nabla_{w,\varepsilon}^2)^{-1}$ then gives compact resolvent.

    Finally, if $\varepsilon>0$, Expression \eqref{nabla-eq} and the fact $(\phi_z\otimes e_n)_{z\in\Z^d,n\in\N}$ is an orthonormal basis of $\Hilbert[S]$, give, for every $\eta\in\operatorname{dom}(\nabla_{w,\varepsilon})$,
    \begin{equation*}
      \|\nabla_{w,\varepsilon}\eta\|^2 =\sum_{n,z}t_{n,z}\norm{\eta_{n,z}}{}^2 \geq\varepsilon^2\norm{\eta}{}^2 \text.
    \end{equation*}
    Thus $\nabla_{w,\varepsilon}$ has closed range and trivial kernel. Since it is also self-adjoint, its range is dense, so it is invertible and
    \begin{equation*}
      \|\nabla_{w,\varepsilon}^{-1}\|\leq\varepsilon^{-1} \text.
    \end{equation*}
    This concludes our proof.
  \end{proof}

  In our current framework, there may be no dense *-subalgebra of $\A$ such that the commutator $\norm{[\nabla_{w,\varepsilon},\pi(\cdot)]}{\B(\Hilbert[S]\otimes\C^q)}$ is defined and finite. Instead, we use the framework defined in \cite{Latremoliere26b} and introduce a twist, and associated twisted spectral triple.

  \subsection{The linear twist}

  In this subsection, we fix an admissible weight $w$ adapted to $\alpha$ and $\Lip$. For all $k \in \Z^d$, we define, for all $z \in \Z^d$,
  \begin{equation*}
    \mu_k(z) = 1 - \frac{w(z-k)}{w(z)} \text,
  \end{equation*}
  and let $M_k$ be defined by setting, for all $\xi \in \Hilbert[S]$ and for all $z \in \Z^d$,
  \begin{equation*}
    (M_k\xi)(z) = \mu_k(z)\xi(z) \text.
  \end{equation*}
  Note that $M_k$ is bounded by assumption on our weight $w$, namely
  \begin{equation*}
    \norm{M_k}{\B(\Hilbert[S])}\leq 1+r_w(-k) \text,
  \end{equation*}
  where $r_w$ was defined in Expression \eqref{r-eq}.

  Admissibility of the weight $w$ also implies that $U^k$ preserves $\dom{\Dirac_w}$. Indeed, for $\eta\in\dom{\Dirac_w}$,
  \begin{equation}\label{Uk-domDw-eq}
    \sum_{z\in\Z^d} w(z)^2\norm{\Dirac\eta_{z-k}}{\Hilbert}^2 = \sum_{z\in\Z^d} w(z+k)^2\norm{\Dirac\eta_{z}}{\Hilbert}^2 \leq r_w(k)^2\norm{\Dirac_w\eta}{\Hilbert[S]}^2\text,
  \end{equation}
  so $U^k\eta = (\eta_{z-k})_{z\in\Z^d} \in \dom{\Dirac_w}$.

  We begin with a norm bound for a particular operator which will appear as a ``defect'' in the commutator with our candidates $\nabla_{w,\varepsilon}$ for Dirac operators on the C*-crossed-product.
  \begin{lemma}
    Let $\varepsilon>0$ and $k\in\Z^d$. If
    \begin{equation*}
      B_k \coloneqq M_k\Dirac_w U^k\otimes\gamma_0
    \end{equation*}
    on $\dom{\nabla_w}$, then $B_k\nabla_{w,\varepsilon}^{-1}$ is a bounded operator and
    \begin{equation*}
      \opnorm{B_k\nabla_{w,\varepsilon}^{-1}}{}{\Hilbert[S]\otimes\C^q} \leq \sup_{z\in\Z^d} \left|\frac{w(z+k)}{w(z)}-1\right| <\infty \text.
    \end{equation*}
  \end{lemma}

  \begin{proof}
    We will use the notation introduced in the proof of Proposition \ref{nabla-compact-res-prop}. In particular, for all $(n,z) \in \N\times\Z^d$, we set:
    \begin{equation*}
      t_{n,z}=w(z)^2\lambda_n^2+\sum_{j=1}^d|z_j|^2 +\varepsilon^2 \text.
    \end{equation*}
    Let $\xi\in\Hilbert[S]\otimes\C^q$ and set $\eta=\nabla_{w,\varepsilon}^{-1}\xi$ (which is well-defined by Proposition \ref{nabla-compact-res-prop}). Write $\xi=(\xi_z)_{z\in\Z^d}$ and $\eta = (\eta_z)_{z\in\Z^d}$, then $\xi_z=\sum_{n\in\N} e_n\otimes\xi_{n,z}$ and $\eta_z=\sum_{n\in\N} e_n\otimes\eta_{n,z}$ for all $z \in \Z^d$, with $\xi_{n,z},\eta_{n,z} \in \C^q$.

    By the computation in Proposition \ref{nabla-compact-res-prop},
    \begin{equation*}
      \nabla_{w,\varepsilon}^2(\phi_z\otimes e_n\otimes c) = t_{n,z}(\phi_z\otimes e_n\otimes c)
    \end{equation*}
    for all $(n,z)\in\N\times\Z^d$ and $c\in \C^q$. Therefore, since $\nabla_{w,\varepsilon}$ is self-adjoint, and $(\phi_z\otimes e_n)_{z \in \Z^d, n \in \N}$ is an orthonormal basis of $\Hilbert[S]$,
    \begin{equation*}
      \norm{\xi_{n,z}}{\C^q}^2=t_{n,z}\norm{\eta_{n,z}}{\C^q}^2 \text.
    \end{equation*}

    By definition, for all $(n,z) \in \N\times\Z^d$, and again using $\Dirac e_n = \lambda_n e_n$,
    \begin{align*}
      (B_k\eta)_{n,z+k} &=\bigl((M_k\Dirac_wU^k)\otimes\gamma_0\;\eta\bigr)_{n,z+k}\\
      &=\mu_k(z+k)w(z+k)\lambda_n\gamma_0\eta_{n,z}\\
      &=\left(1-\frac{w(z)}{w(z+k)}\right) w(z+k)\lambda_n\gamma_0\eta_{n,z}\\
      &=\bigl(w(z+k)-w(z)\bigr) \lambda_n\gamma_0\eta_{n,z}\text.
    \end{align*}
    Consequently,
    \begin{align*}
      \norm{B_k\nabla_{w,\varepsilon}^{-1}\xi}{}^2 &=\sum_{(n,z)\in\N\times\Z^d} |w(z+k)-w(z)|^2\lambda_n^2 \norm{\eta_{n,z}}{\C^q}^2\\
      &=\sum_{(n,z)\in\N\times\Z^d} \left|\frac{w(z+k)}{w(z)}-1\right|^2 \frac{w(z)^2\lambda_n^2}{t_{n,z}} \norm{\xi_{n,z}}{\C^q}^2\\
      &\leq \left(\sup_{z\in\Z^d} \left|\frac{w(z+k)}{w(z)}-1\right|\right)^2 \norm{\xi}{\Hilbert[S]\otimes\C^q}^2 \text,
    \end{align*}
    since $w(z)^2\lambda_n^2\leq t_{n,z}$. Our bound follows since the square root is increasing. The supremum is finite because it is at most $1+r_w(k)$.
  \end{proof}

  We now study the commutator between $\nabla_{w,\varepsilon}$ and some elements in our C*-crossed-product. For $\xi \in \dom{\Dirac_w}$, we observe that:
  \begin{align*}
    [\Dirac_w,U^k]\xi(m) &= (\Dirac_wU^k\xi)(m)-(U^k\Dirac_w\xi)(m)\\
    &= w(m)\Dirac(U^k\xi)(m)-(\Dirac_w\xi)(m-k)\\
    &= w(m)\Dirac\xi(m-k)-w(m-k)\Dirac\xi(m-k)\\
    &=\left(1-\frac{w(m-k)}{w(m)}\right) (\Dirac_wU^k\xi)(m) \text,
  \end{align*}
  so, on $\dom{\Dirac_w}$,
  \begin{equation*}
    [\Dirac_w,U^k]=M_k\Dirac_w U^k \text.
  \end{equation*}

  Moreover, $\partial_j U^k=U^k(\partial_j+k_j)$, so $U^k$ preserves $\dom{\nabla_w}$. Since $[\partial_j,\pi(a)]=0$ for all $a\in\A$, and $\pi(a)$ preserves $\dom{\Dirac_w}$ for $a\in\A_w$, we then obtain, on $\dom{\nabla_w}$, for $a\in\A_w$ and $k\in\Z^d$:
  \begin{align*}
    [\nabla_{w,\varepsilon},\pi(a)U^k] &=[\nabla_{w,\varepsilon},\pi(a)]U^k +\pi(a)[\nabla_{w,\varepsilon},U^k]\\
    &=[\Dirac_w,\pi(a)]U^k\otimes\gamma_0 +\pi(a)\left( [\Dirac_w,U^k]\otimes\gamma_0 +\sum_{j=1}^d[\partial_j,U^k]\otimes\gamma_j \right)\\
    &=[\Dirac_w,\pi(a)]U^k\otimes\gamma_0 +\pi(a)\left( M_k \Dirac_w U^k \otimes\gamma_0 +\sum_{j=1}^d[\partial_j,U^k]\otimes\gamma_j \right)\\
    &=[\Dirac_w,\pi(a)]U^k\otimes\gamma_0 +\pi(a)B_k +\sum_{j=1}^d k_j\pi(a)U^k\otimes\gamma_j \text.
  \end{align*}
  This is typically an unbounded operator. Our desired expression, instead, is
  \begin{equation*}
    [\Dirac_w,\pi(a)]U^k\otimes\gamma_0 +\sum_{j=1}^d k_j\pi(a)U^k\otimes\gamma_j \text,
  \end{equation*}
  which we seek to express as the twisted commutator
  \begin{equation*}
    [\nabla_{w,\varepsilon},\pi(a)U^k]_{\psi_\varepsilon} =\nabla_{w,\varepsilon}\pi(a)U^k -\psi_\varepsilon(\pi(a)U^k)\nabla_{w,\varepsilon}
  \end{equation*}
  for a linear map $\psi_\varepsilon$ initially defined on Fourier polynomials in $\A\rtimes_\alpha\Z^d$ and valued in $\B(\Hilbert[S]\otimes\C^q)$.

  The difference thus suggests that, for $\varepsilon>0$, we set
  \begin{equation}\label{psi-eq}
    \psi_\varepsilon(\pi(a)U^k) =\pi(a)\left(U^k+B_k\nabla_{w,\varepsilon}^{-1}\right) \text,
  \end{equation}
  for all $a\in\A$ and $k\in\Z^d$. We now wish to extend $\psi_\varepsilon$, and proceed as follows.

  The dual action $\beta$ of $\T^d$, where $\T = \{ z\in \C : |z| = 1 \}$, on $\A\rtimes_\alpha\Z^d$, is implemented spatially by
  \begin{equation*}
    z\in\T^d \mapsto \beta^z\coloneqq \mathrm{Ad}_{V^z}\text,
  \end{equation*}
  where, for all $(z_1,\ldots,z_d) \in \T^d$, $(m_1,\ldots,m_d) \in \Z^d$ and $\xi \in \Hilbert[S]\otimes\C^q$:
  \begin{equation*}
    V^{z_1,\ldots,z_d} \xi(m_1,\ldots,m_d) = \prod_{j=1}^d z_j^{m_j} \xi(m_1,\ldots,m_d)\text,
  \end{equation*}
  and, for all $a\in \B(\Hilbert[S]\otimes\C^q)$,
  \begin{equation*}
    \mathrm{Ad}_{V^z}(a) = V^{z}a(V^{z})^* \text.
  \end{equation*}
  Of course, $V^z$ is unitary for all $z\in \T^d$.

  Let $\mathds{P}$ be the Haar probability measure on the compact Abelian group $\T^d$, and define, for all $k \in \Z^d$ and for all $T \in \B(\Hilbert[S]\otimes\C^q)$, the Fourier coefficient:
  \begin{equation*}
    \mathcal{F}_k(T) \coloneqq \int_{z \in \T^d} z^{-k} \beta^z(T) \, d\mathds{P}(z) \text,
  \end{equation*}
  where we use the weak integral. Note that $\mathcal{F}_k$ is well-defined and valued in $\B(\Hilbert[S]\otimes\C^q)$ : in fact, it is a linear map of norm $1$. We record that $\mathcal{F}_0$ is the usual conditional expectation defined by the action $\beta$ from $\A\rtimes_\alpha\Z^d$ onto $\pi(\A)$, and for all $k\in \Z^d$ and $b \in \A\rtimes_\alpha\Z^d$:
  \begin{equation*}
    \mathcal{F}_k(b) = \mathcal{F}_0(b U^{-k}) U^k \text.
  \end{equation*}

  For every $b\in\A\rtimes_\alpha\Z^d$ and $k\in\Z^d$,
  \begin{equation*}
    \mathcal{F}_k(b)U^{-k}\in\pi(\A)\text.
  \end{equation*}
  Indeed, this holds for Fourier polynomials and extends to all $b$ by norm continuity of $\mathcal{F}_k$ and closedness of $\pi(\A)$. Consequently, Expression \eqref{psi-eq} may be rewritten as:
  \begin{equation}\label{psi-Fourier-formula}
    \psi_\varepsilon(\mathcal{F}_k(b)) =\mathcal{F}_k(b)U^{-k}\psi_\varepsilon(U^k)\text.
  \end{equation}
  We now define:
  \begin{equation*}
    \dom{\psi_\varepsilon} \coloneqq \left\{ b \in \A\rtimes_{\alpha}\Z^d : \exists T \in \B(\Hilbert[S]\otimes\C^q)\quad \forall k \in \Z^d \qquad \mathcal{F}_k(T) = \psi_\varepsilon\left(\mathcal{F}_k(b)\right) \right\} \text.
  \end{equation*}
  Let $b \in \dom{\psi_\varepsilon}$. Let $T,S \in \B(\Hilbert[S]\otimes\C^q)$ such that for all $k \in \Z^d$,
  \begin{equation*}
    \mathcal{F}_k(T) = \mathcal{F}_k(S) = \psi_\varepsilon( \mathcal{F}_k(b) )\text.
  \end{equation*}
  Then $T = S$ : indeed, let $P_m$ be the projection onto $\C \phi_m\otimes \Hilbert\otimes\C^q$ for $m\in\Z^d$. Since $V^z$ acts on this subspace as multiplication by $z^m$, we have, for all $m,n,k\in\Z^d$,
  \begin{equation*}
    P_n\mathcal{F}_k(T)P_m = \int_{\T^d}z^{n-m-k}\,d\mathds{P}(z)\;P_n T P_m =
    \begin{cases}
      P_n T P_m & \text{if }n-m=k,\\
      0 & \text{otherwise.}
    \end{cases}
  \end{equation*}
  Thus, choosing $k=n-m$, we obtain $P_n T P_m=P_n S P_m$ for every $n,m \in \Z^d$. By density of the sum of the ranges of $P_m$ over $m\in\Z^d$, and by continuity, $T=S$. So if such an element exists, it is unique.

  Therefore, for all $b \in \dom{\psi_\varepsilon}$, we define
  \begin{equation}\label{psi-def-eq}
    \psi_\varepsilon(b) = T \quad \text{ such that }\quad \forall k \in \Z^d\quad \mathcal{F}_k(T) = \psi_\varepsilon(\mathcal{F}_k(b)) \text,
  \end{equation}
  which is equivalent to $\mathcal{F}_k(T) = \mathcal{F}_k(b)U^{-k}\psi_\varepsilon(U^k)$ for all $k\in\Z^d$.

  Moreover, if $b = \pi(a) U^k$ with $a\in \A$, then we can easily check that $b \in \dom{\psi_\varepsilon}$ with $T = \pi(a)(U^k+B_k\nabla_{w,\varepsilon}^{-1})$, as in Expression \eqref{psi-eq}: indeed, $V^z$ commutes with $\pi(a)$, $M_k$, $\Dirac_w$ and $\nabla_{w,\varepsilon}^{-1}$ by construction for all $z\in \T^d$, while $\beta^z(U^k)=z^kU^k$, so
  \begin{equation*}
    \beta^z(T) = z^k T \text.
  \end{equation*}
  Thus, for every $m\in\Z^d$,
  \begin{equation*}
    \mathcal{F}_m(T) =
    \begin{cases}
      T \text{ if }m=k \text,\\
      0 \text{ otherwise}\end{cases}
    = \psi_\varepsilon(\mathcal{F}_m(b)) \text.
  \end{equation*}
  By uniqueness, $\psi_\varepsilon(b)=T$, so this definition indeed extends Expression \eqref{psi-eq}.

  For $b\in\dom{\psi_\varepsilon}$, the defining condition of the domain therefore reads
  \begin{equation}\label{psi-Fourier-compatibility}
    \mathcal{F}_k(\psi_\varepsilon(b)) =\psi_\varepsilon(\mathcal{F}_k(b))=\mathcal{F}_k(b)U^{-k}\psi_\varepsilon(U^k)\text.
  \end{equation}

  \begin{proposition}\label{psi-closed-graph-prop}
    The set $\dom{\psi_\varepsilon}$ is a subspace of $\A\rtimes_\alpha\Z^d$, and the map $\psi_\varepsilon$ is a linear map whose graph is closed from the norm topology on $\A\rtimes_\alpha\Z^d$ to the weak operator topology on $\B(\Hilbert[S]\otimes\C^q)$.
  \end{proposition}

  \begin{proof}
    Let $b,c\in\dom{\psi_\varepsilon}$ and $x,y\in\C$. For every $k\in\Z^d$, the defining formula gives
    \begin{align*}
      \mathcal{F}_k\bigl(x\psi_\varepsilon(b)+y\psi_\varepsilon(c)\bigr) &=\bigl(x\mathcal{F}_k(b)+y\mathcal{F}_k(c)\bigr) U^{-k}\psi_\varepsilon(U^k)\\
      &=\mathcal{F}_k(xb+yc)U^{-k}\psi_\varepsilon(U^k)\\
      &=\psi_\varepsilon(\mathcal{F}_k(xb+yc))\text.
    \end{align*}
    Thus $xb+yc\in\dom{\psi_\varepsilon}$, and uniqueness in the definition yields $\psi_\varepsilon(xb+yc)=x\psi_\varepsilon(b)+y\psi_\varepsilon(c)$.

    Now, we again let $P_m$ be the projection on $\C \phi_m\otimes \Hilbert\otimes\C^q$ for all $m\in \Z^d$. Let $(b_j)_{j \in J}$ be a net in $\dom{\psi_\varepsilon}$ converging in norm to $b$, and suppose that $(\psi_\varepsilon(b_j))_{j \in J}$ converges to $T$ in the weak operator topology.

    Fix $k,m\in\Z^d$, $\ell\in\N$, and $c\in\C^q$, and set $\xi=\phi_m\otimes e_\ell\otimes c$. By definition of $\mathcal{F}_k$, for every $S \in \B(\Hilbert[S]\otimes\C^q)$, we have
    \begin{equation*}
      \mathcal{F}_k(S)\xi=P_{m+k}S\xi \text.
    \end{equation*}
    Consequently,
    \begin{align*}
      \mathcal{F}_k(T)\xi &=P_{m+k}T\xi\\
      &={\mathrm{weak-}\lim_{j\in J}}\, P_{m+k}\psi_\varepsilon(b_j)\xi\\
      &={\mathrm{weak-}\lim_{j \in J}}\, \mathcal{F}_k(\psi_\varepsilon(b_j))\xi\\
      &={\mathrm{weak-}\lim_{j \in J}}\, \psi_\varepsilon(\mathcal{F}_k(b_j))\xi \text.
    \end{align*}

    Moreover, by the defining formula for $\psi_\varepsilon$, and since $\mathcal{F}_k$ has norm $1$, we conclude:
    \begin{align*}
      \norm{\psi_\varepsilon(\mathcal{F}_k(b_j))-\psi_\varepsilon(\mathcal{F}_k(b))}{\B(\Hilbert[S]\otimes\C^q)} &=\norm{\mathcal{F}_k(b_j-b)U^{-k}\psi_\varepsilon(U^k)}{\B(\Hilbert[S]\otimes\C^q)}\\
      &\leq\norm{b_j-b}{\A\rtimes_\alpha\Z^d}\,\norm{\psi_\varepsilon(U^k)}{\B(\Hilbert[S]\otimes\C^q)} \xrightarrow{j \in J} 0 \text.
    \end{align*}

    The preceding estimate implies that $\psi_\varepsilon(\mathcal{F}_k(b_j))\xi$ converges in norm to $\psi_\varepsilon(\mathcal{F}_k(b))\xi$. By uniqueness of the weak limit, we conclude:
    \begin{equation*}
      \mathcal{F}_k(T)\xi =\psi_\varepsilon(\mathcal{F}_k(b))\xi \text.
    \end{equation*}

    The vectors $\phi_m\otimes e_\ell\otimes c$ span a dense subspace, so these bounded operators agree. Since $k$ was arbitrary, $b\in\dom{\psi_\varepsilon}$ and $\psi_\varepsilon(b)=T$, as required.
  \end{proof}

  \subsection{Twisted commutators}

  We now check that our choice of a twist in the previous section does lead to the desired twisted commutators. We continue to use the notation of the previous section.

  \begin{lemma}
    For all $a \in\A_w$, and for all $k\in\Z^d$,
    \begin{equation*}
      \pi(a)\dom{\nabla_w}\subseteq\dom{\nabla_w} \qquad \text{ and }\qquad U^k\dom{\nabla_w}\subseteq\dom{\nabla_w} \text.
    \end{equation*}
  \end{lemma}

  \begin{proof}
    Let $a\in\A_w$. By Lemma \ref{Dw-Lipw-lemma}, we know that $\pi(a)\dom{\Dirac_w}\subseteq\dom{\Dirac_w}$. By construction, $\pi(a)\dom{\partial_j}\subseteq\dom{\partial_j}$ for all $j\in\{1,\ldots,d\}$. So $\pi(a)\dom{\nabla_w}\subseteq\dom{\nabla_w}$.

    Let now $k\in\Z^d$. By Expression \eqref{Uk-domDw-eq}, we have $U^k\dom{\Dirac_w}\subseteq\dom{\Dirac_w}$. On the other hand, for all $j\in\{1,\ldots,d\}$ and $\xi = (\xi_z)_{z\in\Z^d} \in\dom{\partial_j}$,
    \begin{align*}
      \sum_{z\in\Z^d}|z_j|^2\norm{(U^k\xi)_z}{\Hilbert}^2 &=\sum_{z\in\Z^d}|z_j+k_j|^2\norm{\xi_z}{\Hilbert}^2\\
      &\leq 2\norm{\partial_j\xi}{\Hilbert[S]}^2 +2|k_j|^2\norm{\xi}{\Hilbert[S]}^2<\infty\text.
    \end{align*}
    Thus $U^k\dom{\partial_j}\subseteq\dom{\partial_j}$, and hence $U^k\dom{\nabla_w}\subseteq\dom{\nabla_w}$.

    This completes our proof.
  \end{proof}

  Now, while the domain of the twist $\psi_\varepsilon$ is $\dom{\psi_\varepsilon}$, Definition \ref{metric-twisted-spectral-def} involves a (possibly smaller) domain where not only the twist is defined, but also the twisted commutators. We thus introduce
  \begin{equation*}
    \dom{\rho_\varepsilon} = \Bigl\{ b \in \dom{\psi_\varepsilon} : b \dom{\nabla_{w}}\subseteq\dom{\nabla_w}\text{, } [\nabla_{w,\varepsilon},b]_{\psi_\varepsilon} \text{ is bounded } \Bigr\} \text.
  \end{equation*}
  In particular, all Fourier polynomials with coefficients in $\A_w$ are contained in $\dom{\rho_{\varepsilon}}$, so this $\C$-subspace of $\A\rtimes_\alpha\Z^d$ is dense in norm.

  We then define
  \begin{equation}\label{rho-eq}
    \rho_\varepsilon \coloneqq \psi_{\varepsilon}\left|_{\dom{\rho_\varepsilon}}\right. \text.
  \end{equation}
  Since $M_0=0$, we have $1\in\dom{\rho_\varepsilon}$ and $\rho_\varepsilon(1)=1$.

  We now reach a first milestone in our construction.
  \begin{theorem}\label{twisted-spectral-triple-thm}
    Let $(\A,\Hilbert,\Dirac)$ be a spectral triple. Let $\alpha$ be an action of $\Z^d$ on $\A$, and let $w$ be an admissible weight such that $\A_w$ is dense in $\A$. Let $\Hilbert[S] \coloneqq \ell^2(\Z^d)\otimes\Hilbert$, and let $q$ be chosen so that there exist $d+2$ self-adjoint anticommuting unitaries $\gamma_0,\ldots,\gamma_{d+1}$ acting on $\C^q$.

    For all $\varepsilon > 0$, if $\rho_\varepsilon : \dom{\rho_\varepsilon} \to \B(\Hilbert[S]\otimes\C^q)$ is defined in Expressions \eqref{rho-eq}, and if $\nabla_{w,\varepsilon}$ is defined on $\dom{\nabla_w}$ by Expressions \eqref{nabla-def-eq} and \eqref{nabla-domain-eq}, then
    \begin{equation*}
      \left( \A\rtimes_\alpha\Z^d , \Hilbert[S]\otimes\C^q, \nabla_{w,\varepsilon}, \rho_\varepsilon \right)
    \end{equation*}
    is a twisted spectral triple such that, for all $a \in \A_w$ and for all $k \coloneqq (k_1,\ldots,k_d)\in \Z^d$,
    \begin{equation*}
      [\nabla_{w,\varepsilon}, \pi(a) U^k]_{\rho_\varepsilon} = [\Dirac_w,\pi(a)]U^k \otimes\gamma_0 + \sum_{j=1}^d k_j \pi(a)U^k\otimes\gamma_j \text.
    \end{equation*}
    In particular, $\Lip_w(a) = \norm{[\nabla_{w,\varepsilon},\pi(a)]}{\B(\Hilbert[S]\otimes\C^q)}$ and $\norm{[\nabla_{w,\varepsilon},U^k]_{\rho_\varepsilon}}{\B(\Hilbert[S]\otimes\C^q)} = \norm{k}{\R^d}$.
  \end{theorem}

  \begin{proof}
    Proposition \ref{nabla-compact-res-prop} proves that $\nabla_{w,\varepsilon}$ is self-adjoint with a compact resolvent for all $\varepsilon > 0$.

    Since $[\partial_j,\pi(a)] = 0$ and $\psi_\varepsilon(\pi(a)) = \pi(a)$ for all $a\in \A_w$, we immediately get
    \begin{align*}
      [\nabla_{w,\varepsilon},\pi(a)]_{\rho_\varepsilon} &= \nabla_{w,\varepsilon}\pi(a) - \psi_\varepsilon(\pi(a))\nabla_{w,\varepsilon} \\
      &=[\Dirac_w,\pi(a)]\otimes\gamma_0 \text,
    \end{align*}
    which is bounded since $a\in \A_w$. On the other hand:
    \begin{align*}
      [\nabla_{w,\varepsilon}, U^k]_{\rho_\varepsilon} &= \nabla_{w,\varepsilon} U^k - \psi_\varepsilon(U^k) \nabla_{w,\varepsilon} \\
      &= [\nabla_{w,\varepsilon}, U^k] + (U^k - \psi_{\varepsilon}(U^k))\nabla_{w,\varepsilon} \\
      &= \sum_{j=1}^d [\partial_j,U^k]\otimes\gamma_j + B_k + (-B_k\nabla_{w,\varepsilon}^{-1})\nabla_{w,\varepsilon} \\
      &= \sum_{j=1}^d [\partial_j,U^k]\otimes\gamma_j \\
      &= \sum_{j=1}^d k_j U^k\otimes\gamma_j \text,
    \end{align*}
    which is bounded as well. By the monomial formula above and linearity, we conclude that $\nabla_{w,\varepsilon}$ has bounded twisted commutators with the linear span of $\pi(a)U^k$ for $a\in\A_w$ and $k\in\Z^d$, which is dense.
  \end{proof}

  We will see that in fact, $\nabla_{w,\varepsilon}$ has bounded twisted commutators on a larger domain, as we now focus on the metric properties of our newly constructed twisted spectral triple.

  \subsection{The Lip-norm}

  We now focus on the Lipschitz seminorm associated with our twisted spectral triple. We continue using the notation in this section. Fix $\kappa>0$. The natural Lipschitz seminorm for $(\A\rtimes_\alpha\Z^d,\Hilbert[S]\otimes\C^q, \nabla_{w,\varepsilon},\psi_\varepsilon)$ is defined on the domain
  \begin{equation*}
    \dom{\Lip_{w,\varepsilon,\kappa}} \coloneqq \left\{ b\in\A\rtimes_\alpha\Z^d: b,b^\ast \in \dom{\rho_\varepsilon} \right\} \text.
  \end{equation*}
  For $b\in\dom{\Lip_{w,\varepsilon,\kappa}}$, we set
  \begin{equation*}
    \begin{split}
      \Lip_{w,\varepsilon,\kappa}(b) \coloneqq \max\Bigl\{& \norm{[\nabla_{w,\varepsilon},b]_{\rho_\varepsilon}} {\B(\Hilbert[S]\otimes\C^q)}, \kappa\norm{b-\rho_\varepsilon(b)} {\B(\Hilbert[S]\otimes\C^q)},\\
      & \norm{[\nabla_{w,\varepsilon},b^\ast]_{\rho_\varepsilon}} {\B(\Hilbert[S]\otimes\C^q)}, \kappa\norm{b^\ast-\rho_\varepsilon(b^\ast)} {\B(\Hilbert[S]\otimes\C^q)} \Bigr\}\text.
    \end{split}
  \end{equation*}
  If $b \in\A\rtimes_\alpha\Z^d \setminus\dom{\Lip_{w,\varepsilon,\kappa}}$, then we set $\Lip_{w,\varepsilon,\kappa}(b) = \infty$.

  We make a first remark. If $b\in\sa{\A\rtimes_\alpha\Z^d}\cap \dom{\Lip_{w,\varepsilon,\kappa}}$, then
  \begin{equation*}
    \Lip_{w,\varepsilon,\kappa}(b) = \max\left\{ \norm{[\nabla_{w,\varepsilon},b]_{\rho_\varepsilon}} {\B(\Hilbert[S]\otimes\C^q)}, \kappa\norm{b-\rho_\varepsilon(b)} {\B(\Hilbert[S]\otimes\C^q)} \right\}
  \end{equation*}
  as in \cite{Latremoliere26b}. Moreover, $\dom{\Lip_{w,\varepsilon,\kappa}}$ is closed for $*$ and dense, so $\dom{\Lip_{w,\varepsilon,\kappa}}\cap\sa{\A\rtimes_\alpha\Z^d}$ is dense as well in $\sa{\A\rtimes_\alpha\Z^d}$. For metric purposes, only the restriction of $\Lip_{w,\varepsilon,\kappa}$ to the self-adjoint part of its domain matters, but we see no reason to exclude general elements for now.

  The Lipschitz seminorm candidate must satisfy several conditions for our twisted spectral triple to be metric. First of all, it must be lower semicontinuous. Now, the restriction of $\Lip_{w,\varepsilon,\kappa}$ to $\domsa{\Lip_{w,\varepsilon,\kappa}}$ is lower semicontinuous by \cite[Theorem 4.9]{Latremoliere26b}, since the proof of lower semicontinuity is independent of the topology induced by $\Kantorovich{\Lip_{w,\varepsilon,\kappa}}$. The same argument applies here, but since technically, our seminorm is defined on a larger domain, we include it again.

  We first record the closed graph property which we will use.

  \begin{lemma}\label{twisted-commutator-closed-graph-lemma}
    If $(b_j)_{j\in J}$ is a net in $\dom{\rho_\varepsilon}$ such that
    \begin{equation*}
      b_j\xrightarrow{j \in J} b\text{ in norm}, \qquad \rho_\varepsilon(b_j) \xrightarrow{j \in J} T\text{ in the weak operator topology},
    \end{equation*}
    and, for some $M > 0$,
    \begin{equation*}
      \sup_{j\in J} \norm{[\nabla_{w,\varepsilon},b_j]_{\rho_\varepsilon}}{\B(\Hilbert[S]\otimes\C^q)} \leq M\text,
    \end{equation*}
    then $b\in \dom{\rho_\varepsilon}$ and $\rho_\varepsilon(b)=T$. Furthermore:
    \begin{equation*}
      \norm{[\nabla_{w,\varepsilon},b]_{\rho_\varepsilon}} {\B(\Hilbert[S]\otimes\C^q)} \leq M\text.
    \end{equation*}
  \end{lemma}

  \begin{proof}
    Since $\rho_\varepsilon(b_j)=\psi_\varepsilon(b_j)$ for all $j\in J$, Proposition \ref{psi-closed-graph-prop} gives $b\in\dom{\psi_\varepsilon}$ and $\psi_\varepsilon(b)=T$.

    The bounded net
    \begin{equation*}
      \left( [\nabla_{w,\varepsilon},b_j]_{\rho_\varepsilon} \right)_{j\in J}
    \end{equation*}
    has a subnet converging in the weak operator topology to some $S\in\B(\Hilbert[S]\otimes\C^q)$ with $\norm{S}{}\leq M$, since closed, norm bounded subsets of $\B(\Hilbert[S]\otimes\C^q)$ are WOT compact (by Banach-Alaoglu theorem).

    Let $\xi\in\dom{\nabla_w}$. Along this subnet, $b_j\xi$ converges in norm to $b\xi$, while
    \begin{equation*}
      \nabla_{w,\varepsilon}b_j\xi = \rho_\varepsilon(b_j)\nabla_{w,\varepsilon}\xi + [\nabla_{w,\varepsilon},b_j]_{\rho_\varepsilon}\xi
    \end{equation*}
    converges weakly to
    \begin{equation*}
      T\nabla_{w,\varepsilon}\xi+S\xi\text.
    \end{equation*}
    The graph of the closed operator $\nabla_{w,\varepsilon}$ is a norm-closed linear subspace, and is therefore weakly closed. Consequently, $b\xi\in\dom{\nabla_w}$ and
    \begin{equation*}
      \nabla_{w,\varepsilon}b\xi = T\nabla_{w,\varepsilon}\xi+S\xi = \psi_\varepsilon(b)\nabla_{w,\varepsilon}\xi+S\xi\text.
    \end{equation*}
    Thus $b\in\dom{\rho_\varepsilon}$ and $\rho_\varepsilon(b)=T$, while
    \begin{equation*}
      [\nabla_{w,\varepsilon},b]_{\rho_\varepsilon}=S\text.
    \end{equation*}
    The estimate follows from $\norm{S}{\B(\Hilbert[S]\otimes\C^q)}\leq M$.
  \end{proof}

  \begin{proposition}\label{Lip-lsc-prop}
    The seminorm $\Lip_{w,\varepsilon,\kappa}$ is lower semicontinuous on $\A\rtimes_\alpha\Z^d$.
  \end{proposition}

  \begin{proof}
    Let $(b_n)_{n\in\N}$ be a sequence in $\dom{\Lip_{w,\varepsilon,\kappa}}$ converging in norm to $b\in\A\rtimes_\alpha\Z^d$, and assume that $\Lip_{w,\varepsilon,\kappa}(b_n)\leq1$ for all $n\in\N$.

    Since $(b_n)_{n\in\N}$ converges in norm, there exists $M>0$ such that $\norm{b_n}{\A\rtimes_\alpha\Z^d}\leq M$ for all $n\in\N$. Hence
    \begin{equation*}
      \norm{\rho_\varepsilon(b_n)}{\B(\Hilbert[S]\otimes\C^q)} \leq \norm{b_n-\rho_\varepsilon(b_n)}{\B(\Hilbert[S]\otimes\C^q)} + \norm{b_n}{\A\rtimes_\alpha\Z^d} \leq\frac{1}{\kappa}+M\text.
    \end{equation*}
    The same estimate holds for $(\rho_\varepsilon(b_n^*))_{n\in\N}$.

    Closed balls in $\B(\Hilbert[S]\otimes\C^q)$ are compact for the weak operator topology, which is metrizable on bounded subsets since $\Hilbert[S]\otimes\C^q$ is separable, as $\Dirac$ has compact resolvent. Therefore, there exists a strictly increasing function $f : \N\to\N$, and $T,T' \in \B(\Hilbert[S]\otimes\C^q)$ such that
    \begin{equation*}
      \mathrm{WOT-}\lim_{n\to\infty} \rho_\varepsilon(b_{f(n)}) = T \qquad\text{and}\qquad \mathrm{WOT-}\lim_{n\to\infty} \rho_\varepsilon(b_{f(n)}^*) = T' \text.
    \end{equation*}

    Moreover,
    \begin{equation*}
      \sup_{n\in\N} \norm{[\nabla_{w,\varepsilon},b_n]_{\rho_\varepsilon}}{\B(\Hilbert[S]\otimes\C^q)} \leq 1
    \end{equation*}
    and
    \begin{equation*}
      \sup_{n\in\N}\norm{[\nabla_{w,\varepsilon},b_n^*]_{\rho_\varepsilon}}{\B(\Hilbert[S]\otimes\C^q)} \leq 1\text.
    \end{equation*}
    Lemma \ref{twisted-commutator-closed-graph-lemma}, applied first to $(b_{f(n)})_{n\in\N}$ and then to $(b_{f(n)}^*)_{n\in\N}$, shows that $b,b^*\in\dom{\rho_\varepsilon}$, and
    \begin{equation*}
      \rho_\varepsilon(b)=T \text{, } \qquad \rho_\varepsilon(b^*)=T' \text{, }
    \end{equation*}
    while
    \begin{equation*}
      \norm{[\nabla_{w,\varepsilon},b]_{\rho_\varepsilon}}{\B(\Hilbert[S]\otimes\C^q)} \leq 1\text{, } \qquad \norm{[\nabla_{w,\varepsilon},b^*]_{\rho_\varepsilon}}{\B(\Hilbert[S]\otimes\C^q)} \leq 1\text.
    \end{equation*}
    In particular, $b\in\dom{\Lip_{w,\varepsilon,\kappa}}$.

    Moreover,
    \begin{equation*}
      \mathrm{WOT-}\lim_{n\to\infty} b_{f(n)}-\rho_\varepsilon(b_{f(n)}) = b-\rho_\varepsilon(b)
    \end{equation*}
    and
    \begin{equation*}
      \mathrm{WOT-}\lim_{n\to\infty} b^\ast_{f(n)}-\rho_\varepsilon(b_{f(n)}^\ast) = b^\ast-\rho_\varepsilon(b^\ast)\text.
    \end{equation*}
    Lower semicontinuity of the operator norm for the weak operator topology therefore implies that
    \begin{equation*}
      \kappa\norm{b-\rho_\varepsilon(b)}{\B(\Hilbert[S]\otimes\C^q)} \leq 1
    \end{equation*}
    and
    \begin{equation*}
      \kappa\norm{b^*-\rho_\varepsilon(b^*)}{\B(\Hilbert[S]\otimes\C^q)} \leq 1 \text.
    \end{equation*}
    Hence $\Lip_{w,\varepsilon,\kappa}(b)\leq1$, as desired.
  \end{proof}

  \begin{remark}
    The map $\psi_\varepsilon$ and the twist $\rho_\varepsilon$ have different domains. The maximal extension $\psi_\varepsilon$ has a norm-to-weak-operator closed graph, but its domain may contain elements whose twisted commutators are not bounded. The actual twist is therefore its restriction $\rho_\varepsilon$ to $\dom{\rho_\varepsilon}$.

    The unrestricted graph of this restriction need not be closed from the norm to the weak operator topology any longer, since $\dom{\rho_\varepsilon}$ need not be norm closed. Lemma \ref{twisted-commutator-closed-graph-lemma} proves instead that its graph is closed on every set on which the twisted commutator is uniformly bounded. This is the closedness property required for a metric twisted spectral triple and used in the preceding proof. Notice that even the identity map, when restricted to a proper Lipschitz domain, need not have a closed graph.
  \end{remark}

  \begin{corollary}
    The domain $\dom{\Lip_{w,\varepsilon,\kappa}}$ is a Banach space for the norm
    \begin{equation*}
      b\in\dom{\Lip_{w,\varepsilon,\kappa}} \longmapsto \max\left\{\norm{b}{\A\rtimes_\alpha\Z^d}, \Lip_{w,\varepsilon,\kappa}(b)\right\}\text.
    \end{equation*}
  \end{corollary}

  \begin{proof}
    This follows immediately from the lower semicontinuity of $\Lip_{w,\varepsilon,\kappa}$, see for instance \cite{Latremoliere16c}.
  \end{proof}

  \subsection{The \MongeKant}

  We now develop some tools which, in addition to contributing to the proof that our twisted spectral triple is metric, are also of general usefulness, in particular for establishing later convergence results. The methods of this section are well-understood, and our construction was in part designed to ensure they would indeed apply. The core idea is to use Fourier analysis on the C*-crossed-product, via the dual action $\beta$ of the Lie group $\T^d$.

  We begin with a simple but important observation: the domain of our Lipschitz seminorm is invariant for the dual action.

  \begin{lemma}
    For all $z\in\T^d$, we have
    \begin{equation*}
      \beta^z(\dom{\Lip_{w,\varepsilon,\kappa}}) = \dom{\Lip_{w,\varepsilon,\kappa}}
    \end{equation*}
    and
    \begin{equation*}
      \Lip_{w,\varepsilon,\kappa}\circ\beta^z = \Lip_{w,\varepsilon,\kappa}\text.
    \end{equation*}
  \end{lemma}

  \begin{proof}
    The unitaries $V^z$, for all $z\in\T^d$, commute with $\Dirac_w$ and $\partial_1,\ldots,\partial_d$, while preserving all their domains. Consequently, they commute with $\nabla_{w,\varepsilon}$ and preserve $\dom{\nabla_w}$.

    Let $b\in\dom{\Lip_{w,\varepsilon,\kappa}}$. In particular, $b\dom{\nabla_w}\subseteq\dom{\nabla_w}$. We then observe that
    \begin{align*}
      \beta^z(b)\dom{\nabla_w} &=V^zb(V^z)^*\dom{\nabla_w}\\
      &=V^zb\dom{\nabla_w}\\
      &\subseteq V^z\dom{\nabla_w} =\dom{\nabla_w}\text.
    \end{align*}
    Since $\beta^z(b)^*=\beta^z(b^*)$, the same argument applies to $b^*$.

    We next show that $\psi_\varepsilon$ commutes with the dual action. For all $k\in\Z^d$, we have
    \begin{equation*}
      \mathcal{F}_k(\beta^z(b))=z^k\mathcal{F}_k(b)\text.
    \end{equation*}
    Therefore,
    \begin{align*}
      \mathcal{F}_k\bigl(\beta^z(\psi_\varepsilon(b))\bigr) &=z^k\mathcal{F}_k(\psi_\varepsilon(b))\\
      &=z^k\psi_\varepsilon(\mathcal{F}_k(b))\\
      &=\psi_\varepsilon(\mathcal{F}_k(\beta^z(b)))\text.
    \end{align*}
    By the definition of $\dom{\psi_\varepsilon}$ and the uniqueness in this definition, $\beta^z(b)\in\dom{\psi_\varepsilon}$ and
    \begin{equation*}
      \psi_\varepsilon(\beta^z(b)) = \beta^z(\psi_\varepsilon(b))\text.
    \end{equation*}
    The same conclusion holds for $b^*$.

    Thus, on $\dom{\nabla_w}$,
    \begin{align*}
      [\nabla_{w,\varepsilon},\beta^z(b)]_{\psi_\varepsilon} &=\nabla_{w,\varepsilon}V^zb(V^z)^* -\psi_\varepsilon(V^zb(V^z)^*)\nabla_{w,\varepsilon}\\
      &=V^z\left( \nabla_{w,\varepsilon}b -\psi_\varepsilon(b)\nabla_{w,\varepsilon} \right)(V^z)^*\\
      &=V^z[\nabla_{w,\varepsilon},b]_{\psi_\varepsilon} (V^z)^*\text.
    \end{align*}
    Consequently,
    \begin{equation*}
      \norm{[\nabla_{w,\varepsilon},\beta^z(b)]_{\psi_\varepsilon}}{\B(\Hilbert[S]\otimes\C^q)} = \norm{[\nabla_{w,\varepsilon},b]_{\psi_\varepsilon}}{\B(\Hilbert[S]\otimes\C^q)}\text.
    \end{equation*}
    Moreover,
    \begin{align*}
      \norm{\beta^z(b)-\psi_\varepsilon(\beta^z(b))}{\B(\Hilbert[S]\otimes\C^q)} &=\norm{V^z(b-\psi_\varepsilon(b))(V^z)^*}{\B(\Hilbert[S]\otimes\C^q)}\\
      &=\norm{b-\psi_\varepsilon(b)}{\B(\Hilbert[S]\otimes\C^q)}\text.
    \end{align*}
    The same equalities hold with $b^*$ in place of $b$. Hence $\beta^z(b)\in\dom{\Lip_{w,\varepsilon,\kappa}}$ and
    \begin{equation*}
      \Lip_{w,\varepsilon,\kappa}(\beta^z(b)) = \Lip_{w,\varepsilon,\kappa}(b)\text.
    \end{equation*}
    Applying the same argument to the inverse of $z$ gives
    \begin{equation*}
      \beta^z(\dom{\Lip_{w,\varepsilon,\kappa}}) = \dom{\Lip_{w,\varepsilon,\kappa}}\text.
    \end{equation*}
    This completes our proof.
  \end{proof}

  We also introduce a length function on $\T^d$. For all $z = (\exp(i t_1), \ldots, \exp(i t_d)) \in \T^d$ with $t_1,\ldots,t_d \in (-\pi,\pi]$, we write $\ell(z) = \norm{(t_1,\ldots,t_d)}{\R^d}$, where $\norm{\cdot}{\R^d}$ is the Euclidean norm.

  Next, we introduce the usual integrated form of our dual action.
  \begin{notation}
    Let $\mathds{P}$ be the Haar probability measure on the $d$-torus $\T^d$. For all $f \in L^1(\T^d,\mathds{P})$, and for all $b \in \A\rtimes_\alpha\Z^d$, we define:
    \begin{equation*}
      \beta^f(b) \coloneqq \int_{z\in\T^d} f(z) \beta^z(b) \, d\mathds{P}(z)\text;
    \end{equation*}
    as is easily checked,
    \begin{equation*}
      \norm{\beta^f(b)}{\A\rtimes_\alpha\Z^d}\leq \norm{f}{L^1(\T^d,\mathds{P})} \norm{b}{\A\rtimes_\alpha\Z^d}
    \end{equation*}
    and $\beta^f$ is a linear endomorphism of $\A\rtimes_\alpha\Z^d$.
  \end{notation}

  We now prove a mean value inequality for our Lipschitz seminorm and the dual action. We begin by checking that $\beta^f$ is a Lipschitz linear morphism:

  \begin{lemma}\label{beta-Lip-lemma}
    If $f \in L^1(\T^d,\mathds{P})$, then for all $b \in \dom{\Lip_{w,\varepsilon,\kappa}}$,
    \begin{equation*}
      \Lip_{w,\varepsilon,\kappa}(\beta^f(b)) \leq \Lip_{w,\varepsilon,\kappa}(b) \norm{f}{L^1(\T^d,\mathds{P})} \text.
    \end{equation*}
  \end{lemma}

  \begin{proof}
    Assume first that $f$ is continuous. Approximating the integral defining $\beta^f(b)$ in norm by Riemann sums, using the lower semicontinuity of $\Lip_{w,\varepsilon,\kappa}$ and the invariance of this seminorm under $\beta$, we obtain
    \begin{equation*}
      \Lip_{w,\varepsilon,\kappa}(\beta^f(b)) \leq \int_{z \in \T^d} |f(z)| \Lip_{w,\varepsilon,\kappa}(\beta^z(b))\,d\mathds{P}(z) = \norm{f}{L^1(\T^d,\mathds{P})}\Lip_{w,\varepsilon,\kappa}(b) \text.
    \end{equation*}

    Let now $f \in L^1(\T^d,\mathds{P})$, and let $(f_n)_{n\in\N}$ be a sequence of continuous functions converging to $f$ in $L^1(\T^d,\mathds{P})$. Since
    \begin{equation*}
      \norm{\beta^{f_n}(b)-\beta^f(b)}{\A\rtimes_\alpha\Z^d} \leq \norm{f_n-f}{L^1(\T^d,\mathds{P})}\norm{b}{\A\rtimes_\alpha\Z^d},
    \end{equation*}
    the sequence $(\beta^{f_n}(b))_{n\in\N}$ converges in norm to $\beta^f(b)$. Therefore, by lower semicontinuity,
    \begin{align*}
      \Lip_{w,\varepsilon,\kappa}(\beta^f(b)) &\leq \liminf_{n\to\infty}\Lip_{w,\varepsilon,\kappa}(\beta^{f_n}(b))\\
      &\leq \liminf_{n\to\infty}\norm{f_n}{L^1(\T^d,\mathds{P})}\Lip_{w,\varepsilon,\kappa}(b)\\
      &=\norm{f}{L^1(\T^d,\mathds{P})}\Lip_{w,\varepsilon,\kappa}(b) \text,
    \end{align*}
    as claimed.
  \end{proof}

  \begin{lemma}\label{Fourier-coefficient-Lip-lemma}
    Let $b\in\dom{\Lip_{w,\varepsilon,\kappa}}$ and $k\in\Z^d$, and write
    \begin{equation*}
      \mathcal{F}_k(b)=a_kU^k
    \end{equation*}
    for some $a_k\in\A$. Then $a_k\in\A_w$, and
    \begin{equation*}
      \Lip_w(a_k) \leq \Lip_{w,\varepsilon,\kappa}(\mathcal{F}_k(b)) \leq \Lip_{w,\varepsilon,\kappa}(b) \text.
    \end{equation*}
  \end{lemma}

  \begin{proof}
    By Lemma \ref{beta-Lip-lemma}, applied to the character $z\mapsto z^{-k}$, we have $\mathcal{F}_k(b)\in\dom{\Lip_{w,\varepsilon,\kappa}}$ and
    \begin{equation*}
      \Lip_{w,\varepsilon,\kappa}(\mathcal{F}_k(b)) \leq \Lip_{w,\varepsilon,\kappa}(b) \text.
    \end{equation*}
    Set $c=\mathcal{F}_k(b)=a_kU^k$. Since $c$, $c^*$ and $U^k$, $U^{-k}$ preserve $\dom{\nabla_w}$, the operators
    \begin{equation*}
      \pi(a_k)=cU^{-k} \qquad\text{and}\qquad \pi(a_k)^*=U^kc^*
    \end{equation*}
    preserve $\dom{\nabla_w}$. By applying these operators to vectors supported on a single fiber, we conclude that $\alpha^{-z}(a_k)$ and $\alpha^{-z}(a_k^*)$ preserve $\dom{\Dirac}$ for every $z\in\Z^d$.

    We denote the tracial state of $M_q(\C)$ by $\operatorname{tr}$. On $\dom{\nabla_w}$, the computation of the twisted commutator gives
    \begin{equation*}
      (1_{\Hilbert[S]}\otimes\operatorname{tr}) \left( [\nabla_{w,\varepsilon},c]_{\psi_\varepsilon} (1_{\Hilbert[S]}\otimes\gamma_0) \right) =[\Dirac_w,\pi(a_k)]U^k \text.
    \end{equation*}
    The left-hand side extends to a bounded operator. Since $U^k$ is unitary, $[\Dirac_w,\pi(a_k)]$ thus extends to a bounded operator as well. Its restriction to the fiber indexed by $z\in\Z^d$ is
    \begin{equation*}
      w(z)[\Dirac,\alpha^{-z}(a_k)] \text.
    \end{equation*}
    Consequently, $a_k\in\A_w$, and, since the linear map $1_{\Hilbert[S]}\otimes\operatorname{tr}$ has norm $1$,
    \begin{align*}
      \Lip_w(a_k) &=\norm{[\Dirac_w,\pi(a_k)]}{\B(\Hilbert[S])}\\
      &\leq \norm{[\nabla_{w,\varepsilon},c]_{\psi_\varepsilon}}{\B(\Hilbert[S]\otimes\C^q)}\\
      &\leq \Lip_{w,\varepsilon,\kappa}(\mathcal{F}_k(b))\\
      &\leq \Lip_{w,\varepsilon,\kappa}(b) \text,
    \end{align*}
    as claimed.
  \end{proof}

  We now prove that we can approximate any element $b$ in a ball for the Lipschitz seminorm by elements of the form $\beta^f(b)$ as $f$ ranges over any approximate unit in $L^1(\T^d,\mathds{P})$.

  \begin{lemma}\label{approx-lemma}
    If $b \in \dom{\Lip_{w,\varepsilon,\kappa}}$, and if $(f_n)_{n\in\N}$ is an approximate unit in $L^1(\T^d,\mathds{P})$ with $\int_{\T^d} f_n\, d\mathds{P} = 1$ and $f_n\geq 0$ for all $n\in\N$, then $(\beta^{f_n}(b))_{n\in\N}$ converges to $b$ in norm, and
    \begin{equation*}
      \Lip_{w,\varepsilon,\kappa}(\beta^{f_n}(b)) \leq \Lip_{w,\varepsilon,\kappa}(b) \text.
    \end{equation*}
  \end{lemma}

  \begin{proof}
    A direct computation shows:
    \begin{align*}
      \norm{ b - \beta^{f_n}(b) }{\A\rtimes_\alpha\Z^d} &=\norm{ \int_{z\in \T^d} f_n(z) b \,d\mathds{P}(z) - \int_{z\in\T^d} f_n(z)\beta^z(b)\,d\mathds{P}(z) }{\A\rtimes_\alpha\Z^d} \\
      &\leq \int_{z \in \T^d} f_n(z) \norm{b - \beta^z(b)}{\A\rtimes_\alpha\Z^d} \, d\mathds{P}(z) \\
      &\xrightarrow{n\to\infty} \norm{b-\beta^1(b)}{\A\rtimes_\alpha\Z^d} = 0 \text.
    \end{align*}
    Lemma \ref{beta-Lip-lemma} gives us $\Lip_{w,\varepsilon,\kappa}(\beta^{f_n}(b)) \leq \Lip_{w,\varepsilon,\kappa}(b)$, and our proof is complete.
  \end{proof}

  We call an element of the linear span $\{ a U^k : a \in \A, k \in \Z^d \}$ a Fourier polynomial of $\A\rtimes_\alpha\Z^d$. Elements of the linear span $\{ a U^k : a \in \A_w, k \in \Z^d \}$ will be called Lipschitz Fourier polynomials. Now, by choosing Fej{\'e}r kernels in Lemma \ref{approx-lemma}, we prove that every element in $\dom{\Lip_{w,\varepsilon,\kappa}}$ can be approximated in norm by Lipschitz Fourier polynomials with the same, or lesser, Lipschitz seminorm.

  We recall the following result.
  \begin{lemma}\label{Fejer-exist-lemma}
    For any $n\in\N$, there exists a function $f_n \in L^1(\T^d,\mathds{P})$ with $f_n \geq 0$, $\int_{\T^d} f_n \, d\mathds{P} = 1$, and a finite set $S_n$ of $\Z^d$ with $0 \in S_n$, such that the range of $\beta^{f_n}$ is the linear span of $\{ a U^k : a \in \A, k \in S_n \}$, and
    \begin{equation*}
      \int_{z \in \T^d} f_n(z) \ell(z) \, d\mathds{P}(z) < \frac{1}{n+1} \text;
    \end{equation*}
    moreover $(f_n)_{n\in\N}$ is an approximate unit of $L^1(\T^d,\mathds{P})$.
  \end{lemma}
  Note that the maps $f_n$ in Lemma \ref{Fejer-exist-lemma} are in fact linear combinations of characters of $\T^d$, and are known as Fej{\'e}r kernels.

  Altogether, we thus obtain:
  \begin{corollary}\label{Fejer-approx-cor}
    Every element $b \in \dom{\Lip_{w,\varepsilon,\kappa}}$ is the limit in norm of Lipschitz Fourier polynomials $(b_n)_{n\in\N}$ with $\Lip_{w,\varepsilon,\kappa}(b_n) \leq \Lip_{w,\varepsilon,\kappa}(b)$.
  \end{corollary}

  \begin{proof}
    This follows immediately from Lemma \ref{approx-lemma} applied using the Fej{\'e}r kernels from Lemma \ref{Fejer-exist-lemma}.
  \end{proof}

  To derive our next results, we must introduce notation for the calculus induced by the dual action $\beta$ transporting the Lie algebra to $\A\rtimes_\alpha\Z^d$. We simply denote the exponential map from the Lie algebra $\R^d$ onto the compact connected Lie group $\T^d$ as:
  \begin{equation*}
    \exp : (t_1,\ldots,t_d) \in \R^d \mapsto \left(\exp(it_1), \ldots, \exp(it_d)\right) \in \T^d \text.
  \end{equation*}
  For every $X \in \R^d$, and for any Fourier polynomial $b$ in $\A\rtimes_\alpha\Z^d$, we set:
  \begin{equation*}
    \beta^{X}(b) = \lim_{t\to 0} \frac{1}{t}\left(\beta^{\exp(t X)}(b)-b\right) \text,
  \end{equation*}
  which is well-defined (in fact, it is defined on a larger domain \cite{Bratteli79}, but this will suffice for our purpose; however in general it is not defined for all $b \in \dom{\Lip_{w,\varepsilon,\kappa}}$). One may check that $X \in \R^d\mapsto \beta^X$ is a Lie algebra morphism valued in the derivations over the algebra of Fourier polynomials.

  For any Fourier polynomial $b$, a standard computation shows
  \begin{equation*}
    [\partial_j,b] = -i\beta^{e_j}(b)
  \end{equation*}
  where $(e_1,\ldots,e_d)$ is the canonical basis of $\R^d$.

  We now prove our mean value inequality.
  \begin{lemma}
    For all $b \in \dom{\Lip_{w,\varepsilon,\kappa}}$ and for all $z \in \T^d$,
    \begin{equation*}
      \norm{\beta^{z}(b) - b}{\A\rtimes_\alpha\Z^d} \leq \ell(z) \Lip_{w,\varepsilon,\kappa}(b) \text.
    \end{equation*}
  \end{lemma}

  \begin{proof}
    We will denote the tracial state of $M_q(\C)$ by $\operatorname{tr}$. We also define $\gamma(v) = \sum_{j=1}^d v_j \gamma_j$ for all $v=(v_1,\ldots,v_d) \in \R^d$.

    If $j,k \in \{0,\ldots,d+1\}$ and $j\neq k$, using the relations between our $\gamma$ matrices and the tracial property of $\operatorname{tr}$:
    \begin{equation*}
      \operatorname{tr}(\gamma_j \gamma_k) = \operatorname{tr}(\gamma_k \gamma_j) = \operatorname{tr}(-\gamma_j \gamma_k)
    \end{equation*}
    so $\operatorname{tr}(\gamma_j \gamma_k) = 0$.

    Thus, the following computation follows. Let $b$ be a Lipschitz Fourier polynomial in $\A\rtimes_\alpha\Z^d$. For any $v \in \R^d$,
    \begin{equation*}
      (1_{\Hilbert[S]}\otimes\operatorname{tr})([\nabla_{w,\varepsilon},b]_{\psi_\varepsilon}(1_{\Hilbert[S]}\otimes \gamma(v))) = -i \sum_{j=1}^d v_j \beta^{e_j}(b) = -i \beta^{v}(b) \text.
    \end{equation*}

    Since the norm of the linear operator $1_{\Hilbert[S]}\otimes\operatorname{tr}$ is $1$, and the norm of $\gamma(v)$ is $\norm{v}{\R^d}$, we thus have
    \begin{equation*}
      \norm{ \beta^{v}(b) }{\A\rtimes_\alpha\Z^d} \leq \norm{v}{\R^d} \Lip_{w,\varepsilon,\kappa}(b) \text.
    \end{equation*}

    Therefore, if $z = (\exp(it_1),\ldots,\exp(it_d))$ with $t_1,\ldots,t_d \in (-\pi,\pi]$, and $p:s\in[0,1]\mapsto (\exp(i s t_1),\ldots,\exp(i s t_d))=(p_1(s),\ldots,p_d(s))$, then for every Lipschitz Fourier polynomial $b \in \A\rtimes_\alpha\Z^d$:
    \begin{align*}
      \norm{\beta^z(b) - b}{\A\rtimes_\alpha\Z^d} &=\norm{\int_0^1 \sum_{j=1}^d \frac{p'_j(s)}{i p_j(s)} \beta^{p(s)}(\beta^{e_j}(b)) \, ds}{\A\rtimes_\alpha\Z^d}\\
      &=\norm{\int_0^1 \sum_{j=1}^d t_j \beta^{p(s)}(\beta^{e_j}(b)) \, ds}{\A\rtimes_\alpha\Z^d}\\
      &\leq \int_0^1 \norm{\beta^{p(s)}(\beta^{(t_1,\ldots,t_d)}(b))}{\A\rtimes_\alpha\Z^d}\, ds \\
      &=\int_0^1 \norm{\beta^{(t_1,\ldots,t_d)}(b)}{\A\rtimes_\alpha\Z^d}\, ds \\
      &\leq \int_0^1 \norm{(t_1,\ldots,t_d)}{\R^d} \Lip_{w,\varepsilon,\kappa}(b) \, ds \\
      &= \int_0^1 \ell(z)\Lip_{w,\varepsilon,\kappa}(b) \, ds \\
      &= \ell(z) \Lip_{w,\varepsilon,\kappa}(b) \text,
    \end{align*}
    as desired.

    We now extend this inequality to the whole domain $\dom{\Lip_{w,\varepsilon,\kappa}}$. Let $b \in\dom{\Lip_{w,\varepsilon,\kappa}}$. By Corollary \ref{Fejer-approx-cor}, there exists a sequence $(b_n)_{n\in\N}$ of Fourier polynomials converging to $b$, with $\Lip_{w,\varepsilon,\kappa}(b_n) \leq \Lip_{w,\varepsilon,\kappa}(b)$ for all $n\in\N$. We then have:

    \begin{align*}
      \norm{b - \beta^{z}(b)}{\A\rtimes_\alpha\Z^d} &=\lim_{n\to\infty}\norm{b_n - \beta^{z}(b_n)}{\A\rtimes_\alpha\Z^d}\text{ by continuity, }\\
      &=\liminf_{n\to\infty}\norm{b_n - \beta^{z}(b_n)}{\A\rtimes_\alpha\Z^d}\\
      &\leq \liminf_{n\to\infty} \ell(z) \Lip_{w,\varepsilon,\kappa}(b_n) \\
      &\leq \ell(z) \Lip_{w,\varepsilon,\kappa}(b) \text.
    \end{align*}
    This completes our proof.
  \end{proof}

  We extend our inequality to work with kernels. This corollary, in particular, applied to the functions from Lemma \ref{Fejer-exist-lemma}, would recover our result on approximating elements in $\dom{\Lip_{w,\varepsilon,\kappa}}$ with Fourier polynomials of no greater Lipschitz seminorm, but what it provides us with is a uniform control on such approximations as a function of the Lipschitz seminorm and the kernel itself.

  \begin{corollary}
    For all $f \in L^1(\T^d,\mathds{P})$ with $f \geq 0$ and $\int_{z\in\T^d}f(z)\,d\mathds{P}(z) = 1$, and for all $b \in \dom{\Lip_{w,\varepsilon,\kappa}}$,
    \begin{equation*}
      \norm{b - \beta^f(b)}{\A\rtimes_\alpha\Z^d} \leq \Lip_{w,\varepsilon,\kappa}(b) \int_{z\in\T^d} f(z) \ell(z) \, d\mathds{P}(z) \text,
    \end{equation*}
    and
    \begin{equation*}
      \Lip_{w,\varepsilon,\kappa}(\beta^f(b)) \leq \Lip_{w,\varepsilon,\kappa}(b) \text.
    \end{equation*}
  \end{corollary}

  \begin{proof}
    A direct computation shows that for all $b \in \dom{\Lip_{w,\varepsilon,\kappa}}$:
    \begin{align*}
      \norm{b - \beta^f(b)}{\A\rtimes_\alpha\Z^d} &=\norm{\int_{\T^d} f \, d\mathds{P} b - \int_{z\in\T^d} f(z) \beta^z(b) \, d\mathds{P}(z)}{\A\rtimes_\alpha\Z^d} \\
      &=\norm{\int_{z\in\T^d} f(z)( b - \beta^{z}(b) )\, d\mathds{P}(z)}{\A\rtimes_\alpha\Z^d}\\
      &\leq \int_{z\in\T^d} f(z) \norm{b - \beta^z(b)}{\A\rtimes_\alpha\Z^d} \, d\mathds{P}(z) \\
      &\leq \int_{z\in\T^d} f(z) \ell(z)\Lip_{w,\varepsilon,\kappa}(b) \, d\mathds{P}(z) \\
      &= \Lip_{w,\varepsilon,\kappa}(b) \int_{\T^d} f \ell \, d\mathds{P} \text,
    \end{align*}
    as claimed. The other part of our conclusion follows from Lemma \ref{beta-Lip-lemma}.
  \end{proof}

  We thus conclude:
  \begin{proposition}\label{Lip-metric-prop}
    The distance $\Kantorovich{\Lip_{w,\varepsilon,\kappa}}$ metrizes the weak* topology on $\StateSpace(\A\rtimes_\alpha\Z^d)$.
  \end{proposition}

  \begin{proof}
    Let $b \in \dom{\Lip_{w,\varepsilon,\kappa}}$, and, for all $k\in\Z^d$, write $\mathcal{F}_k(b)=a_kU^k$. By Lemma \ref{Fourier-coefficient-Lip-lemma}, we have $a_k\in\A_w$ and
    \begin{equation*}
      \Lip_w(a_k) \leq \Lip_{w,\varepsilon,\kappa}(\mathcal{F}_k(b)) \leq \Lip_{w,\varepsilon,\kappa}(b) \text.
    \end{equation*}

    Now, let $k\neq 0$. Since
    \begin{equation*}
      \beta^{k}(a_k U^k) = i \norm{k}{\R^d}^2 a_k U^k,
    \end{equation*}
    we conclude, by the directional derivative estimate above, that
    \begin{align*}
      \norm{a_k}{\A} &= \norm{a_k U^k}{\A\rtimes_\alpha\Z^d}\\
      &= \frac{1}{\norm{k}{\R^d}^2} \norm{\beta^{k}(a_k U^k)}{\A\rtimes_\alpha\Z^d}\\
      &\leq \frac{1}{\norm{k}{\R^d}} \Lip_{w,\varepsilon,\kappa}(\mathcal{F}_k(b))\\
      &\leq \frac{1}{\norm{k}{\R^d}} \Lip_{w,\varepsilon,\kappa}(b)\\
      &\leq \Lip_{w,\varepsilon,\kappa}(b) \text.
    \end{align*}

    Fix a state $\mu$ of $\A$, and let $\nu = \mu\circ\mathcal{F}_0$: note that $\nu$ is a state of $\A\rtimes_\alpha\Z^d$. For any $S\subseteq\Z^d$ finite with $0\in S$, let
    \begin{equation*}
      K_S \coloneqq \left\{ b=b^*=\sum_{k \in S} a_k U^k : \forall k \in S \quad a_k \in \dom{\Lip_w}, \Lip_{w,\varepsilon,\kappa}(b) \leq 1, \nu(b) = 0 \right\} \text.
    \end{equation*}
    Now, let
    \begin{equation*}
      K_\mu = \left\{ a \in \domsa{\Lip_w} : \Lip_w(a) \leq 1\text{, } \mu(a) = 0 \right\}\text,
    \end{equation*}
    and
    \begin{equation*}
      K_1 = \left\{ a \in \domsa{\Lip_w} : \Lip_w(a) \leq 1\text{, }\norm{a}{\A} \leq 1 \right\} \text.
    \end{equation*}
    Since $\Lip_w$ is a lower semicontinuous Lipschitz seminorm, $K_\mu$ and $K_1$ are norm compact in $\A$.

    For $b\in K_S$, we have $a_0\in K_\mu$. Moreover, for $k\in S\setminus\{0\}$, the computations above give $\Lip_w(a_k)\leq 1$ and $\norm{a_k}{\A}\leq 1$. Since $\Lip_w$ is invariant under taking adjoints, $\Re a_k,\Im a_k\in K_1$, and thus $a_k\in K_1+iK_1$.

    Therefore, by Tychonoff's theorem, the product $K_\mu\times\prod_{k \in S\setminus\{0\}}(K_1+iK_1)$ is compact as well. Thus the range of the continuous map
    \begin{equation*}
      (a_k)_{k\in S} \in K_\mu\times\prod_{k \in S\setminus\{0\}}(K_1+iK_1) \mapsto \sum_{k \in S} a_k U^k
    \end{equation*}
    is also compact; moreover, it contains $K_S$ via the computations done above. So $K_S$ is totally bounded.

    Now, let
    \begin{equation*}
      \B \coloneqq \left\{ b \in \domsa{\Lip_{w,\varepsilon,\kappa}} : \nu(b) = 0, \Lip_{w,\varepsilon,\kappa}(b) \leq 1 \right\}\text.
    \end{equation*}

    Let $\eta > 0$. Let $N\in\N$ such that $\frac{1}{N+1} < \eta$. Let $f_N$ and $S\subseteq\Z^d$ be given by Lemma \ref{Fejer-exist-lemma}. Since $f_N$ is nonnegative with integral $1$, the map $\beta^{f_N}$ preserves self-adjointness and satisfies $\mathcal{F}_0\circ\beta^{f_N}=\mathcal{F}_0$. Consequently,
    \begin{equation*}
      \beta^{f_N}(\B)\subseteq K_S\cap\B \text.
    \end{equation*}
    Moreover, by the preceding corollary,
    \begin{equation*}
      \Haus{\A\rtimes_\alpha\Z^d} \left( \B, \beta^{f_N}(\B) \right) \leq \int_{\T^d} f_N(z)\ell(z)\,d\mathds{P}(z) < \frac{1}{N+1} < \eta \text.
    \end{equation*}

    Since $\beta^{f_N}(\B)\subseteq K_S$, and since $\eta>0$ is arbitrary, we conclude that $\B$ is totally bounded. Since $\Lip_{w,\varepsilon,\kappa}$ is lower semicontinuous, $\B$ is even compact, though for this result, this is not needed.

    The total boundedness of $\B$ implies that $\Kantorovich{\Lip_{w,\varepsilon,\kappa}}$ metrizes the weak* topology by \cite{Rieffel98a}.
  \end{proof}

  We summarize our construction in the following theorem. The weight controls the growth of the Lipschitz seminorm along the action, while the twist compensates for the failure of the weighted Dirac operator to commute with the shifts. The coordinate operators provide control of the dual action and, consequently, uniform Fourier approximation on Lipschitz balls.

  \begin{theorem}\label{metric-thm}
    Let $(\A,\Hilbert,\Dirac)$ be a metric spectral triple, with associated Lipschitz seminorm $\Lip$, and let $\alpha$ be an action of $\Z^d$ on $\A$. Let $w$ be an admissible weight adapted to $\alpha$ and $\Lip$, in the sense of Definitions \ref{weight-def} and \ref{adapted-def}. Fix $\varepsilon>0$ and $\kappa>0$.

    Let $\Hilbert[S]=\ell^2(\Z^d,\Hilbert)$, equipped with the regular covariant representation $(\pi,U)$ of $(\A,\alpha)$, and choose $d+2$ anticommuting self-adjoint unitaries $\gamma_0,\ldots,\gamma_{d+1}$ on $\C^q$. Define
    \begin{equation*}
      \nabla_{w,\varepsilon} = \Dirac_w\otimes\gamma_0 +\sum_{j=1}^d\partial_j\otimes\gamma_j +\varepsilon\,1_{\Hilbert[S]}\otimes\gamma_{d+1},
    \end{equation*}
    where
    \begin{equation*}
      (\Dirac_w\xi)_z=w(z)\Dirac\xi_z, \qquad (\partial_j\xi)_z=z_j\xi_z,
    \end{equation*}
    on
    \begin{equation*}
      \dom{\nabla_w} = \left( \dom{\Dirac_w}\cap \bigcap_{j=1}^d\dom{\partial_j} \right)\otimes\C^q \text.
    \end{equation*}

    Let $\rho_\varepsilon$ be the linear map with domain and extension defined by Expression \eqref{rho-eq}, satisfying
    \begin{equation*}
      \rho_\varepsilon(\pi(a)U^k) = \pi(a)\left( U^k+ (M_k\Dirac_wU^k\otimes\gamma_0) \nabla_{w,\varepsilon}^{-1} \right)
    \end{equation*}
    for all $a\in\A_w$ and $k\in\Z^d$, where
    \begin{equation*}
      (M_k\xi)_z = \left(1-\frac{w(z-k)}{w(z)}\right)\xi_z \text.
    \end{equation*}

    Then
    \begin{equation*}
      \left( \A\rtimes_\alpha\Z^d, \Hilbert[S]\otimes\C^q, \nabla_{w,\varepsilon}, \rho_\varepsilon \right)
    \end{equation*}
    is a $\kappa$-metric twisted spectral triple.

    More precisely, $\nabla_{w,\varepsilon}$ is self-adjoint with compact resolvent and $\norm{\nabla_{w,\varepsilon}^{-1}}{}\leq\varepsilon^{-1}$. For all $a\in\A_w$ and $k\in\Z^d$, its twisted commutator satisfies
    \begin{equation*}
      [\nabla_{w,\varepsilon},\pi(a)U^k]_{\rho_\varepsilon} = [\Dirac_w,\pi(a)]U^k\otimes\gamma_0 +\sum_{j=1}^d k_j\pi(a)U^k\otimes\gamma_j \text.
    \end{equation*}
    The associated seminorm, given on its self-adjoint domain by
    \begin{equation*}
      \Lip_{w,\varepsilon,\kappa}(b) = \max\left\{ \norm{[\nabla_{w,\varepsilon},b]_{\rho_\varepsilon}}{}, \kappa\norm{b-\rho_\varepsilon(b)}{} \right\},
    \end{equation*}
    is lower semicontinuous, and $\Kantorovich{\Lip_{w,\varepsilon,\kappa}}$ metrizes the weak* topology on $\StateSpace(\A\rtimes_\alpha\Z^d)$.
  \end{theorem}

  \begin{proof}
    Theorem \ref{twisted-spectral-triple-thm} proves that this is a twisted spectral triple and gives the displayed commutator formula. The closed graph condition in Definition \ref{metric-twisted-spectral-def} follows from Lemma \ref{twisted-commutator-closed-graph-lemma}, while $\rho_\varepsilon(1)=1$ was observed after Expression \eqref{rho-eq}. Finally, Proposition \ref{Lip-lsc-prop} proves lower semicontinuity, and Proposition \ref{Lip-metric-prop} proves that $\Kantorovich{\Lip_{w,\varepsilon,\kappa}}$ metrizes the weak* topology.
  \end{proof}

  \subsection{Examples}

  Lemma \ref{exponential-weight-lemma} gives a general source of examples for Theorem \ref{metric-thm}. Indeed, if the action preserves $\dom{\Lip}$, then it admits an adapted exponential weight. Thus, in all the examples below, it is enough to verify that the action preserves the Lipschitz domain. When a quantitative estimate is available, it also gives an explicit choice of weight.

  \begin{example}[Equicontinuous actions]
    Assume that $\alpha$ is equicontinuous for $\Lip$, i.e.
    \begin{equation*}
      \sup_{z\in\Z^d}\Lip(\alpha^z(a))<\infty
    \end{equation*}
    for all $a\in\dom{\Lip}$. We may then choose $w=1$. In this case, $M_k=0$ for all $k\in\Z^d$, and therefore $\psi_\varepsilon=\mathrm{id}$. Hence our construction recovers the usual crossed-product spectral triple for equicontinuous actions, up to the Clifford representation and the invertible term $\varepsilon\gamma_{d+1}$, which does not affect the commutator seminorm.

    This includes rotations of the circle and the resulting quantum tori, odometer actions and Bunce--Deddens algebras, automorphisms of AF algebras which preserve an AF filtration, including product-type automorphisms of UHF algebras, and the $\Z^d$-odometer examples studied by Hawkins, Skalski, White, and Zacharias \cite{Hawkins13}. Thus these give genuinely noncommutative examples in addition to actions on compact spaces.
  \end{example}

  \begin{example}[Commuting diffeomorphisms]
    Let $M$ be a connected compact Riemannian manifold without boundary, and equip $C(M)$ with the metric spectral triple associated with the Hodge--de Rham operator $d+d^*$. Let $\varphi_1,\ldots,\varphi_d$ be commuting $C^1$-diffeomorphisms of $M$, and set
    \begin{equation*}
      \varphi^k = \varphi_1^{k_1}\circ\cdots\circ\varphi_d^{k_d}
    \end{equation*}
    for all $k=(k_1,\ldots,k_d)\in\Z^d$. We define
    \begin{equation*}
      \alpha^k(f)=f\circ\varphi^{-k}
    \end{equation*}
    for all $f\in C(M)$. If
    \begin{equation*}
      C_j = \max\left\{ 1, \norm{d\varphi_j}{\infty}, \norm{d\varphi_j^{-1}}{\infty} \right\},
    \end{equation*}
    then the chain rule gives
    \begin{equation*}
      \Lip(\alpha^k(f)) \leq \prod_{j=1}^d C_j^{|k_j|}\Lip(f)
    \end{equation*}
    for all $k\in\Z^d$ and $f\in\dom{\Lip}$. Thus
    \begin{equation*}
      w(k)=\prod_{j=1}^d C_j^{-|k_j|}
    \end{equation*}
    is adapted, and Theorem \ref{metric-thm} gives a metric twisted spectral triple on
    \begin{equation*}
      C(M)\rtimes_\alpha\Z^d \, .
    \end{equation*}

    For $d=1$, this includes the Arnold cat map and the parabolic automorphisms of $\T^2$ considered by Bellissard, Marcolli, and Reihani \cite{Bellissard10}, as well as the diffeomorphisms considered by Kaad and Kyed \cite{Kaad21}. Taking commuting matrices in $\operatorname{GL}(n,\Z)$ gives higher-rank actions by toral automorphisms. Kaad and Kyed construct compact quantum metrics from an ordinary derivation and a twisted derivation. The construction here instead produces a compact-resolvent twisted spectral triple on the original crossed product.
  \end{example}

  \begin{example}[The Sierpi\'nski gasket]
    Let $SG$ be the standard Sierpi\'nski gasket, equipped with its geodesic distance and with the metric spectral triple of Christensen, Ivan, and Lapidus \cite{Lapidus08}. Its commutator seminorm is the ordinary Lipschitz seminorm for the geodesic distance. Let $r$ be the rotation of $SG$ through an angle of $2\pi/3$, and define an action $\alpha$ of $\Z$ on $C(SG)$ by
    \begin{equation*}
      \alpha^n(f)=f\circ r^{-n}
    \end{equation*}
    for all $n\in\Z$ and $f\in C(SG)$. Since $r$ is an isometry, we have
    \begin{equation*}
      \Lip(\alpha^n(f))=\Lip(f)
    \end{equation*}
    for all $n\in\Z$ and $f\in\dom{\Lip}$. Thus this action is equicontinuous, and our construction applies with $w=1$.

    The homeomorphism group of $SG$ is the dihedral group of order six. 
 In particular, the action above factors through $\Z/3\Z$, and every action of $\Z$ on $C(SG)$ factors through a cyclic subgroup of order one, two, or three. Hence the standard Sierpi\'nski gasket admits no faithful action of $\Z$ by automorphisms of $C(SG)$.
  \end{example}

  \begin{example}[Automorphisms of quantum tori]
    Let $\A_\theta$ be the quantum two-torus in its Weyl presentation, and let $A\in\operatorname{SL}(2,\Z)$. The formula
    \begin{equation*}
      \alpha_{A,\theta}(W_\theta(m))=W_\theta(Am)
    \end{equation*}
    defines an automorphism of $\A_\theta$. For the standard flat metric spectral triple, the canonical derivations satisfy
    \begin{equation*}
      \delta_j\alpha_{A,\theta}^n(a) = \sum_{l=1}^2 (A^n)_{jl} \alpha_{A,\theta}^n(\delta_l(a))
    \end{equation*}
    on the smooth quantum torus. Consequently,
    \begin{equation*}
      \Lip_\theta(\alpha_{A,\theta}^n(a)) \leq \norm{A^n}{}\Lip_\theta(a) \, .
    \end{equation*}
    The same estimate extends to the Lipschitz domain by lower semicontinuity and Fej\'er approximation. Therefore, if
    \begin{equation*}
      C=\max\left\{1,\norm{A}{},\norm{A^{-1}}{}\right\},
    \end{equation*}
    then $w(n)=C^{-|n|}$ is adapted, and Theorem \ref{metric-thm} applies to
    \begin{equation*}
      \A_\theta\rtimes_{\alpha_{A,\theta}}\Z \, .
    \end{equation*}

    Hyperbolic matrices give non-equicontinuous examples, while parabolic matrices give examples with polynomial growth. The lattice real-multiplication systems considered by Bellissard, Marcolli, and Reihani provide related examples \cite{Bellissard10}. We only claim here the metric twisted spectral triple constructed above; its identification with their solvmanifold spectral triple, and the arithmetic information carried by the latter, are separate questions.
  \end{example}

  \begin{example}[Subshifts]
    Let $X$ be a subshift, and let
    \begin{equation*}
      (C(X),\Hilbert,D_X)
    \end{equation*}
    be the canonical spectral triple of Julien and Putnam \cite{Putnam16}. If $U$ denotes the unitary which implements the shift, then $[D_X,U]$ is bounded by \cite[Theorem~3.2]{Putnam16}. Thus, for all $f\in\dom{\Lip}$,
    \begin{align*}
      [D_X,UfU^*] &=[D_X,U]fU^*+U[D_X,f]U^*+Uf[D_X,U^*]
    \end{align*}
    is bounded. The same computation with $U^*$ shows that the shift and its inverse preserve $\dom{\Lip}$. Therefore Lemma \ref{exponential-weight-lemma} applies whenever the base triple is metric.

    Julien and Putnam prove this metric property for linearly recurrent subshifts \cite[Theorem~5.10]{Putnam16}, including primitive substitution subshifts such as the Fibonacci and Thue--Morse systems, and for a large class of Sturmian subshifts \cite[Theorem~5.14 and Corollary~5.15]{Putnam16}, including almost every parameter. Theorem \ref{metric-thm} therefore gives metric twisted spectral triples on the associated crossed products. Their canonical operator also gives a spectral triple on the crossed product, but the metric property of that crossed-product triple is left open in \cite{Putnam16}. Our conclusion concerns the different twisted spectral triple constructed here.
  \end{example}

  \begin{remark}[Comparison and scope]
    The preceding examples illustrate the difference with earlier approaches on how to handle spectral triples for C*-crossed products. For non-equicontinuous actions, Bellissard, Marcolli, and Reihani propose to replace the base space with a metric-bundle algebra, which is nonunital even when the original algebra is unital \cite{Bellissard10}. Kaad and Kyed retain the original crossed product and construct compact quantum metrics using derivations and twisted derivations \cite{Kaad21}, though they are not brought together into a single spectral triple. Our construction remains on the original unital crossed product and produces a single compact-resolvent metric twisted spectral triple. An advantage of our construction is that it fits within our convergence for the spectral propinquity, and is a direct generalization of the equicontinuous case (to which our construction reduces by choosing the trivial weight $1$). This paper focused on actions of $\Z^d$, though we hope it is also the blueprint for further studies of C*-crossed products of other discrete groups. Our focus on $\Z^d$ is owed to our desire to obtain nice quantum metrics from our setting; presumably using groups for which the reduced C*-algebra carries a metric spectral triple are the next natural step. Moreover, some of the papers mentioned above also have constructions which either do not lead to metric spectral triple on the base space, or work with some nonunital situations, which do not fit in the framework described here and will require further study.
  \end{remark}


  \section{Examples of Convergence for the Spectral Propinquity}
  \label{finite-mode-convergence-section}

  While it is noteworthy that our construction of twisted spectral triples on C*-crossed products naturally extends the usual construction from the equicontinuous case to a much more general situation, and that it induces a quantum metric whenever the base spectral triple does, we believe that a test for our construction is a form of continuity for the spectral propinquity for our newly constructed triples. We propose such examples in this section. Our focus of course is on non-equicontinuous actions.

  Our first step is to write a general result for convergence for the propinquity which abstracts our previous work on quantum tori. We then apply this method to our twisted spectral triples, and in turn, apply our work to the example of an Anosov automorphism acting on the continuous family of quantum $2$-tori, indexed by their deformation parameters.

  \subsection{A strategy for proving convergence for the spectral propinquity}

  We first isolate the convergence argument which will be used below. It is an abstract version of the finite-dimensional approximation and continuity methods used for quantum tori in \cite{Latremoliere21a}, which in turn will serve as a starting point for our strategy to prove convergence for our twisted spectral triples. We thus group the hypothesis for our convergence theorem in the following hypothesis.

  We specialize to our situation the idea of a \emph{bridge} as introduced in \cite{Latremoliere13}. Let $\A$ and $\B$ be two unital C*-algebras. A \emph{bridge} $(\B(\Hilbert),\omega)$ is given by the C*-algebra $\B(\Hilbert)$ of all bounded linear operators on a fixed Hilbert space $\Hilbert$ on which both $\A$ and $\B$ are faithfully and unitally represented, and a bounded self-adjoint operator $\omega$ on $\Hilbert$ with norm $1$ such that:
  \begin{equation*}
    \StateSpace_1(\omega) \coloneqq \left\{ \varphi \in \StateSpace(\B(\Hilbert)) : \varphi(\omega) = 1 \right\}
  \end{equation*}
  is not empty. The operator $\omega$ is called the \emph{pivot} of our bridge.

  In \cite{Latremoliere13}, a bridge is more general --- $\B(\Hilbert)$ may be any unital C*-algebra, and the unital *-monomorphisms from $\A$ and $\B$ into the bridge algebra are explicitly specified; we even allow the pivot $\omega$ to be non-self-adjoint, but this will not be necessary here.

  A bridge is a tool to quantify how close two generalized {\qcms s} can be. It serves as the measurement device for the quantum propinquity, which was introduced specifically for {\qcms s}. The notion of bridge and its associated measurement scalar extend trivially to generalized {\qcms s}, but not the quantum propinquity (it is essential for its coincidence property that we work with {\qcms s}).

  If $\gamma=(\B(\Hilbert),\omega)$ is a bridge, then its associated scalar is the \emph{length} $\bridgelength{\gamma}{}$, which is the maximum of two numbers. The first number is the \emph{bridge height} $\bridgeheight{\gamma}$, defined as
  \begin{equation*}
    \max\left\{ \Haus{\Kantorovich{\Lip_\A}}(\StateSpace(\A), \{\varphi\left|_{\A}\right. : \varphi \in \StateSpace_1(\omega) \}), \Haus{\Kantorovich{\Lip_\B}}(\StateSpace(\B), \{\varphi\left|_{\B}\right. : \varphi \in \StateSpace_1(\omega) \}) \right\} \text.
  \end{equation*}
  The second number is the \emph{bridge reach}, which is defined as the maximum of
  \begin{equation*}
    \sup_{\substack{a \in \domsa{\Lip_\A} \\ \Lip_\A(a)\leq 1}} \inf_{\substack{b \in \domsa{\Lip_\B} \\ \Lip_\B(b) \leq 1}} \norm{ a \omega - \omega b }{\B(\Hilbert)}
  \end{equation*}
  and
  \begin{equation*}
    \sup_{\substack{b \in \domsa{\Lip_\B} \\ \Lip_\B(b)\leq 1}} \inf_{\substack{a \in \domsa{\Lip_\A} \\ \Lip_\A(a) \leq 1}} \norm{ a \omega - \omega b }{\B(\Hilbert)}\text.
  \end{equation*}

  Within the context of the dual propinquity, on which the spectral propinquity is built, a bridge is a tool to define tunnels. If $(\B(\Hilbert),\omega)$ is a bridge of length less than or equal to $\varepsilon > 0$, then we associate the tunnel
  \begin{equation*}
    \tau \coloneqq \left(\A\oplus\B, \TLip, \pi_\A, \pi_\B \right)
  \end{equation*}
  where $\pi_\A : (a,b)\in\A\oplus\B\to\A$, $\pi_\B:(a,b)\in\A\oplus\B\to\B$, and for all $(a,b) \in \domsa{\Lip_\A}\oplus\domsa{\Lip_\B}$,
  \begin{equation*}
    \TLip(a,b) \coloneqq \max\left\{ \Lip_\A(a), \Lip_\B(b), \frac{1}{\varepsilon}\norm{a \omega - \omega b}{\B(\Hilbert)} \right\} \text.
  \end{equation*}
  The extent of the tunnel $\tau$ is at most $\varepsilon$.

  For this section, we first formalize how the existence of certain maps and associated estimates provide the existence of tunnels which can be used to prove convergence.

  \begin{hypothesis}\label{finite-mode-hyp}
    Let $T$ be a topological space. For every $t\in T$, let $(\B_t,\Lip_t)$ be a generalized {\qcms}. We assume that every $\B_t$ is faithfully and unitally represented on a common Hilbert space $\Hilbert[K]$. Fix a state $\mu_t\in\StateSpace(\B_t)$ of $\B_t$ and set
    \begin{equation*}
      \mathfrak{L}_t \coloneqq \left\{ b\in\domsa{\Lip_t}: \Lip_t(b)\leq1,\ \mu_t(b)=0 \right\}\text.
    \end{equation*}
    Fix $t_0 \in T$. We assume the following data is available.
    \begin{enumerate}
      \item For every $m\in\N$, for every $t \in T$, we are given linear maps $\Phi_{m,t}:\B_t\longrightarrow\B_t$, preserving adjoints, the unit, and such that
      \begin{equation*}
        \mu_t \circ \Phi_{m,t} = \mu_t \text,
      \end{equation*}
      \item For every $m \in \N$, there exists $\varepsilon_m > 0$ such that, for all $t\in T$ and $b \in \domsa{\Lip_t}$,
      \begin{equation*}
        \Lip_t(\Phi_{m,t}(b)) \leq (1+\varepsilon_m)\Lip_t(b)
      \end{equation*}
      and
      \begin{equation*}
        \norm{b-\Phi_{m,t}(b)}{\B_t} \leq \varepsilon_m\Lip_t(b)
      \end{equation*}
      while
      \begin{equation*}
        \lim_{m\to\infty}\varepsilon_m = 0 \text,
      \end{equation*}

      \item For every $m\in\N$, we are given a compact Hausdorff space
      $C_m$ and, for all $m\in \N$ and $t\in T$, we are given norm-continuous maps
      \begin{equation*}
        P_{m,t}:C_m\longrightarrow\domsa{\Lip_t}
      \end{equation*}
      such that
      \begin{equation*}
        \Phi_{m,t}(\mathfrak L_t) \subseteq P_{m,t}(C_m) \qquad\text{and}\qquad \mu_t(P_{m,t}(x))=0
      \end{equation*}
      for all $x\in C_m$. Moreover,
      \begin{equation*}
        \sup_{x\in C_m} \left| \Lip_t(P_{m,t}(x)) - \Lip_{t_0}(P_{m,t_0}(x)) \right| \xrightarrow{t\in T, t\to t_0} 0\text.
      \end{equation*}
      \item For every $m\in\N$ and every $\xi\in\Hilbert[K]$,
      \begin{equation*}
        \sup_{x\in C_m} \norm{(P_{m,t}(x)-P_{m,t_0}(x))\xi}{\Hilbert[K]} \xrightarrow{t \in T, t\to t_0} 0\text.
      \end{equation*}

      \item For every $m, r\in\N$, there are positive finite-rank contractions $\omega_{m,r}\in\B(\Hilbert[K])$ such that
      \begin{equation*}
        \sup_{x\in C_m} \norm{ \left[ P_{m,t_0}(x), \omega_{m,r} \right]}{\B(\Hilbert[K])} \xrightarrow{r \to \infty} 0\text.
      \end{equation*}
      and $\gamma_{m,r} \coloneqq (\B(\Hilbert[K]),\omega_{m,r})$ is a bridge from $\B_t$ to $\B_{t_0}$ such that
      \begin{equation*}
        \lim_{r\to\infty} \limsup_{t \to t_0} \bridgeheight{\gamma_{m,r}|\Lip_t,\Lip_{t_0}}{} = 0\text.
      \end{equation*}
    \end{enumerate}
  \end{hypothesis}

  \begin{theorem}
    \label{finite-mode-convergence-thm} If we assume Hypothesis \ref{finite-mode-hyp}, then for all $\varepsilon > 0$, there exists a neighborhood $T_\varepsilon$ of $t_0$ in $T$ such that, for all $t \in T_\varepsilon$, there exists a tunnel from $(\B_t,\Lip_t)$ to $(\B_{t_0},\Lip_{t_0})$ of the form:
    \begin{equation*}
      \left(\B_t \oplus \B_{t_0}, \TLip_{t,\varepsilon}, \pi_t, \pi_{t_0} \right)
    \end{equation*}
    where
    \begin{equation*}
      \pi_t(a,b)=a \qquad\text{and}\qquad \pi_{t_0}(a,b)=b
    \end{equation*}
    for all $(a,b)\in\B_t\oplus\B_{t_0}$, and with extent at most $\varepsilon$.
  \end{theorem}

  \begin{proof}
    Let $\varepsilon>0$, and let
    \begin{equation*}
      d_0 \coloneqq \diam{\B_{t_0}}{\Lip_{t_0}} < \infty \text.
    \end{equation*}
    By assumption, there exists $M \in \N$ such that, if $m\geq M$, then
    \begin{equation*}
      \varepsilon_m < \frac{\varepsilon}{16(d_0+1)}\text.
    \end{equation*}
    We fix $m \geq M$. Now, by assumption, there exists $R \in \N$ such that, if $r \geq R$, then
    \begin{equation*}
      \sup_{x\in C_m}\norm{[P_{m,t_0}(x),\omega_{m,r}]}{\B(\Hilbert[K])} <\frac{\varepsilon}{16}
    \end{equation*}
    and
    \begin{equation*}
      \limsup_{t\to t_0}\bridgeheight{\gamma_{m,r}|\Lip_t,\Lip_{t_0}} < \frac{\varepsilon}{8}\text.
    \end{equation*}
    Fix $r \geq R$.

    Write
    \begin{equation*}
      \omega_{m,r}\xi = \sum_{j=1}^N \lambda_j \inner{\xi}{\xi_j}{\Hilbert[K]} \xi_j
    \end{equation*}
    for some orthonormal family $(\xi_j)_{j=1}^N$ and some $\lambda_1,\ldots,\lambda_N\in(0,1]$. By Hypothesis~\ref{finite-mode-hyp}, there exists a neighborhood $T_\varepsilon$ of $t_0$ such that, for all $t\in T_\varepsilon$,
    \begin{equation*}
      \sup_{x\in C_m} \left| \Lip_t(P_{m,t}(x)) - \Lip_{t_0}(P_{m,t_0}(x)) \right| < \frac{\varepsilon}{16(d_0+1)},
    \end{equation*}
    \begin{equation*}
      \sup_{x\in C_m} \norm{ (P_{m,t}(x)-P_{m,t_0}(x))\xi_j }{\Hilbert[K]} < \frac{\varepsilon}{16N}
    \end{equation*}
    for all $j\in\{1,\ldots,N\}$, and
    \begin{equation*}
      \bridgeheight{\gamma_{m,r}} < \frac{\varepsilon}{4}\text.
    \end{equation*}
    Consequently, for every $t\in T_\varepsilon$,
    \begin{align*}
      &\sup_{x\in C_m} \norm{(P_{m,t}(x)-P_{m,t_0}(x))\omega_{m,r}}{\B(\Hilbert[K])}\\
      &\quad\leq \sum_{j=1}^N \lambda_j \sup_{x\in C_m} \norm{(P_{m,t}(x)-P_{m,t_0}(x))\xi_j}{\Hilbert[K]} <\frac{\varepsilon}{16}\text.
    \end{align*}

    Fix $t\in T_\varepsilon$.

    We are now ready to prove our estimate on the reach of the bridges $\gamma_{m,r}$. Let $b\in\mathfrak{L}_t$. Since $\Phi_{m,t}(\mathfrak{L}_t) \subseteq P_{m,t}(C_m)$ by assumption, there exists $x\in C_m$ such that
    \begin{equation*}
      \Phi_{m,t}(b)=P_{m,t}(x)\text.
    \end{equation*}
    We then have
    \begin{align*}
      \Lip_{t_0}(P_{m,t_0}(x)) &\leq \Lip_t(P_{m,t}(x)) + \left| \Lip_t(P_{m,t}(x)) - \Lip_{t_0}(P_{m,t_0}(x)) \right|\\
      &< 1+\varepsilon_m + \frac{\varepsilon}{16(d_0+1)}\\
      &< 1+\frac{\varepsilon}{8(d_0+1)}\text.
    \end{align*}
    Hence, if we set
    \begin{equation*}
      c \coloneqq \frac{1}{1+\frac{\varepsilon}{8(d_0+1)}} P_{m,t_0}(x)
    \end{equation*}
    then $c \in \mathfrak L_{t_0}$. Moreover,
    \begin{equation*}
      \norm{P_{m,t_0}(x)}{\B_{t_0}} < d_0 \left( 1+\frac{\varepsilon}{8(d_0+1)}\right)\text.
    \end{equation*}
    Therefore,
    \begin{align*}
      \norm{b\omega_{m,r} - \omega_{m,r}c}{\B(\Hilbert[K])} &\leq\norm{b-\Phi_{m,t}(b)}{\B_t} \underbracket{\norm{\omega_{m,r}}{\B(\Hilbert[K])} }_{=1}\\
      &\qquad+ \norm{(\underbracket{P_{m,t}(x)}_{=\Phi_{m,t}(b)}-P_{m,t_0}(x))\omega_{m,r}}{\B(\Hilbert[K])}\\
      &\qquad+ \norm{\left[P_{m,t_0}(x),\omega_{m,r}\right]}{\B(\Hilbert[K])}\\
      &\qquad+ \frac{\frac{\varepsilon}{8(d_0+1)}}{1+\frac{\varepsilon}{8(d_0+1)}} \norm{\omega_{m,r}P_{m,t_0}(x)}{\B(\Hilbert[K])}\\
      &< \frac{\varepsilon}{16} + \frac{\varepsilon}{16} + \frac{\varepsilon}{16} + \frac{\varepsilon}{8} < \frac{\varepsilon}{2}\text{.}
    \end{align*}

    Conversely, let $b\in\mathfrak L_{t_0}$, and choose $x\in C_m$ such that
    \begin{equation*}
      \Phi_{m,t_0}(b)=P_{m,t_0}(x)\text.
    \end{equation*}
    We have
    \begin{align*}
      \Lip_t(P_{m,t}(x)) &\leq \Lip_{t_0}(P_{m,t_0}(x)) + \left| \Lip_t(P_{m,t}(x)) - \Lip_{t_0}(P_{m,t_0}(x))\right|\\
      &< 1+\varepsilon_m+\frac{\varepsilon}{16(d_0+1)}\\
      &< 1+\frac{\varepsilon}{8(d_0+1)}\text.
    \end{align*}
    Thus, setting:
    \begin{equation*}
      c \coloneqq \frac{1}{1+\frac{\varepsilon}{8(d_0+1)}} P_{m,t}(x)
    \end{equation*}
    we have $c \in \mathfrak L_t$. Since
    \begin{equation*}
      \norm{P_{m,t_0}(x)}{\B_{t_0}} \leq d_0(1+\varepsilon_m) < d_0 \left( 1+\frac{\varepsilon}{8(d_0+1)} \right),
    \end{equation*}
    we obtain
    \begin{align*}
      \norm{ c\omega_{m,r} - \omega_{m,r}b}{\B(\Hilbert[K])} & \leq \frac{1}{1+\frac{\varepsilon}{8(d_0+1)}}\norm{ (P_{m,t}(x)-P_{m,t_0}(x))\omega_{m,r}}{\B(\Hilbert[K])}\\
      &\qquad+ \frac{1}{1+\frac{\varepsilon}{8(d_0+1)}} \norm{ P_{m,t_0}(x)\omega_{m,r} - \omega_{m,r}P_{m,t_0}(x)}{\B(\Hilbert[K])}\\
      &\qquad+ \frac{\frac{\varepsilon}{8(d_0+1)}}{1+\frac{\varepsilon}{8(d_0+1)}} \norm{ \omega_{m,r}P_{m,t_0}(x) }{\B(\Hilbert[K])}\\
      &\qquad+ \norm{ \omega_{m,r}(P_{m,t_0}(x)-b)}{\B(\Hilbert[K])}\\
      &< \frac{\varepsilon}{16} +\frac{\varepsilon}{16}+\frac{\varepsilon}{8}+\frac{\varepsilon}{16} < \frac{\varepsilon}{2}\text.
    \end{align*}

    Adding the same real scalar to both matched elements does not change the bridge seminorm. Therefore, the two preceding estimates apply to the full Lipschitz unit balls, and
    \begin{equation*}
      \bridgereach{\gamma_{m,r}}{\Lip_t,\Lip_{t_0}} < \frac{\varepsilon}{2}\text.
    \end{equation*}
    Thus, the length of our bridge is no more than $\varepsilon>0$.

    Therefore, $(\B_t \oplus \B_{t_0}, \TLip_{t,\varepsilon}, \pi_t, \pi_{t_0})$, where
    \begin{equation*}
      \TLip_{t,\varepsilon} : (a,b) \in \domsa{\Lip_t}\oplus\domsa{\Lip_{t_0}} \mapsto \max\left\{ \Lip_t(a), \Lip_{t_0}(b), \frac{1}{\varepsilon}\norm{ a \omega_{m,r} - \omega_{m,r} b }{\B(\Hilbert[K])} \right\}
    \end{equation*}
    is a tunnel of extent at most $\varepsilon$ by \cite{Latremoliere13b}.
  \end{proof}

  If we add the assumption that there exist $C\geq 1$ and $D \geq 0$ such that, for all $t \in T$, the Lipschitz seminorm $\Lip_t$ is $(C,D)$-Leibniz, then we have proven that
  \begin{equation*}
    \lim_{t\to t_0} \dpropinquity{C,D}((\B_t,\Lip_t),(\B_{t_0},\Lip_{t_0})) = 0 \text.
  \end{equation*}

  We will however not require this assumption; instead we now prove an extension of the previous theorem to the spectral propinquity.


  \subsubsection{Spectral convergence}

  We now strengthen Theorem~\ref{finite-mode-convergence-thm} to the twisted spectral propinquity.

  \begin{hypothesis}\label{spectral-finite-mode-hyp}
    Fix $\kappa\in(0,1]$. We assume Hypothesis~\ref{finite-mode-hyp}. Moreover, we assume the following conditions.

    For every $t\in T$, we are given a $\kappa$-metric twisted spectral triple
    \begin{equation*}
      (\B_t,\Hilbert[K],\Dirac[\nabla]_t,\rho_t)
    \end{equation*}
    such that $\Lip_t$ is the Lipschitz seminorm for this triple.

    We furthermore assume:
    \begin{enumerate}
      \item There exists a dense subspace $\dom{\nabla}$ of $\Hilbert[K]$ such that $\dom{\Dirac[\nabla]_t} = \dom{\nabla}$ for all $t \in T$. For all $t \in T$ and $\xi \in \dom{\nabla}$, we set
      \begin{equation*}
        \CDN_t(\xi) \coloneqq \norm{\xi}{\Hilbert[K]} + \norm{\Dirac[\nabla]_t\xi}{\Hilbert[K]} \text.
      \end{equation*}

      \item For all $t \in T$, there exists $\delta_t\geq0$ such that, for all $\xi \in \dom{\nabla}$,
      \begin{equation*}
        \norm{(\Dirac[\nabla]_t-\Dirac[\nabla]_{t_0})\xi}{\Hilbert[K]} \leq \delta_t\CDN_{t_0}(\xi)
      \end{equation*}
      while $\lim_{\substack{t \to t_0 \\ t \in T}} \delta_t = 0$.

      \item For all $X \in \B(\Hilbert[K])$, we set
      \begin{equation*}
        \opnorm{X}{}{t_0} \coloneqq \sup \left\{ \norm{X\xi}{\Hilbert[K]} : \xi\in\dom{\nabla}\text{, }\CDN_{t_0}(\xi)\leq1 \right\}\text.
      \end{equation*}
      There exists a neighborhood $V$ of $t_0$ in $T$ such that
      \begin{equation*}
        \sup_{t \in V} \sup_{b\in\mathfrak L_t} \opnorm{\rho_t(b)-\rho_t(\Phi_{m,t}(b))}{}{t_0} \xrightarrow{m \to \infty} 0 \text.
      \end{equation*}

      \item We assume that for every $m\in\N$,
      \begin{equation*}
        \sup_{x\in C_m}\max \left\{ \opnorm{P_{m,t}(x)-P_{m,t_0}(x)}{}{t_0}\text{, }\opnorm{\rho_t(P_{m,t}(x)) - \rho_{t_0}(P_{m,t_0}(x))}{}{t_0} \right\}\xrightarrow{t \in T, t\to t_0} 0 \text.
      \end{equation*}
    \end{enumerate}
  \end{hypothesis}

  We first construct a metrical tunnel between our metrical C*-correspondences.

  \begin{lemma}[Finite-mode metrical tunnels]
    \label{finite-mode-joint-matching-lemma} For every $\varepsilon>0$, there exists a neighborhood $V_\varepsilon$ of $t_0$ such that, for all $t\in V_\varepsilon$, there exists a relaxed $(1,0,\kappa^{-1}+2,2)$-metrical tunnel from
    \begin{equation*}
      (\Hilbert[K], \CDN_t, \B_t,\Lip_t,\C,0)
    \end{equation*}
    to
    \begin{equation*}
      (\Hilbert[K], \CDN_{t_0}, \B_{t_0},\Lip_{t_0},\C,0)
    \end{equation*}
    with extent at most $\varepsilon$.
  \end{lemma}

  \begin{proof}
    We let
    \begin{equation*}
      d_0 \coloneqq \diam{\B_{t_0}}{\Lip_{t_0}}\text{, } \qquad d_\kappa \coloneqq d_0+\kappa^{-1}+1 \text.
    \end{equation*}

    Let $\varepsilon \in (0,1)$. By Hypothesis \ref{spectral-finite-mode-hyp}, there exists a neighborhood $V$ of $t_0$ in $T$ and $M_0 \in \N$ such that, if $m\geq M_0$, then
    \begin{equation*}
      \sup_{s\in V}\sup_{b\in\mathfrak L_s} \opnorm{\rho_s(b)-\rho_s(\Phi_{m,s}(b)) }{}{t_0} < \frac{\varepsilon}{16 d_\kappa}\text.
    \end{equation*}

    Our current proof is in essence a sequel for the proof of Theorem \ref{finite-mode-convergence-thm}, so we bring here our notations from that proof. We replace $\varepsilon$ in that proof with $\frac{d_0+1}{d_\kappa}\varepsilon$. We thus conclude that there exists $m, r \in \N$ and a neighborhood $T_\varepsilon$ of $t_0$ in $T$ such that $m \geq M_0$, and if $t \in T_\varepsilon$, and if we set
    \begin{multline*}
      \TLip_{t,\varepsilon}:(a,b)\in\dom{\Lip_t}\oplus\dom{\Lip_{t_0}} \mapsto \\
      \max\left\{\Lip_t(a),\Lip_{t_0}(b),\frac{d_\kappa}{(d_0+1)\varepsilon}\norm{a\omega_{m,r}-\omega_{m,r}b}{\B(\Hilbert[K])} \right\}
    \end{multline*}
    then $(\B_t\oplus\B_{t_0},\TLip_{t,\varepsilon},\pi_t,\pi_{t_0})$ is a tunnel of extent at most $\frac{d_0+1}{d_{\kappa}}\varepsilon < \varepsilon$. Note that we changed our notation for $\TLip$ (it would technically be $\TLip_{t,\frac{d_0+1}{d_\kappa}\varepsilon}$ using the notation of our previous proof, but we lighten it here). As a rule of thumb, estimates of the form $\frac{\varepsilon}{d_0+1}$ are now replaced with $\frac{\varepsilon}{d_\kappa}$.

    Now, by Hypothesis \ref{spectral-finite-mode-hyp}, there exists an open neighborhood $V'$ of $t_0$ in $T$ such that, for all $t \in V'$,
    \begin{equation*}
      \sup_{x \in C_m} \max\Bigl\{ \opnorm{P_{m,t}(x) - P_{m,t_0}(x)}{}{t_0}, \opnorm{\rho_t(P_{m,t}(x)) - \rho_{t_0}(P_{m,t_0}(x))}{}{t_0} \Bigr\} < \frac{\varepsilon}{16 d_\kappa} \text.
    \end{equation*}

    Last, there exists by Hypothesis \ref{spectral-finite-mode-hyp} an open neighborhood $V''$ of $t_0$ in $T$ such that if $t \in V''$, then $\delta_t < \varepsilon$.

    Henceforth, we set $T'_\varepsilon = T_\varepsilon\cap V \cap V' \cap V''$ (which is a neighborhood of $t_0$ in $T$). We fix $t \in T'_\varepsilon$.

    \medskip

    We now introduce a new tunnel from $(\B_t,\Lip_t)$ to $(\B_{t_0},\Lip_{t_0})$. We define:
    \begin{multline*}
      \SLip_{t,\varepsilon} : (b,c) \in\dom{\Lip_t}\oplus\dom{\Lip_{t_0}} \mapsto \\
      \max\left\{ \TLip_{t,\varepsilon}(b,c) \text{, }\varepsilon^{-1}\opnorm{b-c}{}{t_0} \text{, } \varepsilon^{-1}\opnorm{\rho_t(b)-\rho_{t_0}(c)}{}{t_0}\right\}\text.
    \end{multline*}
    We wish to first prove that indeed, $(\B_t\oplus \B_{t_0},\SLip_{t,\varepsilon},\pi_t,\pi_{t_0})$ is a tunnel. \medskip

    Let $b\in\mathfrak L_t$, and choose $x\in C_m$ such that
    \begin{equation*}
      P_{m,t}(x)=\Phi_{m,t}(b)\text.
    \end{equation*}
    Set
    \begin{equation*}
      p_0\coloneqq P_{m,t_0}(x) \qquad\text{and}\qquad c\coloneqq \frac{1}{1+\frac{\varepsilon}{8d_\kappa}}p_0\text.
    \end{equation*}
    By the proof of Theorem \ref{finite-mode-convergence-thm},
    \begin{equation*}
      c\in\mathfrak L_{t_0} \qquad\text{and}\qquad \TLip_{t,\varepsilon}(b,c)\leq1\text.
    \end{equation*}
    By construction, $\mu_{t_0}(p_0) = 0$ and $\Lip_{t_0}(p_0) \leq \left(1 + \frac{\varepsilon}{8 d_\kappa}\right)$, we conclude
    \begin{equation}\label{p0-eq}
      \opnorm{p_0}{}{t_0} \leq \norm{p_0}{\B(\Hilbert[K])} = \norm{p_0}{\B_{t_0}} < d_0\left( 1 + \frac{\varepsilon}{8 d_\kappa} \right)
    \end{equation}
    and therefore
    \begin{equation}\label{rhop0-eq}
      \begin{split}
        \opnorm{\rho_{t_0}(p_0)}{}{t_0} &\leq \norm{\rho_{t_0}(p_0)}{\B(\Hilbert[K])} \\
        &\leq \norm{p_0}{\B_{t_0}} + \kappa^{-1}\Lip_{t_0}(p_0) \\
        &< (d_0 + \kappa^{-1})\left(1 + \frac{\varepsilon}{8 d_\kappa}\right) \text.
      \end{split}
    \end{equation}

    Consequently, since
    \begin{equation*}
      b - c = (b - \Phi_{m,t}(b)) + (\underbracket{P_{m,t}(x)}_{=\Phi_{m,t}(b)} - \underbracket{P_{m,t_0}(x)}_{=p_0}) + (p_0 - c)
    \end{equation*}
    we conclude
    \begin{align*}
      \opnorm{b-c}{}{t_0} &\leq \opnorm{b-\Phi_{m,t}(b)}{}{t_0} + \opnorm{P_{m,t}(x) - P_{m,t_0}(x)}{}{t_0} + \opnorm{p_0 - c}{}{t_0} \\
      &< \frac{\varepsilon}{16 d_{\kappa}} + \frac{\varepsilon}{16 d_{\kappa}} + \left( 1 - \frac{1}{1 + \frac{\varepsilon}{8 d_\kappa}}\right) \opnorm{p_0}{}{t_0} \\
      &\leq \frac{\varepsilon}{8 d_{\kappa}} + \frac{\frac{\varepsilon}{8 d_\kappa}}{1+\frac{\varepsilon}{8 d_\kappa}} d_0 \left(1+ \frac{\varepsilon}{8 d_\kappa}\right) \quad \text{ by Exp. \eqref{p0-eq}} \\
      &\leq \frac{\varepsilon}{8 d_\kappa} + \frac{\varepsilon d_0}{8 d_\kappa} \\
      &\leq \frac{\varepsilon(d_0+1)}{8 d_\kappa} < \frac{\varepsilon}{8} \text.
    \end{align*}

    Similarly,
    \begin{align*}
      \opnorm{\rho_t(b)-\rho_{t_0}(c)}{}{t_0} &\leq \opnorm{\rho_t(b - \Phi_{m,t}(b))}{}{t_0} + \opnorm{\rho_t(\underbracket{P_{m,t}(x)}_{=\Phi_{m,t}(b)}) - \rho_{t_0}(\underbracket{P_{m,t_0}(x)}_{=p_0})}{}{t_0} \\
      &\qquad + \opnorm{\rho_{t_0}(p_0 - c)}{}{t_0} \\
      &< \frac{\varepsilon}{16 d_\kappa} + \frac{\varepsilon}{16 d_\kappa} + \frac{\frac{\varepsilon}{8 d_\kappa}}{1+\frac{\varepsilon}{8 d_\kappa}} \opnorm{\rho_{t_0}(p_0)}{}{t_0} \\
      &< \frac{\varepsilon}{8 d_\kappa} + \frac{\varepsilon(d_0+\kappa^{-1})}{8d_\kappa} \quad \text{ by Exp. \eqref{rhop0-eq},} \\
      &< \frac{\varepsilon}{8} \text.
    \end{align*}

    Therefore,
    \begin{equation*}
      \SLip(b,c) \leq 1 \text.
    \end{equation*}

    Therefore, by homogeneity and since $\SLip_{t,\varepsilon}(1,1) = 0$, and since $\SLip_{t,\varepsilon}(b,c) \geq \Lip_{t}(b)$ for all $(b,c)\in\dom{\SLip_{t,\varepsilon}}$, we conclude that $\Lip_t(b) = \inf\{ \SLip_{t,\varepsilon}(b,c) : c\in\dom{\Lip_{t_0}}\}$ for all $b \in \dom{\Lip_t}$.

    The same reasoning applies to show that $\Lip_{t_0}(c) = \inf\{ \SLip_{t,\varepsilon}(b,c) : b\in\dom{\Lip_{t}}\}$ for all $c \in \dom{\Lip_{t_0}}$.

    We now prove that $\SLip_{t,\varepsilon}$ is lower semi-continuous. The map
    \begin{equation*}
      (b,c)\longmapsto\opnorm{b-c}{}{t_0}
    \end{equation*}
    is norm-continuous since $2$-Lipschitz: for all $(b,c)\in\dom{\Lip_t}\oplus\dom{\Lip_{t_0}}$,
    \begin{align*}
      &\left|\opnorm{b-c}{}{t_0} - \opnorm{b'-c'}{}{t_0} \right|\\
      &\qquad\leq \opnorm{(b-b')-(c-c')}{}{t_0}\\
      &\qquad\leq \norm{b-b'}{\B_t} + \norm{c-c'}{\B_{t_0}}\text.
    \end{align*}

    Let now $((b_n,c_n))_{n\in\N}$ in $\dom{\Lip_t}\oplus\dom{\Lip_{t_0}}$ converge, in norm, to $(b,c)$, and assume that
    \begin{equation*}
      \liminf_{n\to\infty} \; \SLip_{t,\varepsilon}(b_n,c_n)<\infty\text.
    \end{equation*}
    Passing to a subsequence, we may assume that this limit inferior is a limit and that there exists $M>0$ such that
    \begin{equation*}
      \sup_{n\in\N}\left\{ \norm{b_n}{\B_{t}},\norm{c_n}{\B_{t_0}} ,\SLip_{t,\varepsilon}(b_n,c_n) \right\}\leq M \text.
    \end{equation*}
    In particular,
    \begin{equation*}
      \Lip_t(b_n)\leq M \qquad\text{and}\qquad \Lip_{t_0}(c_n)\leq M\text.
    \end{equation*}
    The sequences $\left(\rho_t(b_n)\right)_{n\in\N}$ and $\left(\rho_{t_0}(c_n)\right)_{n\in\N}$ are therefore bounded in operator norm in $\B(\Hilbert[K])$ by $(1+\kappa^{-1}) M$. Passing to a further subsequence, we may assume that all these sequences converge in the weak operator topology on $\B(\Hilbert[K])$. By Lemma \ref{twisted-commutator-closed-graph-lemma}, their limits are respectively $\rho_t(b)$ and $\rho_{t_0}(c)$. In particular,
    \begin{equation*}
      \rho_t(b_n)-\rho_{t_0}(c_n) \xrightarrow{n \to \infty} \rho_t(b)-\rho_{t_0}(c)
    \end{equation*}
    in the weak operator topology.

    For every $\xi\in\dom{\nabla}$ with $\CDN_{t_0}(\xi)\leq1$, the norm on $\Hilbert[K]$ is weakly lower semi-continuous, and therefore, since $(\rho_t(b_n)-\rho_{t_0}(c_n))_{n\in\N}$ converges to $\rho_t(b)-\rho_{t_0}(c)$ for the WOT, we conclude:
    \begin{equation*}
      \norm{(\rho_t(b)-\rho_{t_0}(c))\xi}{\Hilbert[K]} \leq \liminf_{n\to\infty} \norm{(\rho_t(b_n)-\rho_{t_0}(c_n))\xi}{\Hilbert[K]}\text.
    \end{equation*}
    Taking the supremum over the closed $\CDN_{t_0}$-unit ball gives
    \begin{equation*}
      \opnorm{\rho_t(b)-\rho_{t_0}(c)}{}{t_0} \leq \liminf_{n\to\infty} \opnorm{\rho_t(b_n)-\rho_{t_0}(c_n)}{}{t_0}\text.
    \end{equation*}
    Since $\TLip_{t,\varepsilon}$ is lower semi-continuous, we conclude that $\SLip_{t,\varepsilon}$ is lower semi-continuous.

    Now, $\SLip_{t,\varepsilon} \geq \TLip_{t,\varepsilon}$, so by \cite[Lemma 1.8]{Rieffel98a}, we conclude that $(\B_t\oplus\B_{t_0},\SLip_{t,\varepsilon})$ is a generalized {\qcms}.

    Therefore, altogether,
    \begin{equation*}
      \tau_{t,\varepsilon} \coloneqq (\B_t\oplus\B_{t_0}, \SLip_{t,\varepsilon}, \pi_t,\pi_{t_0})
    \end{equation*}
    is a generalized tunnel. Since $\TLip_{t,\varepsilon}\leq\SLip_{t,\varepsilon}$ (on their domain $\dom{\Lip_t}\oplus\dom{\Lip_{t_0}}$), we conclude by definition that $\Kantorovich{\TLip_{t,\varepsilon}} \geq \Kantorovich{\SLip_{t,\varepsilon}}$, and thus it follows, again by definition, that
    \begin{equation*}
      \tunnelextent{\tau_{t,\varepsilon}} \leq \tunnelextent{(\B_t\oplus\B_{t_0},\TLip_{t,\varepsilon}, \pi_t,\pi_{t_0})} \leq \varepsilon \text.
    \end{equation*}

    We now enhance our tunnel to a (generalized) metrical tunnel. Let
    \begin{equation*}
      \Hilbert[E] = \Hilbert[K]\oplus\Hilbert[K]
    \end{equation*}
    be the right Hilbert $\C^2$-module with coordinatewise right action, inner product
    \begin{equation*}
      \inner{(\xi,\eta)}{(\zeta,\omega)}{\C^2} = \left( \inner{\xi}{\zeta}{\Hilbert[K]}, \inner{\eta}{\omega}{\Hilbert[K]} \right)
    \end{equation*}
    and Hilbert-module norm
    \begin{equation*}
      \norm{(\xi,\eta)}{\Hilbert[E]} = \max\left\{ \norm{\xi}{\Hilbert[K]}, \norm{\eta}{\Hilbert[K]} \right\}\text.
    \end{equation*}
    For all $(\xi,\eta) \in \dom{\nabla}\oplus\dom{\nabla}$, we set
    \begin{equation*}
      \TDN_{t,\varepsilon}(\xi,\eta) = \max \left\{ \CDN_t(\xi), \CDN_{t_0}(\eta) , \frac{1}{\varepsilon}\norm{\xi-\eta}{\Hilbert[K]} \right\}\text.
    \end{equation*}

    Since, for all $\xi \in \dom{\nabla}$:
    \begin{equation*}
      \left|\CDN_t(\xi) - \CDN_{t_0}(\xi)\right| = \left|\norm{\Dirac{\nabla}_t\xi}{\Hilbert[K]} - \norm{\Dirac{\nabla}_{t_0}\xi}{\Hilbert[K]}\right| \leq \norm{(\Dirac{\nabla}_t - \Dirac{\nabla}_{t_0})\xi}{\Hilbert[K]} \leq \delta_t \CDN_{t_0}(\xi) \text,
    \end{equation*}
    we conclude:
    \begin{equation*}
      (1-\delta_t)\CDN_{t_0}(\xi) \leq \CDN_t(\xi) \leq (1+\delta_t)\CDN_{t_0}(\xi)\text.
    \end{equation*}
    Since $\delta_t\leq\varepsilon$ for $t\in T_\varepsilon'$, we obtain
    \begin{equation*}
      \TDN_{t,\varepsilon} \left( \xi, (1-\delta_t)\xi \right) = \CDN_t(\xi)
    \end{equation*}
    and
    \begin{equation*}
      \TDN_{t,\varepsilon}\left(\frac{1}{1+\delta_t}\eta, \eta \right) = \CDN_{t_0}(\eta)\text.
    \end{equation*}

    Thus, if we set $\Pi_t : (\xi,\eta) \in \Hilbert[E] \mapsto \xi$ and $\Pi_{t_0} : (\xi,\eta)\in\Hilbert[E] \mapsto \eta$, then
    \begin{equation*}
      \CDN_t(\xi) = \inf\left\{ \TDN_{t,\varepsilon}(\xi,\eta) : \eta\in\dom{\nabla} \right\} \text{ and } \CDN_{t_0}(\xi) = \inf\left\{ \TDN_{t,\varepsilon}(\eta,\xi) : \eta\in\dom{\nabla} \right\}
    \end{equation*}
    for all $\xi \in \dom{\nabla}$.

    Now, of course, $\dom{\TDN_{t,\varepsilon}} \coloneqq \dom{\nabla}\oplus\dom{\nabla}$ is dense in $\Hilbert[E]$. Moreover, by construction,
    \begin{multline*}
      \left\{ (\xi,\eta) \in \dom{\TDN_{t,\varepsilon}} : \TDN_{t,\varepsilon}(\xi,\eta)\leq 1 \right\} \subseteq \\
      \left\{ \xi \in \dom{\nabla} : \CDN_{t}(\xi)\leq 1 \right\} \times \left\{ \eta \in \dom{\nabla} : \CDN_{t_0}(\eta)\leq 1 \right\}
    \end{multline*}
    and since both $\CDN_t$ and $\CDN_{t_0}$ are $D$-norms, their closed unit balls are compact in $\Hilbert[K]$; therefore their product is compact in $\Hilbert[E]$ by Tychonoff, and thus the closed unit ball of $\TDN_{t,\varepsilon}$ is totally bounded.

    Moreover, $\CDN_t$ and $\CDN_{t_0}$ are both lower semicontinuous and thus so is $\TDN_{t,\varepsilon}$, so its closed unit ball is closed in $\Hilbert[E]$. in particular, since $\Hilbert[E]$ is complete, the closed unit ball of $\TDN_{t,\varepsilon}$ is complete and totally bounded, so it is compact.

    Now, $\Hilbert[E]$ is a $\C^2$-Hilbert module. For all $(z,w) \in \C^2$, we set:
    \begin{equation*}
      q_\varepsilon(z,w) = \frac{1}{\varepsilon}|z-w|\text.
    \end{equation*}
    This is a Leibniz Lipschitz seminorm since, for all $(z,w), (z',w') \in \C^2$,
    \begin{equation*}
      |zz'-ww'| \leq |z|\,|z'-w'|+|w'|\,|z-w|\text.
    \end{equation*}
    Of course,
    \begin{equation*}
      (\C^2, q, j_t:(z,w)\mapsto z, j_{t_0}:(z,w)\mapsto w)
    \end{equation*}
    is a tunnel, with extent at most $\varepsilon$.

    We now check that our prospective metrical tunnel has the expected Leibniz properties. First, for all $(\xi,\eta), (\zeta,\omega) \in \dom{\TDN_{t,\varepsilon}}$, we check:
    \begin{align*}
      q_\varepsilon(\inner{(\xi,\eta)}{(\zeta,\omega)}{\Hilbert[E]}) &=\frac{1}{\varepsilon}\left| \inner{\xi}{\zeta}{\Hilbert[K]} - \inner{\eta}{\omega}{\Hilbert[K]} \right| \\
      &\leq \frac{1}{\varepsilon}\norm{\xi-\eta}{\Hilbert[K]} \norm{\zeta}{\Hilbert[K]}\\
      &\qquad + \norm{\eta}{\Hilbert[K]} \frac{1}{\varepsilon}\norm{\zeta-\omega}{\Hilbert[K]}\\
      &\leq 2\TDN_{t,\varepsilon}(\xi,\eta) \TDN_{t,\varepsilon}(\zeta,\omega)\text.
    \end{align*}
    This proves the required inner Leibniz inequality, with constant $2$, for the real and imaginary parts of the inner product.

    We now prove the modular Leibniz property. Of course, $\B_t\oplus \B_{t_0}$ acts diagonally on $\Hilbert[E]$, i.e. if $(b,c) \in \B_t\oplus\B_{t_0}$ and $(\xi,\eta) \in \Hilbert[E]$, then $(b,c)(\xi,\eta) = (b\xi,c\eta)$. Note that indeed, $(\Hilbert[E], \B_t\oplus\B_{t_0}, \C^2)$ is a C*-correspondence.

    Let $(b,c) \in \dom{\TLip_{t,\varepsilon}}$. We compute:
    \begin{align*}
      \frac{1}{\varepsilon}\norm{b\xi-c\eta}{\Hilbert[K]} &\leq \norm{b}{\B_t}\frac{1}{\varepsilon}\norm{\xi-\eta}{\Hilbert[K]}\\
      &\qquad+ \frac{1}{\varepsilon}\opnorm{b-c}{}{t_0}\CDN_{t_0}(\eta)\text.
    \end{align*}

    Moreover, by \cite[Theorem ]{Latremoliere26b}, for all $b\in\domsa{\Lip_t}$ and $\xi\in\dom{\nabla}$,
    \begin{equation*}
      \CDN_t(b\xi)\leq\left(\norm{b}{\B_t}+\kappa^{-1}\Lip_t(b)\right)\CDN_t(\xi)\text.
    \end{equation*}
    The analogous estimate holds at $t_0$. So altogether, for all $(\xi,\eta)\in\dom{\TDN_{t,\varepsilon}}$,
    \begin{equation*}
      \TDN_{t,\varepsilon}(b\xi,c\eta)\leq \left(\norm{(b,c)}{\B_t\oplus\B_{t_0}} + \frac{1}{\kappa}\TLip_{t,\varepsilon}(b,c)\right)\TDN_{t,\varepsilon}(\xi,\eta) \text.
    \end{equation*}
    In particular,
    \begin{equation*}
      \begin{aligned}
        &\TDN_{t,\varepsilon}(b\xi,c\eta)\\
        &\quad\leq(\frac{1}{\kappa}+2)\max\left\{\norm{(b,c)}{\B_t\oplus\B_{t_0}},\SLip_{t,\varepsilon}(b,c)\right\}\TDN_{t,\varepsilon}(\xi,\eta)\text.
      \end{aligned}
    \end{equation*}

    The coordinate projections from $\C^2$ onto $C$ preserve the inner products, intertwine the left and right actions of our C*-correspondences, and are obviously quantum isometries.. We thus have constructed a relaxed $(1,0,\kappa^{-1}+1,2)$-metrical tunnel:
    \begin{equation*}
      \left( \Hilbert[E], \TDN_{t,\varepsilon}, \B_t\oplus\B_{t_0}, \TLip_{t,\varepsilon}, \C^2, q_\varepsilon \right)
    \end{equation*}
    with quantum isometries $(\Pi_t,\pi_t,j_t)$ onto $(\Hilbert[K],\CDN_t,\B_t,\Lip_t,\C,0)$ and $(\Pi_{t_0},\pi_{t_0}, j_{t_0})$ onto $(\Hilbert[K],\CDN_{t_0},\B_{t_0},\Lip_{t_0},\C,0)$. Its extent is bounded by the extent of $(\B_t\oplus\B_{t_0},\TLip_{t,\varepsilon}, \pi_t, \pi_{t_0})$ and the extent of $(\C^2,q_\varepsilon,j_t,j_{t_0})$, both of which are less or equal to $\varepsilon$. This concludes our proof.
  \end{proof}

  We are now ready to conclude our general convergence result.
  \begin{theorem}[A spectral convergence test]
    \label{finite-mode-spectral-convergence-thm} Under Hypothesis~\ref{spectral-finite-mode-hyp},
    \begin{equation*}
      \lim_{\substack{t \to t_0 \\ t \in T}} \spectralpropinquity{\kappa}\left((\B_t,\Hilbert[K],\Dirac[\nabla]_t,\rho_t),(\B_{t_0},\Hilbert[K],\Dirac[\nabla]_{t_0},\rho_{t_0})\right) = 0 \text.
    \end{equation*}
  \end{theorem}

  \begin{proof}
    Let $\varepsilon>0$, and fix
    \begin{equation*}
      0<\varepsilon_0<\min\left\{1,\frac{\varepsilon}{2}\right\}\text.
    \end{equation*}
    By Lemma~\ref{finite-mode-joint-matching-lemma} and since $\lim_{t\to t_0} \delta_t = 0$, there exists a neighborhood $T_\varepsilon$ of $t_0$ in $T$ such that, for all $t\in T_\varepsilon$, the metrical tunnel constructed in Lemma~\ref{finite-mode-joint-matching-lemma} has extent at most $\varepsilon_0$, and
    \begin{equation*}
      \delta_t\left(1+\frac{1}{\varepsilon}\right) < \frac{\varepsilon_0}{4}\text.
    \end{equation*}
    Fix $t\in T_\varepsilon$. In particular, note that $\delta_t \leq \varepsilon_0<\varepsilon$.

    By \cite{Kato}, for $\xi\in\dom{\nabla}$ and $s\in\R$,
    \begin{equation*}
      (e^{is\Dirac[\nabla]_t} - e^{is\Dirac[\nabla]_{t_0}})\xi = i\int_0^s e^{i(s-x)\Dirac[\nabla]_t}(\Dirac[\nabla]_t-\Dirac[\nabla]_{t_0})e^{ix\Dirac[\nabla]_{t_0}}\xi \,\mathrm dx\text.
    \end{equation*}
    Since the flow of $\Dirac[\nabla]_{t_0}$ preserves $\CDN_{t_0}$, we have
    \begin{align*}
      \norm{(e^{is\Dirac[\nabla]_t} - e^{is\Dirac[\nabla]_{t_0}})\xi}{\Hilbert[K]} &\leq \int_{s^-}^{s^+}\norm{(\Dirac[\nabla]_t-\Dirac[\nabla]_{t_0})e^{ix\Dirac[\nabla]_{t_0}}\xi}{\Hilbert[K]}\,\mathrm dx\\
      &\leq |s|\delta_t\CDN_{t_0}(\xi)\text,
    \end{align*}
    where $s^-=\min\{0,s\}$ and $s^+=\max\{0,s\}$.

    If $\CDN_t(\xi)\leq1$, set
    \begin{equation*}
      \eta=(1-\delta_t)\xi\text.
    \end{equation*}
    Then $\CDN_{t_0}(\eta)\leq 1$, and
    \begin{equation*}
      e^{is\Dirac[\nabla]_t}\xi - e^{is\Dirac[\nabla]_{t_0}}\eta = \delta_t e^{is\Dirac[\nabla]_t}\xi+(1-\delta_t)\left(e^{is\Dirac[\nabla]_t} - e^{is\Dirac[\nabla]_{t_0}}\right)\xi\text.
    \end{equation*}
    Consequently,
    \begin{equation*}
      \norm{e^{is\Dirac[\nabla]_t}\xi - e^{is\Dirac[\nabla]_{t_0}}\eta}{\Hilbert[K]} \leq \delta_t(1+|s|)\text.
    \end{equation*}
    Conversely, if $\CDN_{t_0}(\eta)\leq 1$, then
    \begin{equation*}
      \xi=\frac{1}{1+\delta_t}\eta
    \end{equation*}
    belongs to the $\CDN_t$-unit ball and
    \begin{equation*}
      e^{is\Dirac[\nabla]_t}\xi - e^{is\Dirac[\nabla]_{t_0}}\eta = \frac{1}{1+\delta_t}\left(e^{is\Dirac[\nabla]_t}-e^{is\Dirac[\nabla]_{t_0}}\right)\eta-\frac{\delta_t}{1+\delta_t}e^{is\Dirac[\nabla]_{t_0}}\eta\text.
    \end{equation*}
    Thus the same estimate holds. Therefore, for every $(\zeta,\omega)$ in the $\TDN_{t,\varepsilon_0}$-unit ball, for all $s\in\R$,
    \begin{align*}
      &\left|\inner{e^{is\Dirac[\nabla]_t}\xi}{\zeta}{\Hilbert[K]}-\inner{e^{is\Dirac[\nabla]_{t_0}}\eta}{\omega}{\Hilbert[K]}\right|\\
      &\quad\leq \left|\inner{e^{is\Dirac[\nabla]_t}\xi}{\zeta}{\Hilbert[K]}-\inner{e^{is\Dirac[\nabla]_{t_0}}\eta}{\zeta}{\Hilbert[K]}\right|+\left|\inner{e^{is\Dirac[\nabla]_{t_0}}\eta}{\omega-\zeta}{\Hilbert[K]}\right|\\
      &\quad\leq \delta_t(1+|s|)+\varepsilon_0\text.
    \end{align*}
    The same estimate holds in both directions. Hence, for all $s\in\left[0,\frac{1}{\varepsilon}\right]$,
    \begin{equation*}
      \delta_t(1+|s|)+\varepsilon_0 \leq \delta_t\left(1+\frac{1}{\varepsilon}\right)+\varepsilon_0 < \frac{5\varepsilon_0}{4} < \frac{5\varepsilon}{8} < \varepsilon\text.
    \end{equation*}
    Thus the separation of the two Dirac flows on $\left[0,\frac{1}{\varepsilon}\right]$ is less than $\varepsilon$.

    It remains to check the error on the twists. Set
    \begin{equation*}
      K_\kappa \coloneqq 1+\kappa^{-1}\text.
    \end{equation*}
    Let
    \begin{equation*}
      (b,c) \in \domsa{\Lip_t}\oplus\domsa{\Lip_{t_0}}
    \end{equation*}
    such that
    \begin{equation*}
      \max\left\{\norm{(b,c)}{\B_t\oplus\B_{t_0}},\SLip_{t,\varepsilon_0}(b,c)\right\} \leq 1\text.
    \end{equation*}
    Then
    \begin{equation*}
      \norm{\rho_t(b)}{\B(\Hilbert[K])} \leq K_\kappa \qquad\text{and}\qquad \norm{\rho_{t_0}(c)}{\B(\Hilbert[K])} \leq K_\kappa\text.
    \end{equation*}
    For the matched vectors $\xi,\eta$ above and $(\zeta,\omega)$ in the $\TDN_{t,\varepsilon_0}$-unit ball, moving the adjoints across the inner products gives
    \begin{align*}
      &\left|\inner{\rho_t(b)^*\xi}{\zeta}{\Hilbert[K]}-\inner{\rho_{t_0}(c)^*\eta}{\omega}{\Hilbert[K]}\right|\\
      &\quad\leq \norm{\xi-\eta}{\Hilbert[K]}\norm{\rho_t(b)\zeta}{\Hilbert[K]}+\norm{\eta}{\Hilbert[K]}\norm{\rho_t(b)(\zeta-\omega)}{\Hilbert[K]}\\
      &\qquad+\norm{\eta}{\Hilbert[K]}\norm{(\rho_t(b)-\rho_{t_0}(c))\omega}{\Hilbert[K]}\\
      &\quad\leq K_\kappa\delta_t +(K_\kappa+1)\varepsilon_0\text.
    \end{align*}
    The same estimate holds in the reverse direction. After division by the factor
    \begin{equation*}
      2(\kappa^{-1}+2)=2(K_\kappa+1)
    \end{equation*}
    in the twisted spectral magnitude, we conclude that the normalized twist separation is at most
    \begin{equation*}
      \frac{K_\kappa\delta_t+(K_\kappa+1)\varepsilon_0}{2(K_\kappa+1)} \leq \frac{\delta_t+\varepsilon_0}{2} \leq \varepsilon_0 < \frac{\varepsilon}{2}\text.
    \end{equation*}

    The metrical tunnel has extent less than $\varepsilon$, the separation of the Dirac flows on $\left[0,\frac{1}{\varepsilon}\right]$ is less than $\varepsilon$, and the normalized separation of the twists is less than $\frac{\varepsilon}{2}$. Its $\varepsilon$-magnitude is therefore at most $\varepsilon$. Thus, for all $t\in T_\varepsilon$,
    \begin{equation*}
      \spectralpropinquity{\kappa}\left((\B_t,\Hilbert[K],\Dirac[\nabla]_t,\rho_t),(\B_{t_0},\Hilbert[K],\Dirac[\nabla]_{t_0},\rho_{t_0})\right) \leq \varepsilon \text.
    \end{equation*}
    Since $\varepsilon>0$ was arbitrary, our proof is complete.
  \end{proof}

  \subsection{Application of our strategy to crossed products}

  We now apply the convergence criteria from the previous subsection to the twisted spectral triples constructed in this paper. We first fix the families of crossed products and twisted spectral triples under consideration.

  \begin{hypothesis}\label{crossed-product-family-hyp}
    Let $T$ be a topological space, let $t_0\in T$, and let $d\in\N\setminus\{0\}$. We fix a Hilbert space $\Hilbert$. For every $t\in T$, let
    \begin{equation*}
      (\A_t,\Hilbert,\Dirac_t)
    \end{equation*}
    be a metric spectral triple, and let
    \begin{equation*}
      \alpha_t:\Z^d\longrightarrow\operatorname{Aut}(\A_t)
    \end{equation*}
    be an action of $\Z^d$ on $\A_t$.

    We fix $\varepsilon>0$ and $\kappa\in(0,1]$, and assume that there exists one admissible weight
    \begin{equation*}
      w:\Z^d\longrightarrow(0,1]
    \end{equation*}
    which is appropriate for every action $\alpha_t$ and the Lipschitz seminorm induced by $\Dirac_t$.

    For every $t\in T$, set
    \begin{equation*}
      \B_t\coloneqq\A_t\rtimes_{\alpha_t}\Z^d
    \end{equation*}
    and represent $\B_t$ on the common Hilbert space
    \begin{equation*}
      \Hilbert[K]\coloneqq\ell^2(\Z^d,\Hilbert)\otimes\C^q
    \end{equation*}
    by the regular covariant representation associated with $\alpha_t$. For every $z\in\Z^d$, we denote the canonical unitary in $\B_t$ by $U_t^z$. Theorem~\ref{metric-thm} provides the $\kappa$-metric twisted spectral triple
    \begin{equation*}
      (\B_t,\Hilbert[K],\Dirac[\nabla]_{w,\varepsilon,t},\rho_t)
    \end{equation*}
    and its Lipschitz seminorm $\Lip_t$. We write
    \begin{equation*}
      \Dirac[\nabla]_t\coloneqq\Dirac[\nabla]_{w,\varepsilon,t}
    \end{equation*}
    for every $t\in T$.

    For every $t\in T$, fix a state $\mu_t\in\StateSpace(\A_t)$, set
    \begin{equation*}
      \nu_t\coloneqq\mu_t\circ\mathcal{F}_0,
    \end{equation*}
    and let
    \begin{equation*}
      \mathfrak L_t\coloneqq\left\{b\in\domsa{\Lip_t}:\Lip_t(b)\leq1,\ \nu_t(b)=0\right\}\text.
    \end{equation*}
  \end{hypothesis}

  We now collect the assumptions which identify this family with the abstract finite-mode framework.

  \begin{hypothesis}\label{crossed-product-convergence-hyp}
    We assume Hypothesis~\ref{crossed-product-family-hyp} and the following conditions.
    \begin{enumerate}
      \item For every $N\in\N$ and $t\in T$, there is a bounded unital linear map
      \begin{equation*}
        \Phi_{N,t}:\B_t\longrightarrow\B_t
      \end{equation*}
      which preserves adjoints and satisfies
      \begin{equation*}
        \nu_t\circ\Phi_{N,t}=\nu_t\text.
      \end{equation*}
      There exists a sequence $(\varepsilon_N)_{N\in\N}$ of strictly positive numbers converging to $0$ such that, for every $N\in\N$, $t\in T$, and $b\in\domsa{\Lip_t}$,
      \begin{equation*}
        \norm{b-\Phi_{N,t}(b)}{\B_t}\leq\varepsilon_N\Lip_t(b)
      \end{equation*}
      and
      \begin{equation*}
        \Lip_t(\Phi_{N,t}(b))\leq(1+\varepsilon_N)\Lip_t(b)\text.
      \end{equation*}

      \item For every $N\in\N$, there are a compact Hausdorff space $C_N$ and, for every $t\in T$, a norm-continuous map
      \begin{equation*}
        P_{N,t}:C_N\longrightarrow\domsa{\Lip_t}
      \end{equation*}
      such that
      \begin{equation*}
        \Phi_{N,t}(\mathfrak L_t)\subseteq P_{N,t}(C_N)
      \end{equation*}
      and
      \begin{equation*}
        \nu_t(P_{N,t}(x))=0
      \end{equation*}
      for every $x\in C_N$. We assume that
      \begin{equation*}
        \lim_{\substack{t\to t_0\\t\in T}} \sup_{x\in C_N} \left| \Lip_t(P_{N,t}(x)) - \Lip_{t_0}(P_{N,t_0}(x)) \right| =0
      \end{equation*}
      and, for every $\xi\in\Hilbert[K]$,
      \begin{equation*}
        \lim_{\substack{t\to t_0\\t\in T}} \sup_{x\in C_N} \norm{ (P_{N,t}(x)-P_{N,t_0}(x))\xi }{\Hilbert[K]} =0\text.
      \end{equation*}

      \item For every $N\in\N$, there exist a neighborhood $V_N$ of $t_0$ in $T$ and positive finite-rank contractions
      \begin{equation*}
        \Omega_{N,r}\in\B(\Hilbert[K])
      \end{equation*}
      for $r\in\N$, such that
      \begin{equation*}
        \lim_{r\to\infty} \sup_{x\in C_N} \norm{[P_{N,t_0}(x),\Omega_{N,r}]}{\B(\Hilbert[K])} =0\text.
      \end{equation*}
      For every $r\in\N$ and $t\in V_N$, we assume that
      \begin{equation*}
        \gamma_{N,r}\coloneqq(\B(\Hilbert[K]),\Omega_{N,r})
      \end{equation*}
      is a bridge from $\B_t$ to $\B_{t_0}$, and that
      \begin{equation*}
        \lim_{r\to\infty} \limsup_{\substack{t\to t_0\\t\in V_N}} \bridgeheight{\gamma_{N,r}|\Lip_t,\Lip_{t_0}}{} =0\text.
      \end{equation*}

      \item The operators $\Dirac_t$, for $t\in T$, have a common domain. There are numbers $\delta_t\geq0$, for $t\in T$, such that
      \begin{equation*}
        \lim_{\substack{t\to t_0\\t\in T}}\delta_t=0
      \end{equation*}
      and, for every $t\in T$ and $\xi\in\dom{\Dirac_{t_0}}$,
      \begin{equation*}
        \norm{(\Dirac_t-\Dirac_{t_0})\xi}{\Hilbert} \leq \delta_t \left( \norm{\xi}{\Hilbert} + \norm{\Dirac_{t_0}\xi}{\Hilbert} \right)\text.
      \end{equation*}

      \item For every $t\in T$ and $\xi\in\dom{\Dirac[\nabla]_t}$, set
      \begin{equation*}
        \CDN_t(\xi) \coloneqq \norm{\xi}{\Hilbert[K]} + \norm{\Dirac[\nabla]_t\xi}{\Hilbert[K]}\text.
      \end{equation*}
      For every $X\in\B(\Hilbert[K])$, set
      \begin{equation*}
        \opnorm{X}{}{t_0} \coloneqq \sup \left\{ \norm{X\xi}{\Hilbert[K]}: \xi\in\dom{\Dirac[\nabla]_{t_0}},\
 \CDN_{t_0}(\xi)\leq1 \right\}\text.
      \end{equation*}
      There exists a neighborhood $V$ of $t_0$ in $T$ such that
      \begin{equation*}
        \lim_{N\to\infty} \sup_{t\in V} \sup_{b\in\mathfrak L_t} \opnorm{ \rho_t(b)-\rho_t(\Phi_{N,t}(b)) }{}{t_0} =0\text.
      \end{equation*}
      Moreover, for every $N\in\N$,
      \begin{equation*}
        \lim_{\substack{t\to t_0\\t\in T}} \sup_{x\in C_N} \max \left\{ \opnorm{P_{N,t}(x)-P_{N,t_0}(x)}{}{t_0}, \opnorm{ \rho_t(P_{N,t}(x)) - \rho_{t_0}(P_{N,t_0}(x)) }{}{t_0} \right\} =0\text.
      \end{equation*}
    \end{enumerate}
  \end{hypothesis}

  \begin{theorem}[Finite-mode convergence for crossed products]
    \label{crossed-product-finite-mode-convergence-thm} Under Hypothesis~\ref{crossed-product-convergence-hyp},
    \begin{equation*}
      \spectralpropinquity{\kappa} \left( (\B_t,\Hilbert[K],\Dirac[\nabla]_t,\rho_t), (\B_{t_0},\Hilbert[K],\Dirac[\nabla]_{t_0},\rho_{t_0}) \right) \longrightarrow0
    \end{equation*}
    as $t\longrightarrow t_0$.
  \end{theorem}

  \begin{proof}
    By Theorem~\ref{metric-thm}, $(\B_t,\Lip_t)$ is a generalized {\qcms} for every $t\in T$. Hypothesis~\ref{crossed-product-family-hyp} provides the common faithful representations of the algebras $\B_t$ on $\Hilbert[K]$ and the states $\nu_t$. Items 1--3 of Hypothesis~\ref{crossed-product-convergence-hyp} provide the remaining data and estimates in Hypothesis~\ref{finite-mode-hyp}. Hence Theorem~\ref{finite-mode-convergence-thm} applies.

    It remains to verify the common-domain and relative-bound conditions in Hypothesis~\ref{spectral-finite-mode-hyp}. We may restrict $T$ to a neighborhood of $t_0$ on which $\delta_t<1$. Let $\Dirac_{w,t}$ denote the weighted diagonal lift of $\Dirac_t$ to $\ell^2(\Z^d,\Hilbert)$. If $\xi\in\dom{\Dirac_{w,t_0}}$, then Item 4 of Hypothesis~\ref{crossed-product-convergence-hyp}, applied to every component of $\xi$, gives
    \begin{equation*}
      \left( \sum_{n\in\Z^d} w(n)^2 \norm{\Dirac_t\xi_n}{\Hilbert}^2 \right)^{\frac12} \leq (1+\delta_t)\norm{\Dirac_{w,t_0}\xi}{\Hilbert[S]} + \delta_t\norm{\xi}{\Hilbert[S]}
    \end{equation*}
    and hence $\xi\in\dom{\Dirac_{w,t}}$. Conversely, let $\xi\in\dom{\Dirac_{w,t}}$. For every finite subset $E\subseteq\Z^d$, define $\xi^E\in\ell^2(\Z^d,\Hilbert)$ by
    \begin{equation*}
      \xi_n^E \coloneqq
      \begin{cases}
        \xi_n & \text{if $n\in E$,}\\
        0 & \text{if $n\notin E$.}
      \end{cases}
    \end{equation*}
    Applying Item 4 of Hypothesis~\ref{crossed-product-convergence-hyp} to the finitely many nonzero components of $\xi^E$ gives
    \begin{equation*}
      (1-\delta_t)\norm{\Dirac_{w,t_0}\xi^E}{\Hilbert[S]} \leq \norm{\Dirac_{w,t}\xi^E}{\Hilbert[S]} + \delta_t\norm{\xi^E}{\Hilbert[S]}\text.
    \end{equation*}
    Since $\delta_t<1$, taking the supremum over all finite subsets $E\subseteq\Z^d$ shows that $\xi\in\dom{\Dirac_{w,t_0}}$ and that
    \begin{equation*}
      (1-\delta_t)\norm{\Dirac_{w,t_0}\xi}{\Hilbert[S]} \leq \norm{\Dirac_{w,t}\xi}{\Hilbert[S]} + \delta_t\norm{\xi}{\Hilbert[S]}\text.
    \end{equation*}
    Consequently,
    \begin{equation*}
      \dom{\Dirac_{w,t}} = \dom{\Dirac_{w,t_0}}
    \end{equation*}
    for all $t$ in this neighborhood.

    The group-coordinate operators in $\Dirac[\nabla]_t$ are independent of $t$, and the formula for $\Dirac[\nabla]_t^2$ gives
    \begin{equation*}
      \dom{\Dirac[\nabla]_t} = \left( \dom{\Dirac_{w,t}} \cap \bigcap_{j=1}^d\dom{\partial_j} \right) \otimes\C^q\text.
    \end{equation*}
    Thus, the operators $\Dirac[\nabla]_t$ have a common domain near $t_0$.

    Let $\xi=(\xi_n)_{n\in\Z^d}$ belong to this common domain. Since $w$ takes its values in $(0,1]$, we have
    \begin{align*}
      \norm{ (\Dirac[\nabla]_t-\Dirac[\nabla]_{t_0})\xi }{\Hilbert[K]} &= \left( \sum_{n\in\Z^d} w(n)^2 \norm{(\Dirac_t-\Dirac_{t_0})\xi_n}{\Hilbert\otimes\C^q}^2 \right)^{\frac12}\\
      &\leq \delta_t \left( \norm{(w(n)\xi_n)_{n\in\Z^d}}{\Hilbert[K]} + \norm{\Dirac_{w,t_0}\xi}{\Hilbert[K]} \right)\\
      &\leq \delta_t \left( \norm{\xi}{\Hilbert[K]} + \norm{\Dirac[\nabla]_{t_0}\xi}{\Hilbert[K]} \right)\\
      &= \delta_t\CDN_{t_0}(\xi)\text.
    \end{align*}
    Item 5 of Hypothesis~\ref{crossed-product-convergence-hyp} gives the remaining finite-mode conditions for the twists and their representations in Hypothesis~\ref{spectral-finite-mode-hyp}. Therefore, Theorem~\ref{finite-mode-spectral-convergence-thm} proves the desired convergence.
  \end{proof}

\subsection{Quantum tori with a fixed Anosov automorphism}

\subsubsection{The Quantum Torus}

For every $\theta\in\R$, the quantum torus $\A_\theta$ of deformation parameter $\theta$ is the universal unital $C^\ast$-algebra generated by the Weyl unitaries $(W_\theta^z)_{z\in\Z^2}$, subject to the relations
\begin{equation*}
  W_\theta^z W_\theta^m=\exp(\pi i\theta\det(z,m))W_\theta^{m+z},\qquad (W_\theta^{z})^*=W_\theta^{-z}
\end{equation*}
for all $z,m\in\Z^2$, where $W_\theta^0=1$.

We identify $\A_\theta$ with the C*-algebra of operators on $\Hilbert\coloneqq \ell^2(\Z^2)\otimes\C^2$, by setting
\begin{equation*}
  W_\theta^z(\phi_m\otimes v)=\exp(\pi i\theta\det(z,m))\phi_{z+m}\otimes v
\end{equation*}
for $z,m\in\Z^2$ and $v\in\C^2$, where $(\phi_p)_{p\in\Z^2}$ is the canonical orthonormal basis, and by extending $W_\theta^z$ by linearity and continuity to a unitary on $\Hilbert$. It is a routine matter to check that indeed, the desired relations are satisfied.

For $x\coloneqq(x_1,x_2) \in \T^2$, and for all $z\coloneqq(z_1,z_2) \in \Z^2$, we set
\begin{equation*}
  \inner{z}{x}{} = x_1^{z_1} x_2^{z_2} \text.
\end{equation*}
We then define, for all $\xi \in \ell^2(\Z^2)\otimes\C^2$:
\begin{equation*}
  u^x \xi : z \in \Z^2 \mapsto \inner{z}{x}{} \xi(z) \text.
\end{equation*}
By construction, $x\in\T^2\mapsto u^x$ is a strongly continuous action of $\T^2$ on $\Hilbert$ by unitaries. For all $b \in \B(\Hilbert)$ and $x \in \T^2$, we set
\begin{equation*}
  \zeta^x(b) = u^{x} b (u^x)^* \text.
\end{equation*}
In particular, for all $z \in \Z^2$,
\begin{equation*}
  \zeta^x(W_\theta^z)= \inner{z}{x}{} W_\theta^z \text.
\end{equation*}
The restriction of $\zeta$ to $\A_\theta$ is a strongly continuous action of $\T^2$, called the dual action on $\A_\theta$. In particular,
\begin{equation*}
  \operatorname{tr}_\theta : a \in \A_\theta \mapsto \int_{\T^2} \zeta^x(a) \, d\mu(x)
\end{equation*}
for $\mu$ the Haar probability measure on $\T^2$, is a tracial state on $\A_\theta$ (identifying $\C 1 \subseteq \A_\theta$ with $\C$).

\medskip

The quantum tori all come equipped with a natural metric spectral triple. We however introduce a perturbation of this basic spectral triple for this work. We denote the space of all definite positive real symmetric $2\times 2$ matrices by $\mathscr{S}_+$.

We now set:
\begin{equation*}
  \sigma_1 \coloneqq
  \begin{pmatrix}
    0 & 1 \\
    1 & 0
  \end{pmatrix}
  \qquad\text{ and }\qquad \sigma_2 \coloneqq
  \begin{pmatrix}
    0 & -i \\
    i & 0
  \end{pmatrix}
  \text.
\end{equation*}

For any $x = (x_1,x_2) \in \R^2$, we define
\begin{equation*}
  \sigma(x) \coloneqq x_1 \sigma_1 + x_2 \sigma_2 \text.
\end{equation*}

The standard Euclidean norm of $\R^2$ is denoted by $\norm{\cdot}{\R^2}$, and the standard hermitian norm on $\C^2$ is denoted by $\norm{\cdot}{\C^2}$. For every $G \in \mathscr{S}_+$, we define:
\begin{equation*}
  \dom{\Dirac_G} \coloneqq \left\{ \xi=(\xi_z)_{z\in\Z^2} \in \Hilbert : \sum_{z \in \Z^2} \norm{Gz}{\R^2}^2 \norm{\xi_z}{\C^2}^2 < \infty \right\} \text,
\end{equation*}
and, for all $\xi\coloneqq(\xi_z)_{z\in\Z^2} \in \dom{\Dirac_G}$,
\begin{equation*}
  \Dirac_G\xi : z \in \Z^2 \mapsto \sigma(G z) \xi(z)\text.
\end{equation*}
For all $z \in \Z^2$, we easily compute:
\begin{equation*}
  [\Dirac_G, W_\theta^z] = W_\theta^z(1\otimes \sigma(G z)) \text.
\end{equation*}
In particular, for all $z \in \Z^2$, the unitary $W_\theta^z$ is in the algebra
\begin{equation*}
  \A_{\theta,G} \coloneqq \left\{ a \in \A_\theta : a\dom{\Dirac_G}\subseteq\dom{\Dirac_G} \text, [\Dirac_G,a] \text{ bounded } \right\}
\end{equation*}
which is therefore dense in $\A_\theta$. For all $a \in \A_{\theta,G}$, we set
\begin{equation*}
  \Lip_{\theta,G}(a)\coloneqq\norm{[\Dirac_G,a]}{\B(\Hilbert)}\text.
\end{equation*}

\begin{lemma}
  For every $\theta\in\R$ and $G\in\mathscr{S}_+$, the triple $(\A_\theta,\Hilbert,\Dirac_G)$ is a metric spectral triple; moreover $\dom{\Dirac_G} = \dom{\Dirac_{I_2}}$ and $\A_{\theta,G} = \A_{\theta,I_2}$.
\end{lemma}

\begin{proof}
  Since $\sigma(x)^*=\sigma(x)$ and $\sigma(x)^2=\norm{x}{\R^2}^2 1$ for all $x\in\R^2$, the operator $\Dirac_G$ is densely defined and symmetric. Set
  \begin{equation*}
    K_G\eta\coloneqq\left(\frac{\eta_z}{1+\norm{Gz}{\R^2}^2}\right)_{z\in\Z^2} \qquad(\eta\in\Hilbert).
  \end{equation*}
  Then $K_G\Hilbert\subseteq\dom{\Dirac_G^2}$ and
  \begin{equation*}
    (1+\Dirac_G^2)K_G=1,\qquad K_G(1+\Dirac_G^2)\xi=\xi \quad(\xi\in\dom{\Dirac_G^2}).
  \end{equation*}
  Moreover,
  \begin{equation*}
    \norm{Gz}{\R^2}\geq\norm{G^{-1}}{M_2(\R)}^{-1}\norm{z}{\R^2}, \qquad \lim_{\norm{z}{\R^2}\to\infty}\frac{1}{1+\norm{Gz}{\R^2}^2}=0,
  \end{equation*}
  so $K_G$ is compact. For $\eta\in\Hilbert$, the vectors
  \begin{equation*}
    \xi^\pm\coloneqq(\Dirac_G\mp i)K_G\eta =\left(\frac{\sigma(Gz)\mp i1_{\C^2}}{1+\norm{Gz}{\R^2}^2}\eta_z\right)_{z\in\Z^2}
  \end{equation*}
  satisfy
  \begin{equation*}
    \sum_{z\in\Z^2}\bigl(1+\norm{Gz}{\R^2}^2\bigr)\norm{\xi_z^\pm}{\C^2}^2 =\norm{\eta}{\Hilbert}^2, \qquad (\Dirac_G\pm i)\xi^\pm=\eta.
  \end{equation*}
  Thus $\Dirac_G\pm i$ are surjective, and $\Dirac_G$ is self-adjoint with compact resolvent. Furthermore,
  \begin{equation*}
    [\Dirac_G,W_\theta^z]=W_\theta^z(1\otimes\sigma(Gz)) \qquad(z\in\Z^2),
  \end{equation*}
  so $\A_{\theta,G}$ is dense in $\A_\theta$.

  Write $I \coloneqq I_2$.

  Since $\sigma(x)^2=\norm{x}{\R^2}^2\cdot 1$ for all $x\in\R^2$, and
  \begin{equation*}
    \frac{\norm{x}{\R^2}}{\norm{G^{-1}}{M_2(\R)}}\leq\norm{G x}{\R^2}\leq\norm{G}{M_2(\R)}\norm{x}{\R^2}
  \end{equation*}
  we conclude
  \begin{align*}
    \dom{\Dirac_G} &=\left\{\xi\in\Hilbert:\sum_{z\in\Z^2}\norm{z}{\R^2}^2\norm{\xi_z}{\C^2}^2<\infty\right\} \\
    &= \dom{\Dirac_{I}} \text.
  \end{align*}

  Write $S_j=1\otimes\sigma_j$ and $G=(G_{jk})_{1\leq j,k\leq 2}$. On the common domain $\dom{\Dirac_I}$,
  \begin{equation*}
    \Dirac_G=\frac{1}{2}\sum_{j,k=1}^2G_{jk}S_j(S_k\Dirac_I+\Dirac_IS_k), \qquad \Dirac_I=\frac{1}{2}\sum_{j,k=1}^2(G^{-1})_{jk}S_j(S_k\Dirac_G+\Dirac_GS_k).
  \end{equation*}
  Each $S_j$ commutes with $\A_\theta$; hence $\A_{\theta,G}=\A_{\theta,I}$ and
  \begin{equation*}
    \left(\sum_{j,k=1}^2|(G^{-1})_{jk}|\right)^{-1}\Lip_{\theta,I} \leq\Lip_{\theta,G} \leq\left(\sum_{j,k=1}^2|G_{jk}|\right)\Lip_{\theta,I} \text.
  \end{equation*}
  By \cite{Rieffel98a}, $\Lip_{\theta,I}$ is a Lip-norm; therefore so is $\Lip_{\theta,G}$.
\end{proof}

\begin{notation}
  Henceforth, we set $\dom{\Dirac} \coloneqq \dom{\Dirac_{I_2}}$.
\end{notation}

Notably, $(\A_\theta,\Hilbert,\Dirac_{I_2})$ is the usual ``flat'' noncommutative torus spectral triple, known to be metric by \cite{Rieffel98a}, and the object of \cite{Latremoliere26b} which initiated the theory underpinning the current work. In general, $(\A_\theta,\Hilbert,\Dirac_G)$ is the base metric spectral triple for this section. We now introduce an action of $\Z$.

We also record that, for every positive definite real symmetric matrix $G$, and for all $x \in \T^2$,
\begin{equation*}
  u^x\Dirac_G=\Dirac_G u^x \text.
\end{equation*}

For $N\in\N$, set
\begin{equation*}
  f_N: x \in \T \mapsto \sum_{|k|\leq N}\left(1-\frac{|k|}{N+1}\right)x^k \text,
\end{equation*}
\begin{equation*}
  ff_N : (x_1,x_2) \in \T^2 \mapsto f_N(x_1) f_N(x_2) \text,
\end{equation*}
and define the product Fej\'er averages by
\begin{align*}
  \zeta^{ff_N}(a) &=\int_{\T^2}f_N(x_1)f_N(x_2)\zeta^x(a)\, d\mu(x_1,x_2) \\
  &=\sum_{\substack{z=(z_1,z_2)\in\Z^2 \\ \norm{z}{\ell^\infty}\leq N}}\prod_{j=1}^2\left(1-\frac{|z_j|}{N+1}\right)\operatorname{tr}_\theta(a(W_\theta^z)^*)W_\theta^z \text.
\end{align*}
The Fej\'er kernels form a positive normalized approximate identity. Moreover, for $a\in\dom{\Lip_{\theta,G}}$,
\begin{equation*}
  [\Dirac_G,\zeta^{ff_N}(a)]=\int_{\T^2} ff_N(z) u^z [\Dirac_G,a] (u^z)^* \, d\mu(z),
\end{equation*}
where the operator integral is taken in the weak operator topology. Consequently,
\begin{equation}\label{vm:torus-fejer}
  \zeta^{ff_N}(a)\xrightarrow[N\longrightarrow\infty]{\norm{\cdot}{\A_\theta}}a,\qquad \Lip_{\theta,G}(\zeta^{ff_N}(a))\leq\Lip_{\theta,G}(a)\text.
\end{equation}

\subsubsection{The C*-crossed-product}

Let us now fix:
\begin{equation*}
  A \in SL_2(\Z) \text.
\end{equation*}

For all $\xi \in \Hilbert = \ell^2(\Z^2,\C^2)$, we define
\begin{equation*}
  V_A \xi : z \in \Z^2 \mapsto \xi(A^{-1} z) \text.
\end{equation*}
The operator $V_A$ is a unitary operator on $\Hilbert$. A simple computation shows that, for all $z \in \Z^2$,
\begin{equation*}
  V_A W_\theta^z V_A^{\ast} = W_\theta^{Az} \text.
\end{equation*}
Thus, $b \in \A_\theta \mapsto V_A b V_A^{\ast}$ is an automorphism of $\A_\theta$ (of inverse $V_A^\ast\cdot V_A$). This may also be derived from universality and the fact that $\det(Az,Am) = \det(z,m)$ for all $z,m \in \Z^2$ since $A\in\operatorname{SL}(2,\Z)$. We denote the action of $\Z$ on $\A_\theta$ generated by $V_A \cdot V_A^\ast$ by $\alpha_{A,\theta}$.

We shall prove continuity of the twisted spectral triples obtained from our construction on
\begin{equation*}
  \A_\theta\rtimes_{\alpha_{A,\theta}}\Z
\end{equation*}
when the parameter $\theta$ and the metric on the base quantum torus vary. The argument applies to every fixed $A\in\operatorname{SL}(2,\Z)$, in particular to every hyperbolic matrix, for which the action $\alpha_{A,\theta}$ is \emph{not} equicontinuous, so our framework is indeed needed.

\subsubsection{The weight}

Since $\mathscr{S}_+$ is locally compact, every point in it has a compact neighborhood, so we shall henceforth assume that we fix some nonempty compact subset $\mathscr{G}$ of $\mathscr{S}_+$. The following lemma proves that $\alpha_{A,\theta}$ preserves $\dom{\Lip_{\theta,G}}$ and provides a common adapted exponential weight for all $\theta\in\R$ and $G\in\mathscr{G}$.

\begin{lemma}
  \label{vm:relative} If $\mathscr{G}$ is a nonempty compact subset of $\mathscr{S}_+$, and if $r \in \R$ is such that
  \begin{equation*}
    r\geq\sup_{G\in\mathscr{G}}\max\left\{1,\norm{GAG^{-1}}{M_2(\R)},\norm{GA^{-1}G^{-1}}{M_2(\R)}\right\},
  \end{equation*}
  and
  \begin{equation*}
    w : z \in \Z \mapsto r^{-|z|}
  \end{equation*}
  then $w$ is an admissible, adapted weight for $\alpha_{A,\theta}$ and $\Lip_{\theta,G}$, for all $G \in \mathscr{G}$ and $\theta\in \R$.

  For $G,G_0\in\mathscr{G}$, set
  \begin{equation*}
    \delta_G\coloneqq\norm{(G-G_0)G_0^{-1}}{M_2(\R)}.
  \end{equation*}
  Then, for every $\xi\in\dom{\Dirac_{G_0}}$,
  \begin{equation*}
    \norm{(\Dirac_G-\Dirac_{G_0})\xi}{\Hilbert}\leq\delta_G\norm{\Dirac_{G_0}\xi}{\Hilbert}.
  \end{equation*}
\end{lemma}

\begin{proof}
  Set
  \begin{equation*}
    \sigma_3\coloneqq\begin{pmatrix}1&0\\0&-1\end{pmatrix}.
  \end{equation*}
  The matrices $1,\sigma_1,\sigma_2,\sigma_3$ form a basis of $M_2(\C)$. For $A\in M_2(\R)$, write
  \begin{equation*}
    A=s_0 1+s_1\sigma_1-is_2\sigma_2+s_3\sigma_3 =\begin{pmatrix}s_0+s_3&s_1-s_2\\s_1+s_2&s_0-s_3\end{pmatrix}, \qquad s_0,s_1,s_2,s_3\in\R.
  \end{equation*}
  For every $x\in\R^2$, the Pauli relations give
  \begin{equation*}
    \sigma(\sigma_1x)=-i\sigma_3\sigma_1\sigma(x)\sigma_1\text, \qquad \sigma(i\sigma_2x)=i\sigma_3\sigma(x)\text,\qquad \sigma(\sigma_3x)=\sigma_1\sigma(x)\sigma_1\text.
  \end{equation*}
  Consequently,
  \begin{equation*}
    \sigma(Ax)=(s_0 1-is_2\sigma_3)\sigma(x)+(s_3 1-is_1\sigma_3)\sigma_1\sigma(x)\sigma_1.
  \end{equation*}
  Moreover, setting $\rho=(s_0^2+s_2^2)^{1/2}$ and $\tau=(s_1^2+s_3^2)^{1/2}$, we have
  \begin{equation*}
    \norm{s_0 1-is_2\sigma_3}{M_2(\C)}=\rho,\qquad \norm{s_3 1-is_1\sigma_3}{M_2(\C)}=\tau.
  \end{equation*}
  A direct computation gives
  \begin{equation*}
    \operatorname{tr}(A^{\mathsf T}A)=2(\rho^2+\tau^2),\qquad \det(A^{\mathsf T}A)=(\rho^2-\tau^2)^2,
  \end{equation*}
  so the eigenvalues of $A^{\mathsf T}A$ are $(\rho+\tau)^2$ and $(\rho-\tau)^2$. Thus
  \begin{equation*}
    \norm{A}{M_2(\R)}=\rho+\tau.
  \end{equation*}

  Now let $a=\sum_p a_pW_\theta^p$ be a Fourier polynomial and $n\in\Z$. Since
  \begin{equation*}
    V_A^n(\phi_p\otimes v)=\phi_{A^np}\otimes v,
  \end{equation*}
  the unitary $V_A^n$ commutes with $1\otimes\sigma_j$ for $j\in\{1,2,3\}$, and
  \begin{equation*}
    V_A^{-n}\Dirac_GV_A^n(\phi_p\otimes v)=\phi_p\otimes\sigma(GA^np)v.
  \end{equation*}
  Therefore,
  \begin{equation*}
    \begin{aligned}
      \Lip_{\theta,G}(\alpha_{A,\theta}^n(a)) &=\norm{[\Dirac_G,V_A^naV_A^{-n}]}{\B(\Hilbert)}\\
      &=\norm{V_A^n[V_A^{-n}\Dirac_GV_A^n,a]V_A^{-n}}{\B(\Hilbert)}\\
      &=\norm{\sum_p a_pW_\theta^p(1\otimes\sigma(GA^np))}{\B(\Hilbert)}.
    \end{aligned}
  \end{equation*}
  Apply the preceding decomposition to $M=GA^nG^{-1}$, denoting its coefficients again by $s_0,s_1,s_2,s_3$. Set
  \begin{equation*}
    C\coloneqq[\Dirac_G,a]=\sum_p a_pW_\theta^p(1\otimes\sigma(Gp)), \qquad S\coloneqq1\otimes\sigma_1.
  \end{equation*}
  Since $GA^np=M(Gp)$, we obtain
  \begin{equation*}
    \sum_p a_pW_\theta^p(1\otimes\sigma(GA^np)) =(1\otimes(s_0 1-is_2\sigma_3))C+(1\otimes(s_3 1-is_1\sigma_3))SCS.
  \end{equation*}
  Hence
  \begin{equation}\label{vm:matrix-estimate}
    \begin{aligned}
      \Lip_{\theta,G}(\alpha_{A,\theta}^n(a)) &\leq\left(\sqrt{s_0^2+s_2^2}+\sqrt{s_1^2+s_3^2}\right)\norm{C}{\B(\Hilbert)}\\
      &=\norm{GA^nG^{-1}}{M_2(\R)}\Lip_{\theta,G}(a) \leq r^{|n|}\Lip_{\theta,G}(a).
    \end{aligned}
  \end{equation}
  For $a\in\dom{\Lip_{\theta,G}}$, lower semicontinuity and \eqref{vm:torus-fejer} give
  \begin{equation*}
    \Lip_{\theta,G}(\alpha_{A,\theta}^n(a))\leq\liminf_{N\longrightarrow\infty}\Lip_{\theta,G}(\alpha_{A,\theta}^n(\zeta^{ff_N}(a)))\leq r^{|n|}\Lip_{\theta,G}(a)\text.
  \end{equation*}
  Hence
  \begin{equation*}
    \sup_{n\in\Z}w(n)\Lip_{\theta,G}(\alpha_{A,\theta}^{-n}(a))=\Lip_{\theta,G}(a),\qquad r^{-|k|}\leq\frac{w(n+k)}{w(n)}\leq r^{|k|}\quad(n,k\in\Z).
  \end{equation*}
  Since $w(0)=1$, the weight $w$ is admissible and appropriate, and $(\Lip_{\theta,G})_w = \Lip_{\theta,G}$.

  For every $\xi\in\dom{\Dirac}$,
  \begin{equation*}
    \begin{aligned}
      \norm{(\Dirac_G-\Dirac_{G_0})\xi}{\Hilbert}^2 &=\sum_{p\in\Z^2}\norm{(G-G_0)p}{\R^2}^2\norm{\xi_p}{\C^2}^2\\
      &\leq\delta_G^2\sum_{p\in\Z^2}\norm{G_0p}{\R^2}^2\norm{\xi_p}{\C^2}^2 =\delta_G^2\norm{\Dirac_{G_0}\xi}{\Hilbert}^2\text,
    \end{aligned}
  \end{equation*}
  as claimed.
\end{proof}

We are thus able to apply Theorems \ref{twisted-spectral-triple-thm} and \ref{metric-thm}.

We fix a compact nonempty subset $\mathscr{G}$ of definite positive real symmetric matrices. We let $w : \Z \longrightarrow (0,1]$ and $r$ be defined by Lemma \ref{vm:relative}.

Fix $\tau>0$ and $\kappa\in(0,1]$, and let $\gamma_0,\gamma_1,\gamma_2\in M_2(\C)$ be the anticommuting self-adjoint unitaries used in our construction (again we can choose Pauli matrices $\sigma_1$,$\sigma_2$ and $\sigma_3$ here).

Following Theorem \ref{twisted-spectral-triple-thm}, we thus set:
\begin{equation*}
  \Hilbert[K]\coloneqq\ell^2(\Z,\Hilbert)\otimes\C^2 \text.
\end{equation*}
For $a\in\A_\theta$, $k\in\Z$, and $\xi\in\ell^2(\Z,\Hilbert)$, we define
\begin{equation*}
  (\pi_\theta(a)\xi) : z \in \Z \mapsto \alpha_{A,\theta}^{-z}(a)\xi_z,\qquad U^k\xi : z\in\Z \mapsto \xi_{z-k}\text.
\end{equation*}
In particular, $U^k\pi_\theta(a)U^{-k}=\pi_\theta(\alpha_{A,\theta}^k(a))$. This representation is faithful, since $\Z$ is amenable.

We define
\begin{equation*}
  \partial\xi : z \in \Z \mapsto z \xi_z
\end{equation*}
for all
\begin{equation*}
  \xi \in \dom{\partial} \coloneqq \{ (\xi_z)_{z\in\Z} \in \ell^2(\Z,\Hilbert) : \sum_{z\in\Z} z^2 \norm{\xi_z}{\Hilbert}^2 < \infty \} \text.
\end{equation*}
Let $M_w \xi : z\in\Z \mapsto w(z)\xi_z$ for all $\xi \in \ell^2(\Z,\Hilbert)$, i.e., the multiplication operator by the weight $w$ on $\ell^2(\Z,\Hilbert)$. For all $(\xi_z)_{z\in\Z}$ in $\ell^2(\Z,\Hilbert)$, if $\xi_z \in \dom{\Dirac}$ for all $z\in \Z$, we then set
\begin{equation*}
  \Dirac_G^{\mathrm{diag}}\xi : z \in \Z \mapsto \Dirac_G \xi_z \text.
\end{equation*}

On the dense subspace
\begin{equation*}
  \begin{aligned}
    \dom{\nabla}\coloneqq\biggl\{&(\xi_z)_{z\in\Z}\in\ell^2(\Z,\Hilbert):\xi_z\in\dom{\Dirac}\text{ for all }z\in\Z,\\
    &\sum_{z\in\Z}\left(z^2\norm{\xi_z}{\Hilbert}^2+w(z)^2\norm{\Dirac_G\xi_z}{\Hilbert}^2\right)<\infty\biggr\}\otimes\C^2
  \end{aligned}
\end{equation*}
we then define
\begin{equation*}
  \Dirac[\nabla] \coloneqq M_w \Dirac_G^{\mathrm{diag}}\otimes\gamma_0 + \partial\otimes\gamma_1 + \tau 1\otimes\gamma_2 \text.
\end{equation*}

We write $\Dirac[\nabla]_G$ for the operator $\Dirac[\nabla]$ just defined, to display its dependence on $G$. The weighted expression $M_w\Dirac_G^{\mathrm{diag}}$ is taken on its maximal domain
\begin{equation*}
  \begin{aligned}
    \dom{M_w\Dirac_G^{\mathrm{diag}}}=\biggl\{&\xi\in\ell^2(\Z,\Hilbert):\xi_n\in\dom{\Dirac}\text{ for all }n\in\Z,\\
    &\sum_{n\in\Z}w(n)^2\norm{\Dirac_G\xi_n}{\Hilbert}^2<\infty\biggr\}\text.
  \end{aligned}
\end{equation*}
Set 
\begin{equation}\label{YG-eq}
	Y_G\coloneqq(M_w\Dirac_G^{\mathrm{diag}})\otimes\gamma_0\text,
\end{equation}
so $\Dirac[\nabla]_G = Y_G + \partial\otimes\gamma_1 + \tau \gamma_2$. Then
\begin{equation*}
  \dom{\Dirac[\nabla]_G}=\dom{\nabla},\qquad \Dirac[\nabla]_G^2=(M_w\Dirac_G^{\mathrm{diag}})^2\otimes1+\partial^2\otimes1+\tau^2 1\text.
\end{equation*}
Lemma~\ref{vm:relative} implies that this domain is independent of $G\in\mathscr{G}$.

For every $k\in\Z$, let $M_{q_k}$ be multiplication by
\begin{equation*}
  q_k(n)\coloneqq\frac{w(n+k)}{w(n)}-1,\qquad C_w(k)\coloneqq\sup_{n\in\Z}|q_k(n)|\text.
\end{equation*}
The twist from Theorem~\ref{twisted-spectral-triple-thm} satisfies
\begin{equation}\label{rho-eq}
  \rho_{\theta,G}(\pi_\theta(a)U^k)=\pi_\theta(a)(U^k+\Theta_{G,k})
 \end{equation}
where
\begin{equation}\label{Theta-eq}
	\Theta_{G,k}\coloneqq U^kM_{q_k}Y_G\Dirac[\nabla]_G^{-1}\text.
\end{equation}
Indeed, on the finite-support core,
\begin{equation*}
  [M_w,U^k](\Dirac_G^{\mathrm{diag}}\otimes\gamma_0)=U^kM_{q_k}Y_G,\qquad \norm{Y_G\Dirac[\nabla]_G^{-1}}{\B(\Hilbert[K])}\leq1,
\end{equation*}
and hence $\norm{\Theta_{G,k}}{\B(\Hilbert[K])}\leq C_w(k)$. Identity operators on the spin factors are suppressed when no ambiguity arises.

We retain the twist domain from the general construction and denote the crossed-product Lipschitz seminorm by $\Lip^{\nabla}_{\theta,G}$, reserving $\Lip_{\theta,G}$ for the base algebra. By Theorem~\ref{metric-thm},
\begin{equation*}
  \left(\A_\theta\rtimes_{\alpha_{A,\theta}}\Z,\Hilbert[K],\Dirac[\nabla]_G,\rho_{\theta,G}\right)
\end{equation*}
is a $\kappa$-metric twisted spectral triple. We also write $\operatorname{tr}_\theta$ for the canonical trace of the crossed product, obtained by composing the base trace with the zeroth outer Fourier projection $\mathcal{F}_0$.

For $\xi\in\dom{\nabla}$ and $T\in\B(\Hilbert[K])$, set
\begin{equation*}
  \begin{aligned}
    &\CDN_G(\xi)\coloneqq\norm{\xi}{\Hilbert[K]}+\norm{\Dirac[\nabla]_G\xi}{\Hilbert[K]},\\
    &\opnorm{T}{}{t_0}\coloneqq\sup\left\{\norm{T\xi}{\Hilbert[K]}:\xi\in\dom{\nabla},\ \CDN_{G_0}(\xi)\leq1\right\},
  \end{aligned}
\end{equation*}
where $t_0=(\theta_0,G_0)$ is the reference parameter. We write $\CDN_{t_0}=\CDN_{G_0}$.

To prove that this construction is continuous in the parameters $\theta$ and $G$, we will emply a strategy motivated by our work in \cite{Latremoliere13c}. Thus, in particular, we will introduce a pivot which almost commutes, appropriately defined, with the canonical generators of our  C*-crossed-product, is finite rank, and has an eigenvalue of $1$.  This is best achieved when we change our representation by a well chosen unitary.

\subsubsection{Semidirect product realization}

Let $\Gamma\coloneqq\Z^2\rtimes_A\Z$, with identity $e=(0,0)$, multiplication, and cocycle given for all $(m,k), (p,n) \in \Gamma$ by
\begin{equation*}
  (m,k)(p,n)=(m+A^kp,k+n),\qquad \sigma_\theta((m,k),(p,n))=\exp\left(\pi i\theta\det(m,A^kp)\right)\text.
\end{equation*}
We denote the canonical orthonormal basis of $\ell^2(\Gamma)$ by $(\phi_g)_{g\in\Gamma}$. The group $\Gamma$ is torsion-free: $(m,k)^j=e$, for $j\geq1$, implies $jk=0$ and then $jm=0$.

For all $g \in \Gamma$, we define the operator $\lambda_\theta(g)$ on $\ell^2(\Gamma)\otimes\C^2\otimes\C^2$ as  
\begin{equation*}
	\lambda_\theta(g) \xi : h \in \Gamma \mapsto \sigma(g,g^{-1} h)\xi(g^{-1} h)
\end{equation*}
for all $\xi \in \ell^2(\Gamma)$. The canonical basis for $\ell^2(\Gamma)$ is denoted by $(\phi_g)_{g\in \Gamma}$.

\begin{lemma}\label{vm:semidirect-realization}
  The function $\sigma_\theta$ is a normalized unitary $2$-cocycle on $\Gamma$. Let $J$ be the unique unitary from $\Hilbert[K]$ onto $\ell^2(\Gamma)\otimes \C^2 \otimes \C^2$ such that
  \begin{equation*}
     J(\phi_n\otimes\phi_p\otimes v\otimes c)=\phi_{(A^np,n)}\otimes v\otimes c
  \end{equation*}
  for all $n\in\Z$, $p\in\Z^2$, and $v,c\in\C^2$. 
  
  We then observe that, for all $z \in \Z^2$ and $k\in \Z$,
  \begin{equation*}
    J\pi_\theta(W_\theta^z)U^k J^*=\lambda_\theta(z,k) \text.
  \end{equation*}
  Consequently, $\mathrm{Ad}_J$ is a *-isomorphism from $\A_\theta\rtimes_{\alpha_{A,\theta}}\Z$ onto $C^\ast_{\mathrm r}(\Gamma,\sigma_\theta)$. Moreover, $J\Dirac[\nabla]_G J^*$ is multiplication by
  \begin{equation*}
    H_G(p,n)=w(n)\sigma(GA^{-n}p)\otimes\gamma_0+n1\otimes\gamma_1+\tau1\otimes\gamma_2
  \end{equation*}
  on $J\dom{\nabla}$.
\end{lemma}
\begin{proof}
  The map $(p,n)\mapsto(A^np,n)$ is a bijection of $\Z^2\times\Z$, so the stated formula defines a unique unitary $J$.  Moreover, a simple computation shows
  \begin{equation*}
  J^*(\phi_{(p,n)}\otimes v\otimes c)=\phi_n\otimes\phi_{A^{-n}p}\otimes v\otimes c
  \qquad(p\in\Z^2,\ n\in\Z,\ v,c\in\C^2)
\end{equation*}
for all $p \in \Z^2$, $n \in \Z$, $v,c \in \C^2$.

  For $z,p\in\Z^2$, $k,n\in\Z$, and $v,c\in\C^2$, since
  \begin{equation}\label{det-eq}
  \begin{split}
    \det(A^{-(n+k)}z,A^{-n}p) 
    &=\det(A^{n+k}A^{-(n+k)}z,A^{n+k}A^{-n}p) \\
    &=\underbracket{\det(A^{n+k})}_{=1}\det(z,A^k p) =\det(z,A^k p)\text,
  \end{split}
  \end{equation}
  we compute
  \begin{align*}
    &J\pi_\theta(W_\theta^z)U^k J^*(\phi_{(p,n)}\otimes v\otimes c)\\
    &=J\pi_\theta(W_\theta^z)U^k J^*(\phi_{n}\otimes\phi_{A^{-n}p}\otimes v\otimes c)\\
    &=J\pi_\theta(W_\theta^z)(\phi_{n+k}\otimes\phi_{A^{-n}p}\otimes v\otimes c)\\
    &=J W_\theta^{A^{-n-k} z}(\phi_{n+k}\otimes\phi_{A^{-n}p}\otimes v\otimes c)\\
    &=\exp\left(\pi i\theta\underbracket{\det(A^{-(n+k)}z,A^{-n}p)}_{=\det(z,A^k p) \text{ by Eq. \eqref{det-eq}}}\right)
      J(\phi_{n+k}\otimes\phi_{A^{-(n+k)}z+A^{-n}p}\otimes v\otimes c)\\
    &=\exp\left(\pi i\theta\det(z,A^kp)\right)
      \phi_{(z+A^kp,n+k)}\otimes v\otimes c\\
    &=\lambda_\theta(z,k)(\phi_{(p,n)}\otimes v\otimes c)\text.
  \end{align*}
 
  For $g=(z,k)$ and $h=(p,n)$ in $\Gamma$, covariance and the Weyl relations now give
  \begin{align*}
    \lambda_\theta(g)\lambda_\theta(h)
    &=J\pi_\theta(W_\theta^z)U^k\pi_\theta(W_\theta^p)U^nJ^*\\
    &=J\pi_\theta(W_\theta^zW_\theta^{A^kp})U^{k+n}J^*\\
    &=\exp\left(\pi i\theta\det(z,A^kp)\right)
      J\pi_\theta(W_\theta^{z+A^kp})U^{k+n}J^*\\
    &=\sigma_\theta(g,h)\lambda_\theta(gh)\text.
  \end{align*}
  Associativity therefore yields, for all $g,h,l\in\Gamma$,
  \begin{equation*}
    \sigma_\theta(g,h)\sigma_\theta(gh,l)\lambda_\theta(ghl)
    =\sigma_\theta(h,l)\sigma_\theta(g,hl)\lambda_\theta(ghl).
  \end{equation*}
  Cancelling the unitary $\lambda_\theta(ghl)$ proves the cocycle identity. The defining formula gives $|\sigma_\theta(g,h)|=1$ and $\sigma_\theta(g,e)=\sigma_\theta(e,g)=1$.  Since $\pi_\theta(W_\theta^z)U^k$, for $(z,k)\in\Gamma$, generate $\A_\theta\rtimes_{\alpha_{A,\theta}}\Z$, we obtain
  \begin{equation*}
    J\bigl(\A_\theta\rtimes_{\alpha_{A,\theta}}\Z\bigr)J^*
    =C^\ast\bigl(\lambda_\theta(z,k):(z,k)\in\Gamma\bigr)
    =C^\ast_{\mathrm r}(\Gamma,\sigma_\theta)\text.
  \end{equation*}

  Finally, for $p\in\Z^2$, $n\in\Z$, and $v,c\in\C^2$,
  \begin{align*}
    &J\Dirac[\nabla]_GJ^*(\phi_{(p,n)}\otimes v\otimes c)\\
    &=J\Dirac[\nabla]_G(\phi_n\otimes\phi_{A^{-n}p}\otimes v\otimes c)\\
    &=J\Bigl(\phi_n\otimes\phi_{A^{-n}p}\otimes
      \bigl(w(n)\sigma(GA^{-n}p)v\otimes\gamma_0c
      +nv\otimes\gamma_1c+\tau v\otimes\gamma_2c\bigr)\Bigr)\\
    &=\phi_{(p,n)}\otimes H_G(p,n)(v\otimes c)\text.
  \end{align*}
  Since
  \begin{equation*}
    H_G(p,n)^2=\left(w(n)^2\norm{GA^{-n}p}{\R^2}^2+n^2+\tau^2\right)1,
  \end{equation*}
  its domain is precisely
  \begin{equation*}
    J\dom{\nabla}
    =\left\{\xi\in\ell^2(\Gamma,\C^2\otimes\C^2):
    \sum_{(p,n)\in\Gamma}\left(w(n)^2\norm{GA^{-n}p}{\R^2}^2+n^2+\tau^2\right)
    \norm{\xi_{p,n}}{\C^2\otimes\C^2}^2<\infty\right\}\text,
  \end{equation*}
  conclusing our poof.
\end{proof}
Henceforth, we use this unitary identification and retain the notation for the transported operators and representations. In particular, $\rho_{\theta,G}$ is transported by conjugation with $\mathcal J$; these are the triples constructed above.

\begin{lemma}\label{vm:twist-norm}
  Let $G,G_0\in\mathscr{G}$ and set
  \begin{equation*}
  	\delta_G \coloneqq \norm{(G-G_0)G_0^{-1}}{M_2(\R)}
 \end{equation*}

  If $\delta_G<1$, then for every $k\in\Z$, with the notation of Eq. \eqref{Theta-eq},
    \begin{equation*}
    \norm{\Theta_{G,k}-\Theta_{G_0,k}}{\B(\Hilbert[K])} \leq 2(r^{|k|}-1)\frac{\delta_G}{1-\delta_G}\text.
  \end{equation*}
\end{lemma}

\begin{proof}
  For every $k\in\Z$,
  \begin{equation*}
    \norm{M_{q_k}}{\B(\Hilbert[K])}
    =C_w(k)
    =\sup_{n\in\Z}\left|\frac{w(n+k)}{w(n)}-1\right|
    =r^{|k|}-1\text.
  \end{equation*}
  Set $\Delta\coloneqq\Dirac[\nabla]_G-\Dirac[\nabla]_{G_0}=Y_G-Y_{G_0}$.
  By \eqref{Theta-eq} and the resolvent identity,
  \begin{align*}
    \Theta_{G,k}-\Theta_{G_0,k}
    &=U^kM_{q_k}\bigl(Y_G\Dirac[\nabla]_G^{-1}
      -Y_{G_0}\Dirac[\nabla]_{G_0}^{-1}\bigr)\\
    &=U^kM_{q_k}\bigl(\Delta\Dirac[\nabla]_G^{-1}
      +Y_{G_0}(\Dirac[\nabla]_G^{-1}-\Dirac[\nabla]_{G_0}^{-1})\bigr)\\
    &=U^kM_{q_k}(1-Y_{G_0}\Dirac[\nabla]_{G_0}^{-1})
      \Delta\Dirac[\nabla]_G^{-1}\text.
  \end{align*}
  The formula for $\Dirac[\nabla]_{G_0}^2$ gives
  \begin{equation*}
    \norm{Y_{G_0}\Dirac[\nabla]_{G_0}^{-1}}{\B(\Hilbert[K])}\leq1\text,
  \end{equation*}
  so
  \begin{equation*}
    \norm{\Theta_{G,k}-\Theta_{G_0,k}}{\B(\Hilbert[K])}
    \leq2C_w(k)\norm{\Delta\Dirac[\nabla]_G^{-1}}{\B(\Hilbert[K])}\text.
  \end{equation*}
  To bound the remaining factor, Lemma~\ref{vm:relative} gives, for every $\xi\in\dom{\nabla}$,
  \begin{align*}
    \norm{\Delta\xi}{\Hilbert[K]}
    &\leq\delta_G\norm{\Dirac[\nabla]_{G_0}\xi}{\Hilbert[K]}\\
    &\leq\delta_G\bigl(\norm{\Dirac[\nabla]_G\xi}{\Hilbert[K]}
      +\norm{\Delta\xi}{\Hilbert[K]}\bigr)\text.
  \end{align*}
  Since $\delta_G<1$, this implies
  \begin{equation*}
    \norm{\Delta\xi}{\Hilbert[K]}
    \leq\frac{\delta_G}{1-\delta_G}\norm{\Dirac[\nabla]_G\xi}{\Hilbert[K]}\text.
  \end{equation*}
  Applying this inequality to $\xi=\Dirac[\nabla]_G^{-1}\eta$ for $\eta\in\Hilbert[K]$ yields
  \begin{equation*}
    \norm{\Delta\Dirac[\nabla]_G^{-1}}{\B(\Hilbert[K])}
    \leq\frac{\delta_G}{1-\delta_G}\text.
  \end{equation*}
  Hence
  \begin{equation*}
    \norm{\Theta_{G,k}-\Theta_{G_0,k}}{\B(\Hilbert[K])}
    \leq2C_w(k)\frac{\delta_G}{1-\delta_G}
    =2(r^{|k|}-1)\frac{\delta_G}{1-\delta_G}\text,
  \end{equation*}
  as claimed.
\end{proof}
\subsubsection{Finite Fourier approximations}

Fix the word length $|\cdot|$ on $\Gamma$ associated with the generators
\begin{equation*}
  g_0\coloneqq((0,0),1),\qquad g_1\coloneqq((1,0),0),\qquad g_2\coloneqq((0,1),0)
\end{equation*}
and their inverses.

\begin{lemma}\label{lem:folner-explicit}
  For every finite subset $S$ of $\Gamma$ and every $\varepsilon>0$, there exists a nonempty finite subset $F$ of $\Gamma$ such that
  \begin{equation*}
    \frac{|F\mathbin\triangle gF|}{|F|}<\varepsilon\qquad(g\in S)\text.
  \end{equation*}
\end{lemma}

\begin{proof}
  Choose $L\in\N$ such that $|g|\leq L$ for all $g\in S$. For positive integers $J,H$, set
  \begin{equation*}
    B_H\coloneqq[-H,H]^2\cap\Z^2,\qquad
    F_{J,H}\coloneqq\{(A^np,n):|n|\leq J,\ p\in B_H\}\text.
  \end{equation*}
  For $k\in\Z$ and $t\in\{(k,0),(0,k)\}$,
  \begin{equation*}
    |B_H\cap(B_H+t)|=(2H+1)\max\{2H+1-|k|,0\}\text,
  \end{equation*}
  so
  \begin{equation*}
    |B_H\mathbin\triangle(B_H+t)|
    =2(2H+1)\min\{|k|,2H+1\}
    \leq2(2H+1)|k|\text.
  \end{equation*}
  For $z=(z_1,z_2)\in\Z^2$, translation invariance and the triangle inequality give
  \begin{align*}
    |B_H\mathbin\triangle(B_H+z)|
    &\leq|B_H\mathbin\triangle(B_H+(z_1,0))|
      +|B_H\mathbin\triangle(B_H+(0,z_2))|\\
    &\leq2(2H+1)\norm{z}{\ell^1}\text.
  \end{align*}
  Dividing by $|B_H|=(2H+1)^2$ yields
  \begin{equation*}
    \frac{|B_H\mathbin\triangle(B_H+z)|}{|B_H|}
    \leq\frac{2\norm{z}{\ell^1}}{2H+1}\text.
  \end{equation*}
   Since
  \begin{equation*}
    g_0 F_{J,H}=\{(A^np,n):-J+1\leq n\leq J+1,\ p\in B_H\}\text,
  \end{equation*}
  we have
  \begin{equation*}
    F_{J,H}\mathbin\triangle g_0 F_{J,H}
    =\{(A^{-J}p,-J):p\in B_H\}
      \cup\{(A^{J+1}p,J+1):p\in B_H\}\text.
  \end{equation*}
  These two sets are disjoint and each has cardinality $|B_H|$, so
  \begin{equation*}
    \frac{|F_{J,H}\mathbin\triangle g_0F_{J,H}|}{|F_{J,H}|}
    =\frac{2|B_H|}{(2J+1)|B_H|}
    =\frac{2}{2J+1}\text.
  \end{equation*}
  
  For $z\in\{(1,0),(0,1)\}$,
  \begin{equation*}
    (z,0)(A^np,n)=(A^n(p+A^{-n}z),n)\text.
  \end{equation*}
  Applying
  \begin{equation*}
    \frac{|B_H\mathbin\triangle(B_H+t)|}{|B_H|}
    \leq\frac{2\norm{t}{\ell^1}}{2H+1}
    \qquad(t\in\Z^2)
  \end{equation*}
  with $t=A^{-n}z$ gives
  \begin{align*}
    \frac{|F_{J,H}\mathbin\triangle(z,0)F_{J,H}|}{|F_{J,H}|}
    &=\frac{1}{2J+1}\sum_{|n|\leq J}
      \frac{|B_H\mathbin\triangle(B_H+A^{-n}z)|}{|B_H|}\\
    &\leq\frac{2\max_{|n|\leq J}\norm{A^{-n}z}{\ell^1}}{2H+1}\text.
  \end{align*}
  Choose $J$, then $H$, so that the bounds for all three generators are less than $\varepsilon/(L+1)$. The same bounds hold for the inverses of the three generators $g_0$, $g_1$ and $g_2$, since
  \begin{equation*}
    |F_{J,H}\mathbin\triangle g_j^{-1}F_{J,H}|
    =|g_jF_{J,H}\mathbin\triangle F_{J,H}|
    =|F_{J,H}\mathbin\triangle g_jF_{J,H}|
    \qquad(j\in\{0,1,2\})\text.
  \end{equation*}
  For $g\in S$, write $g=s_1\cdots s_l$ with $l=|g|\leq L$ and each $s_j$ a generator or its inverse. Translation invariance and the triangle inequality give
  \begin{equation*}
    |F_{J,H}\mathbin\triangle gF_{J,H}|
    \leq\sum_{j=1}^l|F_{J,H}\mathbin\triangle s_jF_{J,H}|
    \leq\frac{l\varepsilon}{L+1}|F_{J,H}|
    <\varepsilon|F_{J,H}|\text.
  \end{equation*}
  We conclude by choosing $F=F_{J,H}$.
\end{proof}

Let us explain the next lemma, which is essential to us, and requires a bit of context. We shall approximate finite Fourier sums by multiplying their coefficients by
\begin{equation*}
  \psi_F(g)\coloneqq\frac{|F\cap gF|}{|F|}\qquad(g\in\Gamma),
\end{equation*}
where $F\subseteq\Gamma$ is nonempty and finite. Indeed, for a finite subset $S\subseteq\Gamma$,
\begin{equation*}
  \norm{\sum_{g\in S}b_g\lambda_\theta(g)
    -\sum_{g\in S}\psi_F(g)b_g\lambda_\theta(g)}{}
  \leq\sum_{g\in S}(1-\psi_F(g))|b_g|\text.
\end{equation*}
Lemma~\ref{lem:folner-explicit} makes this bound arbitrarily small by choosing $F$ sufficiently invariant under the elements of $S$.

To apply the same averaging to twisted commutators, we extend it to $\B(\Hilbert[K])$. We use the standard construction of completely positive Schur multipliers from factorizations of positive definite kernels; see \cite[Sections 2.3--2.4]{MukherjeeRecke2024}. In the present case, the factorization is explicit:
\begin{equation*}
  \psi_F(sh^{-1})
  =\frac{|\{x\in\Gamma:s,h\in Fx\}|}{|F|}
  =\frac1{|F|}\sum_{x\in\Gamma}1_{Fx}(s)1_{Fx}(h)
  \qquad(s,h\in\Gamma).
\end{equation*}
Thus, if $Q_x$ denotes multiplication by $1_{Fx}$ on $\Hilbert[K]$, the corresponding map is
\begin{equation*}
  \mathcal S_F(T)\coloneqq\frac1{|F|}
    \sum_{x\in\Gamma}Q_xTQ_x
  \qquad(T\in\B(\Hilbert[K])).
\end{equation*}

We justify this sum and record the estimates needed below. Since each $h\in\Gamma$ belongs to exactly $|F|$ translates $Fx$,
\begin{equation*}
  \sum_{x\in\Gamma}\norm{Q_x\xi}{\Hilbert[K]}^2
  =|F|\norm{\xi}{\Hilbert[K]}^2
  \qquad(\xi\in\Hilbert[K]).
\end{equation*}
Consequently, for $\xi,\eta\in\Hilbert[K]$,
\begin{align*}
  \frac1{|F|}\sum_{x\in\Gamma}
    \left|\inner{Q_xTQ_x\xi}{\eta}{}\right|
  &\leq\frac{\sup_x\norm{Q_xTQ_x}{}}{|F|}
    \sum_{x\in\Gamma}\norm{Q_x\xi}{}\norm{Q_x\eta}{}\\
  &\leq\sup_x\norm{Q_xTQ_x}{}
    \norm{\xi}{\Hilbert[K]}\norm{\eta}{\Hilbert[K]}\text.
\end{align*}
The scalar series therefore defines a bounded sesquilinear form, represented by $\mathcal S_F(T)$. This defines the operator sum in the weak operator topology and proves
\begin{equation}\label{vm:compression-bound}
  \norm{\mathcal S_F(T)}{}
  \leq\sup_{x\in\Gamma}\norm{Q_xTQ_x}{}
  \leq\norm{T}{}\text.
\end{equation}
Every finite partial sum is completely positive, so the weak operator limit is completely positive at every matrix level. The same counting identity gives $\mathcal S_F(1)=1$.

For $g,h\in\Gamma$ and $v,c\in\C^2$, the vector
\begin{equation*}
  \lambda_\theta(g)(\psi_h\otimes v\otimes c)
  =\sigma_\theta(g,h)\psi_{gh}\otimes v\otimes c
\end{equation*}
is retained by $Q_x\lambda_\theta(g)Q_x$ precisely when $h,gh\in Fx$. There are $|F\cap gF|$ such translates, so
\begin{equation}\label{vm:averaging-fourier}
  \mathcal S_F(\lambda_\theta(g))
  =\psi_F(g)\lambda_\theta(g)\text.
\end{equation}
More generally, every bounded multiplication operator $M_d$ commutes with each $Q_x$, and hence
\begin{equation*}
  \mathcal S_F(M_d\lambda_\theta(g))
  =\psi_F(g)M_d\lambda_\theta(g)\text.
\end{equation*}
For $T_\theta=\sum_{g\in S}M_{d,g}\lambda_\theta(g)$ with $S$ finite, it follows that
\begin{equation*}
  \norm{T_\theta-\mathcal S_F(T_\theta)}{}
  \leq\sum_{g\in S}(1-\psi_F(g))\norm{M_{d,g}}{}\text.
\end{equation*}
Combining this with \eqref{vm:compression-bound} gives
\begin{equation}\label{vm:finite-compression-approximation}
  0\leq\norm{T_\theta}{}
    -\sup_{x\in\Gamma}\norm{Q_xT_\theta Q_x}{}
  \leq\sum_{g\in S}(1-\psi_F(g))\norm{M_{d,g}}{}\text.
\end{equation}
Thus the F{\o}lner property allows us to approximate the norm of $T_\theta$ by the supremum of its finite compression norms, uniformly in $\theta$ and in prescribed bounds on the coefficients.

\begin{lemma}\label{lem:diagonal-continuity}
  Let $S\subseteq\Gamma$ be finite and let $C:S\to[0,\infty)$. For every
  $d:S\times\Gamma\to M_4(\C)$ satisfying
  \begin{equation*}
    \sup_{h\in\Gamma}\norm{d(g,h)}{M_4(\C)}\leq C(g)
    \qquad(g\in S),
  \end{equation*}
  define the multiplication operators
  \begin{equation*}
    (M_{d,g}\xi)(h)\coloneqq d(g,h)\xi(h)
    \qquad(\xi\in\Hilbert[K],\ h\in\Gamma).
  \end{equation*}
  Then
  \begin{equation*}
    \theta\longmapsto
    \norm{\sum_{g\in S}M_{d,g}\lambda_\theta(g)}{\B(\Hilbert[K])}
  \end{equation*}
  is continuous, uniformly over all such $d$.
\end{lemma}
\begin{proof}
  Fix $\varepsilon>0$ and put
  \begin{equation*}
    T_\theta\coloneqq\sum_{g\in S}M_{d,g}\lambda_\theta(g)\text.
  \end{equation*}
  By Lemma~\ref{lem:folner-explicit}, choose a nonempty finite $F\subseteq\Gamma$ such that
  \begin{equation*}
    \sum_{g\in S}(1-\psi_F(g))C(g)<\frac{\varepsilon}{2}\text.
  \end{equation*}
  For each $x\in\Gamma$, let $Q_x$ be the orthogonal projection onto $\ell^2(Fx)\otimes\C^2\otimes\C^2$. We first prove that
  \begin{equation*}
    \norm{T_\theta}{}\leq\sup_{x\in\Gamma}\norm{Q_xT_\theta Q_x}{}+\frac{\varepsilon}{2}\text.
  \end{equation*}

  Since $\{x:h\in Fx\}=F^{-1}h$ has cardinality $|F|$ for every $h\in\Gamma$, we have
  \begin{equation*}
    \sum_{x\in\Gamma}\norm{Q_x\xi}{\Hilbert[K]}^2
    =|F|\norm{\xi}{\Hilbert[K]}^2
    \qquad(\xi\in\Hilbert[K]).
  \end{equation*}
  Thus, for $\xi,\eta\in\Hilbert[K]$, the Cauchy--Schwarz inequality gives
  \begin{align*}
    \frac1{|F|}\sum_{x\in\Gamma}\left|\inner{Q_xT_\theta Q_x\xi}{\eta}{}\right|
    &\leq\frac{\sup_{x\in\Gamma}\norm{Q_xT_\theta Q_x}{\B(\Hilbert[K])}}{|F|}
      \sum_{x\in\Gamma}\norm{Q_x\xi}{\Hilbert[K]}\norm{Q_x\eta}{\Hilbert[K]}\\
    &\leq\sup_{x\in\Gamma}\norm{Q_xT_\theta Q_x}{\B(\Hilbert[K])}
      \norm{\xi}{\Hilbert[K]}\norm{\eta}{\Hilbert[K]}\text.
  \end{align*}
  The scalar series therefore converges absolutely and defines a bounded sesquilinear form. Hence there is a unique $R_\theta\in\B(\Hilbert[K])$ such that
  \begin{equation*}
    \inner{R_\theta\xi}{\eta}{}
    =\frac1{|F|}\sum_{x\in\Gamma}\inner{Q_xT_\theta Q_x\xi}{\eta}{}
    \qquad(\xi,\eta\in\Hilbert[K]).
  \end{equation*}
  This is the meaning of the sum $R_\theta=|F|^{-1}\sum_{x\in\Gamma}Q_xT_\theta Q_x$ in the weak operator topology, and the preceding bound gives
  \begin{equation*}
    \norm{R_\theta}{}\leq\sup_{x\in\Gamma}\norm{Q_xT_\theta Q_x}{}\text.
  \end{equation*}

  For $g\in S$, $h\in\Gamma$, and $v,c\in\C^2$,
  \begin{equation*}
    Q_xM_{d,g}\lambda_\theta(g)Q_x(\psi_h\otimes v\otimes c)
    =
    \begin{cases}
      \sigma_\theta(g,h)\psi_{gh}\otimes d(g,gh)(v\otimes c),
        &h,gh\in Fx,\\
      0,&\text{otherwise}.
    \end{cases}
  \end{equation*}
  The number of nonzero contributions as $x$ varies is
  \begin{equation*}
    |\{x:h,gh\in Fx\}|=|F\cap g^{-1}F|=|F\cap gF|=|F|\psi_F(g)\text.
  \end{equation*}
  Summing on these basis tensors therefore yields
  \begin{equation*}
    R_\theta=\sum_{g\in S}\psi_F(g)M_{d,g}\lambda_\theta(g)\text,
  \end{equation*}
  so
  \begin{equation*}
    \norm{T_\theta-R_\theta}{}
    \leq\sum_{g\in S}(1-\psi_F(g))C(g)
    <\frac{\varepsilon}{2}\text.
  \end{equation*}
  Therefore,
  \begin{equation*}
    \norm{T_\theta}{}
    \leq\norm{R_\theta}{}+\norm{T_\theta-R_\theta}{}
    <\sup_{x\in\Gamma}\norm{Q_xT_\theta Q_x}{}+\frac{\varepsilon}{2}\text.
  \end{equation*}

  We next compare these compressions on the fixed space $\ell^2(F)\otimes\C^2\otimes\C^2$. For $x\in\Gamma$, define
  \begin{equation*}
    V_{x,\theta}:\ell^2(F)\otimes\C^2\otimes\C^2
    \longrightarrow Q_x\Hilbert[K]
  \end{equation*}
  by linear extension of
  \begin{equation*}
    V_{x,\theta}(\psi_h\otimes v\otimes c)
    \coloneqq\sigma_\theta(h,x)\psi_{hx}\otimes v\otimes c
    \qquad(h\in F,\ v,c\in\C^2).
  \end{equation*}
  Since $h\mapsto hx$ is a bijection from $F$ onto $Fx$ and $|\sigma_\theta(h,x)|=1$, this map sends an orthonormal basis onto an orthonormal basis and is therefore unitary.

  For $g\in S$ and $h\in F$, the cocycle identity implies
  \begin{equation*}
    \sigma_\theta(h,x)\sigma_\theta(g,hx)\overline{\sigma_\theta(gh,x)}
    =\sigma_\theta(g,h)\text.
  \end{equation*}
  Hence, for $v,c\in\C^2$,
  \begin{equation*}
    V_{x,\theta}^*Q_xM_{d,g}\lambda_\theta(g)Q_xV_{x,\theta}
      (\psi_h\otimes v\otimes c)
    =
    \begin{cases}
      \sigma_\theta(g,h)\psi_{gh}\otimes d(g,ghx)(v\otimes c),
        &gh\in F,\\
      0,&gh\notin F.
    \end{cases}
  \end{equation*}
  Since $h\mapsto gh$ is injective, the bound on $d$ and the triangle inequality give
  \begin{align*}
    \left|\norm{Q_xT_\theta Q_x}{}-\norm{Q_xT_{\theta'}Q_x}{}\right|
    &\leq\norm{V_{x,\theta}^*Q_xT_\theta Q_xV_{x,\theta}
      -V_{x,\theta'}^*Q_xT_{\theta'}Q_xV_{x,\theta'}}{}\\
    &\leq\sum_{g\in S}C(g)\max_{h\in F}|\sigma_\theta(g,h)-\sigma_{\theta'}(g,h)|\text.
  \end{align*}
  Only finitely many cocycle phases occur in this bound. By their continuity, it is less than $\varepsilon/2$ when $\theta'$ is sufficiently close to $\theta$, independently of $x$ and $d$. For such $\theta'$,
  \begin{align*}
    \norm{T_\theta}{}
    &<\sup_{x\in\Gamma}\norm{Q_xT_\theta Q_x}{}+\frac{\varepsilon}{2}\\
    &\leq\sup_{x\in\Gamma}\norm{Q_xT_{\theta'}Q_x}{}+\varepsilon\\
    &\leq\norm{T_{\theta'}}{}+\varepsilon\text.
  \end{align*}
  Interchanging $\theta$ and $\theta'$ yields
  \begin{equation*}
    \left|\norm{T_\theta}{}-\norm{T_{\theta'}}{}\right|<\varepsilon\text,
  \end{equation*}
  uniformly over all admissible $d$.
\end{proof}

For every nonempty finite $F\subseteq\Gamma$, for any $b \in \A_\theta\rtimes_{\alpha_{A,\theta}}\Z$, and for any $g \in \Gamma$, we define
\begin{equation*}
b_g\coloneqq\operatorname{tr}_\theta(b\lambda_\theta(g)^*)
\end{equation*}
 and
\begin{equation*}
  \Phi_{F,\theta}(b)\coloneqq\sum_{g\in FF^{-1}}\phi_F(g)b_g\lambda_\theta(g)\text.
\end{equation*}

\begin{lemma}\label{lem:full-smoothing}
  The map $\Phi_{F,\theta}$ is unital and completely positive, preserves $\operatorname{tr}_\theta$, has finite-dimensional range, and satisfies
  \begin{equation*}
    \Lip^{\nabla}_{\theta,G}(\Phi_{F,\theta}(b))\leq\Lip^{\nabla}_{\theta,G}(b)
  \end{equation*}
  for all $b\in\dom{\Lip^{\nabla}_{\theta,G}}$.
\end{lemma}

\begin{proof}
  For $x\in\Gamma$, let $Q_x$ project onto $\ell^2(Fx)\otimes\C^2\otimes\C^2$. The Cauchy--Schwarz estimate in the proof of Lemma~\ref{lem:diagonal-continuity} defines a bounded operator $\mathcal S_F(T)$ by
  \begin{equation*}
    \inner{\mathcal S_F(T)\xi}{\eta}{}
    \coloneqq\frac1{|F|}\sum_{x\in\Gamma}\inner{Q_xTQ_x\xi}{\eta}{}
    \qquad(T\in\B(\Hilbert[K]),\ \xi,\eta\in\Hilbert[K]).
  \end{equation*}
  In particular,
  \begin{equation*}
    \norm{\mathcal S_F(T)}{}\leq\sup_{x\in\Gamma}\norm{Q_xTQ_x}{}\leq\norm{T}{}\text.
  \end{equation*}
  Each finite partial sum is completely positive, so its weak operator limit is completely positive at every matrix level. Moreover, the identity
  \begin{equation*}
    \frac1{|F|}\sum_{x\in\Gamma}\norm{Q_x\xi}{}^2=\norm{\xi}{}^2
    \qquad(\xi\in\Hilbert[K])
  \end{equation*}
  implies $\mathcal S_F(1)=1$.

  For $h\in\Gamma$, let $P_h$ project onto $\C\psi_h\otimes\C^2\otimes\C^2$. Since $P_hQ_x=P_h$ when $h\in Fx$, and $P_hQ_x=0$ otherwise, we have
  \begin{align*}
    P_s\mathcal S_F(T)P_h
    &=\frac{|\{x:s,h\in Fx\}|}{|F|}P_sTP_h\\
    &=\psi_F(sh^{-1})P_sTP_h
    \qquad(s,h\in\Gamma).
  \end{align*}
  For $g,h\in\Gamma$ and $v,c\in\C^2$,
  \begin{equation*}
    \lambda_\theta(g)(\psi_h\otimes v\otimes c)
    =\sigma_\theta(g,h)\psi_{gh}\otimes v\otimes c\text.
  \end{equation*}
  Thus
  \begin{equation*}
    \mathcal S_F(\lambda_\theta(g))
    =\psi_F(g)\lambda_\theta(g)\text.
  \end{equation*}
  By linearity and norm continuity, the restriction of $\mathcal S_F$ to the crossed product equals $\Phi_{F,\theta}$. Consequently, $\Phi_{F,\theta}$ is unital and completely positive. Its defining formula also gives
  \begin{equation*}
    \operatorname{tr}_\theta(\Phi_{F,\theta}(b))
    =\psi_F(e)b_e
    =\operatorname{tr}_\theta(b)\text,
  \end{equation*}
  and
  \begin{equation*}
    \operatorname{ran}\Phi_{F,\theta}
    \subseteq\operatorname{span}\{\lambda_\theta(g):g\in FF^{-1}\}\text.
  \end{equation*}

  Fix $b\in\dom{\Lip^\nabla_{\theta,G}}$. For $k\in\Z$, set
  \begin{equation*}
    K_{G,k}\coloneqq1+M_{q_k}Y_G\Dirac[\nabla]_G^{-1}\text.
  \end{equation*}
   By construction, $K_{G,k}$ preserve each subspace
  $\C\psi_h\otimes\C^2\otimes\C^2$, so
  \begin{equation*}
    K_{G,k}P_h=P_hK_{G,k}\qquad(h\in\Gamma)\text.
  \end{equation*}
   By \eqref{psi-Fourier-compatibility} and \eqref{Theta-eq},
  \begin{align*}
    \mathcal F_k(\rho_{\theta,G}(b))
    &=\mathcal F_k(b)U^{-k}\rho_{\theta,G}(U^k)\\
    &=\mathcal F_k(b)K_{G,k}\text.
  \end{align*}
  If $s=(p,n)$ and $h=(q,m)$ belong to $\Gamma$, then
  \begin{equation*}
    P_s\mathcal F_k(T)P_h
    =
    \begin{cases}
      P_sTP_h,&k=n-m,\\
      0,&k\ne n-m,
    \end{cases}
    \qquad(T\in\B(\Hilbert[K])).
  \end{equation*}
  Taking $k=n-m$ therefore yields
  \begin{align*}
    P_s\rho_{\theta,G}(b)P_h
    &=P_s\mathcal F_k(b)K_{G,k}P_h\\
    &=P_sbP_hK_{G,k}\text.
  \end{align*}
  The element $\Phi_{F,\theta}(b)$ is a finite sum of group unitaries, so it belongs to $\dom{\rho_{\theta,G}}$. Applying the same identity to it gives
  \begin{align*}
    P_s\rho_{\theta,G}(\Phi_{F,\theta}(b))P_h
    &=P_s\Phi_{F,\theta}(b)P_hK_{G,k}\\
    &=\psi_F(sh^{-1})P_sbP_hK_{G,k}\\
    &=\psi_F(sh^{-1})P_s\rho_{\theta,G}(b)P_h\\
    &=P_s\mathcal S_F(\rho_{\theta,G}(b))P_h\text.
  \end{align*}
  Since this holds for all $s,h\in\Gamma$, we conclude that
  \begin{equation*}
    \rho_{\theta,G}(\Phi_{F,\theta}(b))
    =\mathcal S_F(\rho_{\theta,G}(b))\text.
  \end{equation*}
  By linearity and contractivity,
  \begin{align*}
    \norm{\Phi_{F,\theta}(b)-\rho_{\theta,G}(\Phi_{F,\theta}(b))}{}
    &=\norm{\mathcal S_F(b-\rho_{\theta,G}(b))}{}\\
    &\leq\norm{b-\rho_{\theta,G}(b)}{}\text.
  \end{align*}

  Write $D=\Dirac[\nabla]_G$. Each $P_h$ commutes with $D$ and has range contained in $\dom{D}$. Thus, for $s,h\in\Gamma$,
  \begin{align*}
    P_s[D,\Phi_{F,\theta}(b)]_{\rho_{\theta,G}}P_h
    &=DP_s\Phi_{F,\theta}(b)P_h
      -P_s\rho_{\theta,G}(\Phi_{F,\theta}(b))P_hD\\
    &=\psi_F(sh^{-1})
      \bigl(DP_sbP_h-P_s\rho_{\theta,G}(b)P_hD\bigr)\\
    &=\psi_F(sh^{-1})P_s[D,b]_{\rho_{\theta,G}}P_h\\
    &=P_s\mathcal S_F([D,b]_{\rho_{\theta,G}})P_h\text.
  \end{align*}
  Equality for all $s,h$ gives
  \begin{equation*}
    [D,\Phi_{F,\theta}(b)]_{\rho_{\theta,G}}
    =\mathcal S_F([D,b]_{\rho_{\theta,G}})\text,
  \end{equation*}
  so
  \begin{equation*}
    \norm{[D,\Phi_{F,\theta}(b)]_{\rho_{\theta,G}}}{}
    \leq\norm{[D,b]_{\rho_{\theta,G}}}{}\text.
  \end{equation*}
  The same estimates hold for $b^*$, and $\Phi_{F,\theta}(b^*)=\Phi_{F,\theta}(b)^*$. Taking the maximum of the terms defining $\Lip^\nabla_{\theta,G}$ proves
  \begin{equation*}
    \Lip^\nabla_{\theta,G}(\Phi_{F,\theta}(b))
    \leq\Lip^\nabla_{\theta,G}(b)\text.
  \end{equation*}
\end{proof}

\begin{proposition}\label{vm:polynomials}
  Let $S\subseteq\Gamma$ be finite and set $b_\theta(z)=\sum_{g\in S}z_g\lambda_\theta(g)$ for $z\in\C^S$. The functions
  \begin{equation*}
    (\theta,G,z)\longmapsto\norm{b_\theta(z)}{\A_\theta\rtimes_{\alpha_{A,\theta}}\Z},\qquad (\theta,G,z)\longmapsto\Lip^{\nabla}_{\theta,G}(b_\theta(z))
  \end{equation*}
  are jointly continuous; their continuity in $(\theta,G)$ is uniform on bounded coefficient sets.

  For every $\varepsilon>0$, there exists a nonempty finite $F\subseteq\Gamma$ such that
  \begin{equation*}
    \norm{b-\Phi_{F,\theta}(b)}{\A_\theta\rtimes_{\alpha_{A,\theta}}\Z}\leq\varepsilon\Lip^{\nabla}_{\theta,G}(b)
  \end{equation*}
  for all $(\theta,G)\in\R\times\mathscr{G}$ and $b\in\dom{\Lip^{\nabla}_{\theta,G}}$. Moreover, there exists $R_0>0$, independent of $(\theta,G)$, such that
  \begin{equation*}
    \norm{b}{\A_\theta\rtimes_{\alpha_{A,\theta}}\Z}\leq R_0\Lip^{\nabla}_{\theta,G}(b)\qquad\text{if }\operatorname{tr}_\theta(b)=0\text.
  \end{equation*}
\end{proposition}
\begin{proof}
  For $g=(m,k)\in\Gamma$, define
  \begin{equation*}
    d_{G,m,k}(p,n)\coloneqq
    w(n)\sigma(GA^{-n}m)\otimes\gamma_0+k1\otimes\gamma_1
    \qquad((p,n)\in\Gamma).
  \end{equation*}
  The twisted commutator formula in Theorem~\ref{twisted-spectral-triple-thm} gives, for $h\in\Gamma$ and $v,c\in\C^2$,
  \begin{equation*}
    [\Dirac[\nabla]_G,\lambda_\theta(g)]_{\rho_{\theta,G}}
      (\psi_h\otimes v\otimes c)
    =\sigma_\theta(g,h)\psi_{gh}\otimes
      d_{G,m,k}(gh)(v\otimes c)\text.
  \end{equation*}
  Consequently,
  \begin{equation*}
    [\Dirac[\nabla]_G,\lambda_\theta(m,k)]_{\rho_{\theta,G}}
    =M_{d_{G,m,k}}\lambda_\theta(m,k)\text.
  \end{equation*}

  For $G,G_0\in\mathscr G$, set
  \begin{equation*}
    \delta_G\coloneqq\norm{(G-G_0)G_0^{-1}}{M_2(\R)}\text.
  \end{equation*}
  By the choice of $w$,
  \begin{equation*}
    w(n)\norm{GA^{-n}m}{\R^2}
    \leq w(n)\norm{GA^{-n}G^{-1}}{M_2(\R)}\norm{Gm}{\R^2}
    \leq\norm{Gm}{\R^2}\text.
  \end{equation*}
  Applying this estimate with $G_0$ yields
  \begin{align*}
    \norm{d_{G,m,k}-d_{G_0,m,k}}{\ell^\infty(\Gamma,M_4(\C))}
    &=\sup_{n\in\Z}w(n)\norm{(G-G_0)A^{-n}m}{\R^2}\\
    &\leq\delta_G\sup_{n\in\Z}w(n)\norm{G_0A^{-n}m}{\R^2}\\
    &\leq\delta_G\norm{G_0m}{\R^2}\text.
  \end{align*}

  To treat the displacement term, put
  \begin{equation*}
    R_{G,k}\coloneqq M_{q_k}Y_G\Dirac[\nabla]_G^{-1}\text.
  \end{equation*}
  In the semidirect product representation, this operator acts by
  \begin{equation*}
    (R_{G,k}\xi)(p,n)
    =q_k(n)\bigl(w(n)\sigma(GA^{-n}p)\otimes\gamma_0\bigr)
      H_G(p,n)^{-1}\xi(p,n)\text.
  \end{equation*}
  Thus its coefficient matrices are independent of $\theta$, and
  \begin{equation*}
    \norm{R_{G,k}}{}\leq C_w(k)\text.
  \end{equation*}
  Since $\Theta_{G,k}=U^kR_{G,k}$, Lemma~\ref{vm:twist-norm} gives
  \begin{equation*}
    \norm{R_{G,k}-R_{G_0,k}}{}
    =\norm{\Theta_{G,k}-\Theta_{G_0,k}}{}
    \leq2C_w(k)\frac{\delta_G}{1-\delta_G}
    \qquad(\delta_G<1).
  \end{equation*}

  Let $P_\theta=\sum_{g=(m,k)\in S}c_g\lambda_\theta(g)$ be a polynomial with fixed finite support. By \eqref{Theta-eq},
  \begin{equation*}
    \rho_{\theta,G}(P_\theta)-P_\theta
    =\sum_{g=(m,k)\in S}c_g\lambda_\theta(g)R_{G,k}\text.
  \end{equation*}
  Taking adjoints gives
  \begin{equation*}
    \bigl(\rho_{\theta,G}(P_\theta)-P_\theta\bigr)^*
    =\sum_{g=(m,k)\in S}\overline{c_g}R_{G,k}^*\lambda_\theta(g^{-1})\text.
  \end{equation*}
  Lemma~\ref{lem:diagonal-continuity} therefore proves continuity in $\theta$ of the norms of $P_\theta$, its twisted commutator, and its displacement. For variation in $G$, the triangle inequality gives
  \begin{equation*}
    \norm{[\Dirac[\nabla]_G,P_\theta]_{\rho_{\theta,G}}
      -[\Dirac[\nabla]_{G_0},P_\theta]_{\rho_{\theta,G_0}}}{}
    \leq\delta_G\sum_{g=(m,k)\in S}|c_g|\norm{G_0m}{\R^2}\text,
  \end{equation*}
  and
  \begin{equation*}
    \norm{\rho_{\theta,G}(P_\theta)-\rho_{\theta,G_0}(P_\theta)}{}
    \leq\frac{2\delta_G}{1-\delta_G}
      \sum_{g=(m,k)\in S}|c_g|C_w(k)\text.
  \end{equation*}
  The same argument applies to $P_\theta^*$. Taking the maximum of the defining seminorm terms proves the asserted joint continuity. All bounds are uniform when the coefficients $c_g$ range over a bounded set.

  We next estimate the base Fourier coefficients. Write
  \begin{equation*}
    \mathcal F_k(b)=\pi_\theta(a_k(b))U^k\text.
  \end{equation*}
  In the original regular representation, let
  \begin{equation*}
    \iota_n:\Hilbert\otimes\C^2\longrightarrow\Hilbert[K],
    \qquad \iota_n\eta\coloneqq\phi_n\otimes\eta
    \qquad(n\in\Z).
  \end{equation*}
  By \eqref{psi-Fourier-compatibility} and the twisted commutator formula,
  \begin{align*}
    T_k\coloneqq
    \iota_0^*[\Dirac[\nabla]_G,b]_{\rho_{\theta,G}}\iota_{-k}
    &=\iota_0^*[\Dirac[\nabla]_G,\mathcal F_k(b)]_{\rho_{\theta,G}}\iota_{-k}\\
    &=[\Dirac_G,a_k(b)]\otimes\gamma_0
      +ka_k(b)\otimes\gamma_1\text.
  \end{align*}
  Since $\gamma_0\gamma_1=-\gamma_1\gamma_0$ and $\gamma_1^2=1$,
  \begin{equation*}
    [\Dirac_G,a_k(b)]\otimes\gamma_0
    =\frac12\bigl(T_k-(1\otimes\gamma_1)T_k(1\otimes\gamma_1)\bigr)\text.
  \end{equation*}
  Hence
  \begin{equation*}
    \Lip_{\theta,G}(a_k(b))
    \leq\norm{T_k}{}
    \leq\norm{[\Dirac[\nabla]_G,b]_{\rho_{\theta,G}}}{}
    \leq\Lip^\nabla_{\theta,G}(b)\text.
  \end{equation*}

  Set
  \begin{equation*}
    K\coloneqq\sup_{G\in\mathscr G}\norm{G^{-1}}{M_2(\R)}\text.
  \end{equation*}
  The Pauli-matrix computation in Lemma~\ref{vm:relative}, applied to $G^{-1}$, gives
  \begin{equation*}
    \Lip_{\theta,I_2}(a)
    \leq K\Lip_{\theta,G}(a)
    \qquad(a\in\A_{\theta,G}).
  \end{equation*}
  For $y=(y_1,y_2)\in\R^2$, write $e^{iy}=(e^{iy_1},e^{iy_2})$. If $a$ is a base Fourier polynomial, then
  \begin{equation*}
    \frac{d}{dt}\zeta^{e^{ity}}(a)
    =\frac{i}{2}\zeta^{e^{ity}}\bigl(
      (1\otimes\sigma(y))[\Dirac_{I_2},a]
      +[\Dirac_{I_2},a](1\otimes\sigma(y))\bigr)\text.
  \end{equation*}
  Integrating from $0$ to $1$ yields
  \begin{equation*}
    \norm{a-\zeta^{e^{iy}}(a)}{}
    \leq\norm{y}{\R^2}\Lip_{\theta,I_2}(a)\text.
  \end{equation*}
  Applying this inequality to $\zeta^{ff_J}(a)$ and using \eqref{vm:torus-fejer} extends it to every $a\in\A_{\theta,G}$.

  For the outer dual action $\beta$, the estimate proved in the first section gives
  \begin{equation*}
    \norm{b-\beta^{e^{ix}}(b)}{}
    \leq|x|\Lip^\nabla_{\theta,G}(b)\text.
  \end{equation*}
  Averaging this inequality over $[-\pi,\pi]$ gives
  \begin{align*}
    \norm{b-\mathcal F_0(b)}{}
    &\leq\int_{-\pi}^{\pi}\norm{b-\beta^{e^{ix}}(b)}{}\frac{dx}{2\pi}\\
    &\leq\pi\Lip^\nabla_{\theta,G}(b)\text.
  \end{align*}
  Averaging the base estimate over $[-\pi,\pi]^2$ gives
  \begin{align*}
    \norm{\mathcal F_0(b)-\operatorname{tr}_\theta(b)1}{}
    &\leq\int_{[-\pi,\pi]^2}
      \norm{a_0(b)-\zeta^{e^{iy}}(a_0(b))}{}
      \frac{dy_1dy_2}{(2\pi)^2}\\
    &\leq\pi\sqrt2 K\Lip_{\theta,G}(a_0(b))\\
    &\leq\pi\sqrt2 K\Lip^\nabla_{\theta,G}(b)\text.
  \end{align*}
  Consequently,
  \begin{equation*}
    \norm{b-\operatorname{tr}_\theta(b)1}{}
    \leq(\pi+\pi\sqrt2 K)\Lip^\nabla_{\theta,G}(b)\text,
  \end{equation*}
  so we may take $R_0=\pi+\pi\sqrt2 K$.

  To obtain a uniform finite approximation, put
  \begin{equation*}
    \varepsilon_N^{(1)}
    \coloneqq\int_{-\pi}^{\pi}|x|f_N(e^{ix})\frac{dx}{2\pi}\text,
  \end{equation*}
  and
  \begin{equation*}
    \varepsilon_J^{(2)}
    \coloneqq\int_{[-\pi,\pi]^2}\norm{y}{\R^2}ff_J(e^{iy})
      \frac{dy_1dy_2}{(2\pi)^2}\text.
  \end{equation*}
  Both sequences converge to zero because the Fej\'er kernels form approximate identities. The preceding estimates imply
  \begin{equation*}
    \norm{b-\beta^{f_N}(b)}{}
    \leq\varepsilon_N^{(1)}\Lip^\nabla_{\theta,G}(b)\text,
  \end{equation*}
  and
  \begin{equation*}
    \norm{a-\zeta^{ff_J}(a)}{}
    \leq K\varepsilon_J^{(2)}\Lip_{\theta,G}(a)\text.
  \end{equation*}
  Define
  \begin{equation*}
    p_{N,J}(b)\coloneqq
    \sum_{|k|\leq N}\left(1-\frac{|k|}{N+1}\right)
      \pi_\theta\bigl(\zeta^{ff_J}(a_k(b))\bigr)U^k\text.
  \end{equation*}
  Its group Fourier support is contained in
  \begin{equation*}
    S_{N,J}\coloneqq
    \{(m,k)\in\Gamma:\norm{m}{\ell^\infty}\leq J,\ |k|\leq N\}\text.
  \end{equation*}
  Comparing it with $\beta^{f_N}(b)$ gives
  \begin{align*}
    \norm{b-p_{N,J}(b)}{}
    &\leq\norm{b-\beta^{f_N}(b)}{}
      +\sum_{|k|\leq N}\norm{a_k(b)-\zeta^{ff_J}(a_k(b))}{}\\
    &\leq\bigl(\varepsilon_N^{(1)}
      +(2N+1)K\varepsilon_J^{(2)}\bigr)\Lip^\nabla_{\theta,G}(b)\text.
  \end{align*}

  Suppose $\operatorname{tr}_\theta(b)=0$. Since
  \begin{equation*}
    \norm{a_k(b)}{}=\norm{\mathcal F_k(b)}{}
    \leq\norm{b}{}
    \leq R_0\Lip^\nabla_{\theta,G}(b)\text,
  \end{equation*}
  the Fourier coefficient formula in \eqref{vm:torus-fejer} shows that every scalar coefficient of $p_{N,J}(b)$ has modulus at most $R_0\Lip^\nabla_{\theta,G}(b)$.

  Given $\varepsilon>0$, choose $N$, then $J$, such that
  \begin{equation*}
    \varepsilon_N^{(1)}+(2N+1)K\varepsilon_J^{(2)}
    <\frac{\varepsilon}{4}\text.
  \end{equation*}
  By Lemma~\ref{lem:folner-explicit}, choose a nonempty finite $F\subseteq\Gamma$ such that
  \begin{equation*}
    R_0\sum_{g\in S_{N,J}}(1-\psi_F(g))
    <\frac{\varepsilon}{2}\text.
  \end{equation*}
  The coefficient bound therefore gives
  \begin{equation*}
    \norm{p_{N,J}(b)-\Phi_{F,\theta}(p_{N,J}(b))}{}
    \leq R_0\Lip^\nabla_{\theta,G}(b)
      \sum_{g\in S_{N,J}}(1-\psi_F(g))
    \leq\frac{\varepsilon}{2}\Lip^\nabla_{\theta,G}(b)\text.
  \end{equation*}
  Since $\Phi_{F,\theta}$ is contractive,
  \begin{align*}
    \norm{b-\Phi_{F,\theta}(b)}{}
    &\leq\norm{b-p_{N,J}(b)}{}
      +\norm{\Phi_{F,\theta}(b-p_{N,J}(b))}{}
      +\norm{p_{N,J}(b)-\Phi_{F,\theta}(p_{N,J}(b))}{}\\
    &\leq2\norm{b-p_{N,J}(b)}{}
      +\frac{\varepsilon}{2}\Lip^\nabla_{\theta,G}(b)\\
    &\leq\varepsilon\Lip^\nabla_{\theta,G}(b)\text.
  \end{align*}
  For arbitrary $b$, apply this estimate to $b-\operatorname{tr}_\theta(b)1$. Since $\Phi_{F,\theta}(1)=1$ and subtracting a scalar does not change the seminorm, the same estimate holds for $b$.
\end{proof}

\begin{lemma}\label{vm:twist-columns}
  Fix $t_0=(\theta_0,G_0)$. Let $Q_J$ project onto the outer coordinates $|n|\leq J$, where $J\geq1$, and put $c_J=\min_{|n|\leq J}w(n)$. For every $b\in\dom{\rho_{\theta,G}}$,
  \begin{equation*}
    (\rho_{\theta,G}(b)-b)Q_J=[M_w,b](\Dirac_G^{\mathrm{diag}}\otimes\gamma_0)\Dirac[\nabla]_G^{-1}Q_J,
  \end{equation*}
  and
  \begin{equation*}
    \norm{(\Dirac_G^{\mathrm{diag}}\otimes\gamma_0)\Dirac[\nabla]_G^{-1}Q_J}{}\leq c_J^{-1},\qquad \norm{\rho_{\theta,G}(b)Q_J}{}\leq(1+2c_J^{-1})\norm{b}{}\text.
  \end{equation*}
  If $b,c\in\dom{\Lip^{\nabla}_{\theta,G}}$ satisfy $\Lip^{\nabla}_{\theta,G}(b),\Lip^{\nabla}_{\theta,G}(c)\leq1$ and $\norm{b}{},\norm{c}{}\leq R$, then
  \begin{equation*}
    \opnorm{\rho_{\theta,G}(b)-\rho_{\theta,G}(c)}{}{t_0}\leq(1+2c_J^{-1})\norm{b-c}{}+\frac{2(R+\kappa^{-1})}{J}\text.
  \end{equation*}
\end{lemma}

\begin{proof}
  Let $P_n$ project onto the $n$th outer coordinate, and put $T_{G,J}=(\Dirac_G^{\mathrm{diag}}\otimes\gamma_0)\Dirac[\nabla]_G^{-1}Q_J$. Since all the diagonal factors commute with $Q_J$,
  \begin{equation*}
    T_{G,J}=M_w^{-1}Q_JY_G\Dirac[\nabla]_G^{-1},\qquad \norm{T_{G,J}}{}\leq c_J^{-1}\text.
  \end{equation*}
  Here $M_w^{-1}Q_J$ is bounded. For a Fourier monomial $b=\pi_\theta(a)U^k$,
  \begin{equation*}
    [M_w,b]T_{G,J}=\pi_\theta(a)[M_w,U^k]T_{G,J}=\pi_\theta(a)\Theta_{G,k}Q_J=(\rho_{\theta,G}(b)-b)Q_J\text.
  \end{equation*}
  For general $b\in\dom{\rho_{\theta,G}}$, its outer Fourier coefficients satisfy
  \begin{equation*}
    P_n(\rho_{\theta,G}(b)-b)Q_JP_m=P_n\pi_\theta(a_{n-m}(b))\Theta_{G,n-m}Q_JP_m=P_n[M_w,b]T_{G,J}P_m\text.
  \end{equation*}
  Equality of the blocks proves the first identity. Since $\norm{M_w}{}\leq1$,
  \begin{equation*}
    \norm{\rho_{\theta,G}(b)Q_J}{}\leq\norm{b}{}+\norm{[M_w,b]}{}\norm{T_{G,J}}{}\leq(1+2c_J^{-1})\norm{b}{}\text.
  \end{equation*}
  The displacement bound gives $\norm{\rho_{\theta,G}(b)}{},\norm{\rho_{\theta,G}(c)}{}\leq R+\kappa^{-1}$. If $\CDN_{G_0}(\xi)\leq1$, then
  \begin{equation*}
    \norm{Q_J\xi}{}\leq1,\qquad \norm{(1-Q_J)\xi}{}\leq J^{-1}\norm{(\partial\otimes1)\xi}{}\leq J^{-1}\text.
  \end{equation*}
  Therefore,
  \begin{equation*}
    \norm{(\rho_{\theta,G}(b)-\rho_{\theta,G}(c))\xi}{}\leq(1+2c_J^{-1})\norm{b-c}{}+2(R+\kappa^{-1})J^{-1}\text.
  \end{equation*}
\end{proof}

\subsubsection{Finite-rank pivots}

\begin{lemma}\label{vm:states}
  Fix $t_0=(\theta_0,G_0)$. For every $\varepsilon>0$, there exist finite-support unit vectors $\xi_1,\ldots,\xi_l\in\dom{\nabla}$ and a neighborhood $V$ of $t_0$ such that every state of $\A_\theta\rtimes_{\alpha_{A,\theta}}\Z$, for $(\theta,G)\in V$, is within $\varepsilon$ for $\Kantorovich{\Lip^{\nabla}_{\theta,G}}$ of a convex combination of their vector states.
\end{lemma}
\begin{proof}
  We first reduce the approximation of states to a finite-dimensional space. Fix $\varepsilon>0$. By Proposition~\ref{vm:polynomials}, choose a nonempty finite $F\subseteq\Gamma$ such that
  \begin{equation*}
    \norm{b-\Phi_{F,\theta}(b)}{}
    \leq\frac{\varepsilon}{8}\Lip^\nabla_{\theta,G}(b)\text.
  \end{equation*}
  For any two states $\varphi,\omega$ and any $b$ with $\Lip^\nabla_{\theta,G}(b)\leq1$, we therefore have
  \begin{align*}
    |\varphi(b)-\omega(b)|
    &\leq2\norm{b-\Phi_{F,\theta}(b)}{}
      +|\varphi(\Phi_{F,\theta}(b))-\omega(\Phi_{F,\theta}(b))|\\
    &\leq\frac{\varepsilon}{4}
      +|\varphi(\Phi_{F,\theta}(b))-\omega(\Phi_{F,\theta}(b))|\text.
  \end{align*}

  We express the self-adjoint elements in the range of $\Phi_{F,\theta}$ using real coefficients. Choose $g_1,\ldots,g_s$ containing one element from each inverse pair in $FF^{-1}\setminus\{e\}$, and set $d=1+2s$. For $z=(z_0,x_1,y_1,\ldots,x_s,y_s)\in\R^d$, define
  \begin{equation*}
    a_\theta(z)\coloneqq z_0 1+
    \sum_{j=1}^s\bigl(
      x_j(\lambda_\theta(g_j)+\lambda_\theta(g_j)^*)
      +iy_j(\lambda_\theta(g_j)-\lambda_\theta(g_j)^*)\bigr)\text.
  \end{equation*}
  Since $\lambda_\theta(g)^*=\lambda_\theta(g^{-1})$, these polynomials form the self-adjoint part of the range of $\Phi_{F,\theta}$. Let $R_0$ be the constant from Proposition~\ref{vm:polynomials}, and put
  \begin{equation*}
    C_F\coloneqq\max\{1,\sqrt d\,R_0\}\text.
  \end{equation*}
  If $b=b^*$, $\operatorname{tr}_\theta(b)=0$, and $\Lip^\nabla_{\theta,G}(b)\leq1$, then
  \begin{equation*}
    |\psi_F(g)\operatorname{tr}_\theta(b\lambda_\theta(g)^*)|
    \leq\norm{b}{}\leq R_0
    \qquad(g\in\Gamma).
  \end{equation*}
  Consequently, $\Phi_{F,\theta}(b)=a_\theta(z)$ for some $z$ with $\norm{z}{\R^d}\leq C_F$. Every state $\varphi$ determines a unique vector $x\in\R^d$ such that
  \begin{equation*}
    \varphi(a_\theta(z))=z\cdot x
    \qquad(z\in\R^d).
  \end{equation*}
  Likewise, every unit vector $\xi\in\Hilbert[K]$ determines a vector $y_\xi\in\R^d$ by
  \begin{equation*}
    \inner{a_\theta(z)\xi}{\xi}{}=z\cdot y_\xi
    \qquad(z\in\R^d).
  \end{equation*}

  We next choose finitely many finite-support vector states that nearly attain the largest state value in every coefficient direction. For $\norm{z}{\R^d}=1$, set
  \begin{equation*}
    m_\theta(z)\coloneqq
    \sup_{\varphi\in\StateSpace(\A_\theta\rtimes_{\alpha_{A,\theta}}\Z)}
      \varphi(a_\theta(z))\text.
  \end{equation*}
  Since $\norm{a_\theta(z)}{}\leq2\sqrt d$, taking $R=1+2\sqrt d$ gives
  \begin{equation*}
    m_\theta(z)=\norm{R1+a_\theta(z)}{}-R\text.
  \end{equation*}
  Proposition~\ref{vm:polynomials} implies that $(\theta,z)\mapsto m_\theta(z)$ is continuous on $\R$ times the unit sphere of $\R^d$. Faithfulness of the representation and density of finite-support vectors give
  \begin{equation*}
    m_{\theta_0}(z)
    =\sup_{\substack{\xi\in\Hilbert[K],\ \norm{\xi}{}=1\\
      \supp{\xi}\text{ finite}}}
      \inner{a_{\theta_0}(z)\xi}{\xi}{}\text.
  \end{equation*}
  For each unit $z$, choose such a vector $\xi_z$ with
  \begin{equation*}
    m_{\theta_0}(z)-\inner{a_{\theta_0}(z)\xi_z}{\xi_z}{}
    <\frac{\varepsilon}{8C_F}\text.
  \end{equation*}
  For fixed $\xi_z$, the value $\inner{a_\theta(z')\xi_z}{\xi_z}{}$ depends on finitely many continuous cocycle phases and is linear in $z'$. Thus there are a neighborhood $I_z$ of $\theta_0$ and a neighborhood $O_z$ of $z$ in the unit sphere such that
  \begin{equation*}
    m_\theta(z')-\inner{a_\theta(z')\xi_z}{\xi_z}{}
    <\frac{\varepsilon}{4C_F}
    \qquad(\theta\in I_z,\ z'\in O_z).
  \end{equation*}
  Choose a finite subcover $O_{z_1},\ldots,O_{z_l}$, put $\xi_j=\xi_{z_j}$, and set $I=\bigcap_{j=1}^l I_{z_j}$. Then
  \begin{equation*}
    m_\theta(z)\leq
    \max_{1\leq j\leq l}\inner{a_\theta(z)\xi_j}{\xi_j}{}
      +\frac{\varepsilon}{4C_F}
    \qquad(\theta\in I,\ \norm{z}{\R^d}=1).
  \end{equation*}

  We now prove that convex combinations of these finitely many vector states approximate every state on the polynomials $a_\theta(z)$. Fix $\theta\in I$ and a state $\varphi$, and let $x,y_1,\ldots,y_l\in\R^d$ represent $\varphi$ and the vector states of $\xi_1,\ldots,\xi_l$, respectively. The preceding estimate gives
  \begin{equation*}
    z\cdot x\leq m_\theta(z)
    \leq\max_{1\leq j\leq l}z\cdot y_j+\frac{\varepsilon}{4C_F}
    \qquad(\norm{z}{\R^d}=1).
  \end{equation*}
  Let $y$ minimize the distance to $x$ in the compact convex set
  $C=\operatorname{conv}\{y_1,\ldots,y_l\}$. For every $y'\in C$, minimality along the segment from $y$ to $y'$ gives
  \begin{equation*}
    0\leq
    \left.\frac{d}{dt}\norm{x-y-t(y'-y)}{\R^d}^2\right|_{t=0^+}
    =-2(x-y)\cdot(y'-y)\text.
  \end{equation*}
  If $x\ne y$, take $z=(x-y)/\norm{x-y}{\R^d}$. Then
  \begin{equation*}
    z\cdot y=\max_{y'\in C}z\cdot y'
    =\max_{1\leq j\leq l}z\cdot y_j\text,
  \end{equation*}
  so
  \begin{equation*}
    \norm{x-y}{\R^d}=z\cdot(x-y)\leq\frac{\varepsilon}{4C_F}\text.
  \end{equation*}
  This bound also holds when $x=y$. Write $y=\sum_{j=1}^l t_jy_j$, where $t_j\geq0$ and $\sum_{j=1}^l t_j=1$, and define
  \begin{equation*}
    \omega_\theta(b)\coloneqq\sum_{j=1}^l t_j\inner{b\xi_j}{\xi_j}{}\text.
  \end{equation*}
  For every $z\in\R^d$, it follows that
  \begin{equation*}
    |\varphi(a_\theta(z))-\omega_\theta(a_\theta(z))|
    =|z\cdot(x-y)|
    \leq\frac{\varepsilon}{4C_F}\norm{z}{\R^d}\text.
  \end{equation*}

  We finally extend this approximation to the Lipschitz unit ball. If $b=b^*$, $\operatorname{tr}_\theta(b)=0$, and $\Lip^\nabla_{\theta,G}(b)\leq1$, then $\Phi_{F,\theta}(b)=a_\theta(z)$ with $\norm{z}{\R^d}\leq C_F$. Hence
  \begin{align*}
    |\varphi(b)-\omega_\theta(b)|
    &\leq2\norm{b-\Phi_{F,\theta}(b)}{}
      +|\varphi(a_\theta(z))-\omega_\theta(a_\theta(z))|\\
    &\leq\frac{\varepsilon}{4}+\frac{\varepsilon}{4}
    =\frac{\varepsilon}{2}\text.
  \end{align*}
  Subtracting $\operatorname{tr}_\theta(b)1$ does not change the seminorm or the difference of the two states. Therefore,
  \begin{equation*}
    \Kantorovich{\Lip^\nabla_{\theta,G}}(\varphi,\omega_\theta)
    \leq\frac{\varepsilon}{2}<\varepsilon\text.
  \end{equation*}
  Each $\xi_j$ has finite support and thus belongs to $\dom{\nabla}$. The neighborhood $V=I\times\mathscr G$ has the required property.
\end{proof}
For $P\in\N$ and $R\in\N\setminus\{0\}$, define
\begin{equation*}
  h_{P,R}(g)\coloneqq\max\left\{0,\min\left\{1,1-\frac{|g|-P}{R}\right\}\right\}\text,
\end{equation*}
and let $\omega_{P,R}$ be multiplication by $h_{P,R}$ on $\Hilbert[K]$.

\begin{lemma}\label{plateau-estimate-lemma}
  The operator $\omega_{P,R}$ is a positive finite-rank contraction of norm one, equals the identity on $\ell^2(\{g:|g|\leq P\})\otimes\C^2\otimes\C^2$, and satisfies
  \begin{equation*}
    \norm{[\lambda_\theta(s),\omega_{P,R}]}{\B(\Hilbert[K])}\leq\frac{|s|}{R}\qquad(s\in\Gamma)\text.
  \end{equation*}
  The same estimate holds for every normalized unitary $2$-cocycle on $\Gamma$.
\end{lemma}

\begin{proof}
  We have $0\leq h_{P,R}\leq1$, $h_{P,R}=1$ on $\{|g|\leq P\}$, and $\supp{h_{P,R}}\subseteq\{|g|<P+R\}$, a finite set. Moreover,
  \begin{equation*}
    |h_{P,R}(sg)-h_{P,R}(g)|\leq\frac{\bigl||sg|-|g|\bigr|}{R}\leq\frac{|s|}{R}\text.
  \end{equation*}
  For $v\in\C^2\otimes\C^2$,
  \begin{equation*}
    [\lambda_\theta(s),\omega_{P,R}](\phi_g\otimes v)=\sigma_\theta(s,g)(h_{P,R}(g)-h_{P,R}(sg))\phi_{sg}\otimes v\text.
  \end{equation*}
  Since $|\sigma_\theta(s,g)|=1$, squaring and summing over $g$ proves the estimate for any unitary cocycle.
\end{proof}

\begin{remark}
  The parameters $P$ and $R$ independently control the one-eigenspace and the commutator bound. A projection onto a finite set has commutator norm one with any shift that does not preserve that set. The preceding proof applies to every finitely generated discrete group and every normalized unitary $2$-cocycle.
\end{remark}

\subsubsection{Continuity}

\begin{theorem}\label{quantum-torus-Anosov-cor}
  For every fixed $A\in\operatorname{SL}(2,\Z)$, the family
  \begin{equation*}
    \left(\A_\theta\rtimes_{\alpha_{A,\theta}}\Z,\Hilbert[K],\Dirac[\nabla]_G,\rho_{\theta,G}\right)_{(\theta,G)\in\R\times\mathscr{G}}
  \end{equation*}
  is jointly continuous for the twisted spectral propinquity. Its restriction to any compact parameter set has compact range. In particular, this applies when $A$ is hyperbolic.
\end{theorem}

\begin{proof}
  Fix $t_0=(\theta_0,G_0)$ and write $t=(\theta,G)$, $\Lip_t=\Lip^{\nabla}_{\theta,G}$, and $\rho_t=\rho_{\theta,G}$. Put
  \begin{equation*}
    \mathfrak L_t\coloneqq\{b\in\domsa{\Lip_t}:\operatorname{tr}_\theta(b)=0,\ \Lip_t(b)\leq1\}\text.
  \end{equation*}
  Lemma~\ref{vm:relative} supplies the common adapted weight, and the base triples act on $\Hilbert$, with common domain $\dom{\Dirac}$. Taking the canonical traces as the states gives Hypothesis~\ref{crossed-product-family-hyp}.

  For nonempty finite $F\subseteq\Gamma$, put $S_F=\{k:(m,k)\in FF^{-1}\text{ for some }m\}$ and define
  \begin{equation*}
    K_{F,k,\theta}(a)\coloneqq\sum_{\substack{m\in\Z^2\\(m,k)\in FF^{-1}}}\phi_F(m,k)\operatorname{tr}_\theta(a(W_\theta^m)^*)W_\theta^m\qquad(k\in S_F)\text.
  \end{equation*}
  These bounded finite-rank maps satisfy
  \begin{equation*}
    \Phi_{F,\theta}(b)=\sum_{k\in S_F}\pi_\theta(K_{F,k,\theta}(a_k(b)))U^k\text.
  \end{equation*}
  By Lemma~\ref{lem:full-smoothing}, $\Phi_{F,\theta}$ is unital, preserves adjoints and the trace, and $\Lip_t\circ\Phi_{F,\theta}\leq\Lip_t$.

  For $N\in\N$, put $J_N=N+1$ and $c_N=\min_{|n|\leq J_N}w(n)$. Proposition~\ref{vm:polynomials} provides a nonempty finite $F_N\subseteq\Gamma$ such that
  \begin{equation*}
    \norm{b-\Phi_{F_N,\theta}(b)}{}\leq\frac{\Lip_t(b)}{(N+1)(1+2c_N^{-1})}\text.
  \end{equation*}
  Set $\Phi_{N,t}=\Phi_{F_N,\theta}$ and $\varepsilon_N=(N+1)^{-1}$. These satisfy the approximation and seminorm estimates in Hypothesis~\ref{crossed-product-convergence-hyp}.

  Fix $N$ and abbreviate $F=F_N$. Choose one element from each inverse pair in $FF^{-1}\setminus\{e\}$, forming $E_F^+$. For $z=((x_g,y_g))_{g\in E_F^+}$, set
  \begin{equation*}
    \begin{aligned}
      &P_{N,t}(z)\coloneqq\sum_{g\in E_F^+}\left(x_g(\lambda_\theta(g)+\lambda_\theta(g)^*)+iy_g(\lambda_\theta(g)-\lambda_\theta(g)^*)\right),\\
      &C_N\coloneqq[-R_0,R_0]^{2|E_F^+|}\text.
    \end{aligned}
  \end{equation*}
  Every $P_{N,t}(z)$ is self-adjoint and trace centered. Proposition~\ref{vm:polynomials} and contractivity of $\Phi_{N,t}$ imply
  \begin{equation*}
    \Phi_{N,t}(\mathfrak L_t)\subseteq P_{N,t}(C_N),\qquad \sup_{z\in C_N}|\Lip_t(P_{N,t}(z))-\Lip_{t_0}(P_{N,t_0}(z))|\longrightarrow0\text.
  \end{equation*}
  For every finite-support $\xi$ and $g\in\Gamma$,
  \begin{equation*}
    \norm{(\lambda_\theta(g)-\lambda_{\theta_0}(g))\xi}{\Hilbert[K]}^2=\sum_{h\in\supp{\xi}}|\sigma_\theta(g,h)-\sigma_{\theta_0}(g,h)|^2\norm{\xi_h}{\C^2\otimes\C^2}^2\longrightarrow0\text.
  \end{equation*}
  The unitaries have norm one, so density gives
  \begin{equation*}
    \sup_{z\in C_N}\norm{(P_{N,t}(z)-P_{N,t_0}(z))\xi}{\Hilbert[K]}\longrightarrow0\qquad(\xi\in\Hilbert[K])\text.
  \end{equation*}

  Lemma~\ref{plateau-estimate-lemma} gives
  \begin{equation*}
    \sup_{z\in C_N}\norm{[P_{N,t_0}(z),\omega_{P,R}]}{}\leq\frac{R_N}{R},\qquad R_N\coloneqq\sup_{z\in C_N}2\sum_{g\in E_F^+}(|x_g|+|y_g|)|g|<\infty\text.
  \end{equation*}
  For each $j\in\N$, apply Lemma~\ref{vm:states} with tolerance $(j+1)^{-1}$ and choose $P_j$ containing the supports of its vectors. Choose $R_j\geq1$ with $R_N/R_j<(j+1)^{-1}$ and put $\Omega_{N,j}=\omega_{P_j,R_j}$. The selected vectors lie in the one-eigenspace of $\Omega_{N,j}$, so their convex combinations define one-level states. Consequently,
  \begin{equation*}
    \begin{aligned}
      &\sup_{z\in C_N}\norm{[P_{N,t_0}(z),\Omega_{N,j}]}{}<\frac1{j+1},\\
      &\limsup_{t\to t_0}\bridgeheight{(\B(\Hilbert[K]),\Omega_{N,j})|\Lip_t,\Lip_{t_0}}{}\leq\frac1{j+1}\text.
    \end{aligned}
  \end{equation*}
  Each pivot has a nonzero one-eigenspace and hence defines a bridge for every parameter. This proves the pivot conditions.

  By Lemma~\ref{vm:relative},
  \begin{equation*}
    \norm{(\Dirac_G-\Dirac_{G_0})\xi}{\Hilbert}\leq\delta_G\norm{\Dirac_{G_0}\xi}{\Hilbert},\qquad \delta_G=\norm{(G-G_0)G_0^{-1}}{M_2(\R)}\longrightarrow0\text.
  \end{equation*}
  For $b\in\mathfrak L_t$, both $b$ and $\Phi_{N,t}(b)$ have norm at most $R_0$ and seminorm at most one. Lemma~\ref{vm:twist-columns} therefore yields
  \begin{equation*}
    \begin{aligned}
      \opnorm{\rho_t(b)-\rho_t(\Phi_{N,t}(b))}{}{t_0}&\leq(1+2c_N^{-1})\norm{b-\Phi_{N,t}(b)}{}+\frac{2(R_0+\kappa^{-1})}{J_N}\\
      &\leq\frac{1+2(R_0+\kappa^{-1})}{N+1}\text.
    \end{aligned}
  \end{equation*}
  This proves the uniform twist-approximation condition.

  For fixed $N$, let $X_t(z)$ denote either
  \begin{equation*}
    P_{N,t}(z)-P_{N,t_0}(z)\qquad\text{or}\qquad \rho_t(P_{N,t}(z))-\rho_{t_0}(P_{N,t_0}(z))\text.
  \end{equation*}
  The supports and coefficient bounds are fixed. Since
  \begin{equation*}
    \rho_t(\lambda_\theta(m,k))=\lambda_\theta(m,k)(1+M_{q_k}Y_G\Dirac[\nabla]_G^{-1}),
  \end{equation*}
  Lemma~\ref{vm:twist-norm} and the cocycle formula imply, on a neighborhood of $t_0$,
  \begin{equation*}
    \sup_{z\in C_N}\norm{X_t(z)}{}\leq M_N,\qquad \sup_{z\in C_N}\norm{X_t(z)Q}{}\longrightarrow0
  \end{equation*}
  for every finite-rank coordinate projection $Q$. Indeed, the diagonal multiplier preserves the range of $Q$, its variation in $G$ tends to zero in norm, and the shifts involve only finitely many cocycle phases on that range.

  Let $\varepsilon>0$. Compactness of the $\CDN_{t_0}$-unit ball provides a finite-rank coordinate projection $Q$ such that
  \begin{equation*}
    \sup_{\CDN_{t_0}(\xi)\leq1}\norm{(1-Q)\xi}{\Hilbert[K]}<\frac{\varepsilon}{2(M_N+1)}\text.
  \end{equation*}
  For $t$ sufficiently close to $t_0$, $\sup_{z\in C_N}\norm{X_t(z)Q}{}<\varepsilon/2$, whence
  \begin{equation*}
    \sup_{z\in C_N}\opnorm{X_t(z)}{}{t_0}\leq\sup_{z\in C_N}\norm{X_t(z)Q}{}+M_N\sup_{\CDN_{t_0}(\xi)\leq1}\norm{(1-Q)\xi}{}<\varepsilon\text.
  \end{equation*}
  Thus all conditions of Hypothesis~\ref{crossed-product-convergence-hyp} hold. Theorem~\ref{crossed-product-finite-mode-convergence-thm} proves continuity, and continuity preserves compactness of the parameter sets.
\end{proof}

  \bibliographystyle{amsplain}
  \bibliography{../thesis.bib}

  \vfill

\end{document}